%% file: main-arxiv.tex
\documentclass[journal]{IEEEtran}

\input{latex/packages}
\input{latex/tikz_utils}
\input{latex/algorithm}

\input{latex/theorems}
\input{latex/commands}
\input{latex/glossary}

\pgfplotsset{compat=1.18}

\begin{document}
\title{Compositional Sensitivity Analysis for Multiobjective System Co-Design}

\author{
Alan~Bonomi$^{1,2}$,
Antonio~Terpin$^{2}$,
Yubo~Cai$^{1}$,
Yujun~Huang$^{1}$,
Meshal~Alharbi$^{1}$,
Gioele~Zardini$^{1}$%
\thanks{$^{1}$Laboratory for Information \& Decision Systems (LIDS), Massachusetts Institute of Technology, Cambridge, MA, USA. Emails: {\tt \{abonomi,yubocai,yujun233,meshal,gzardini\}@mit.edu}}%
\thanks{$^{2}$Institute for Dynamic Systems and Control (IDSC), ETH Zurich, Zurich, Switzerland. Email: {\tt aterpin@ethz.ch}.}
\thanks{This material is based on research supported in part by the Defense Advanced Research Projects Agency (DARPA) under Award D25AC00373, and in part by the Air Force Office of Scientific Research (AFOSR) under Grant Agreement FA9550261B038 (the Center of Excellence on Compositional Optimization, Dynamical Systems and Control).}
}

\maketitle

\begin{abstract}
\input{sections/abstract}
\end{abstract}

\begin{IEEEkeywords}
Co-design, sensitivity analysis, multiobjective optimization, set-valued analysis, variational analysis, compositional methods, system design.
\end{IEEEkeywords}

\input{sections/introduction}

\input{sections/problem_and_main_results}

\input{sections/preliminaries}

\input{sections/tangential_design_problems_and_queries}

\input{sections/compositional_sensitivity_in_codesign}

\input{sections/local_computation_of_system_level_sensitivity}

\input{sections/finite_variations}

\input{sections/numerical_example}

\input{sections/conclusion}

\bibliographystyle{IEEEtran}
\bibliography{bib}

\appendices

\crefalias{section}{appendix}
\renewcommand{\theoremsection}{\Alph{section}}

\input{sections/appendix_tangential_composition_reductions}

\input{sections/appendix_finite_constraint_design_problems}

\input{sections/appendix_beyond_convexity}

\end{document}

%% file: latex/packages.tex
\usepackage{amsmath}
\usepackage{amssymb}
\usepackage{amsthm}
\usepackage{cite}

\usepackage{xcolor}
\usepackage{graphicx}
\usepackage{subcaption}
\usepackage{tikz}
\usepackage{ifthen}
\usepackage{mathtools}
\usepackage{pgfplotstable}
\usepackage[acronym]{glossaries}
\usepackage{comment}
\usepackage[hidelinks, citecolor=blue, urlcolor=cyan]{hyperref}
\usepackage[capitalise]{cleveref}

%% file: latex/tikz_utils.tex
\definecolor{dpgreen}{rgb}{0.094869,0.500000,0.000000}
\definecolor{dpred}{rgb}{0.555789,0.000000,0.000000}
\definecolor{dpyellow}{RGB}{214,120,5}

\usetikzlibrary{positioning,intersections, hobby, patterns, calc, decorations.pathmorphing, decorations.markings, shadows,shapes, cd,arrows.meta, fit,quotes}

\tikzset{
   tick/.style={postaction={
      decorate,
      decoration={markings, mark=at position 0.5 with {\draw[-] (0,.4ex) -- (0,-.4ex);}}}
   }
}
\tikzstyle{block} = [draw, rectangle, minimum height=2em, minimum width=3em,font=\bfseries,rounded corners,thick]
\tikzstyle{block} = [draw, rectangle, minimum height=2em, minimum width=3em]
\tikzstyle{block1} = [draw, rectangle, minimum height=1.5em, minimum width=2.5em]
\tikzstyle{blockDyn} = [draw, rectangle, minimum height=2.5em, minimum width=3.5em, align=center, inner sep=10pt, thick, fill=white, copy shadow={draw=black,fill=black,opacity=1,shadow xshift=0.5ex,shadow yshift=-0.5ex}]
\tikzstyle{blockAlg} = [draw, rectangle, minimum height=1.5em, minimum width=2.5em, align=center, inner sep=10pt, thick]
\tikzstyle{sum} = [draw,circle]

\tikzstyle{nodePre} = [circle, draw,inner sep=1pt,node contents={$\preceq$},thick]
\tikzstyle{nodePreEmpty} = [circle, draw,inner sep=1pt,thick]
\tikzstyle{nodePos} = [circle, draw,inner sep=1pt,node contents={$\posceq$},thick]
\tikzstyle{nodeProd} = [rectangle, draw,inner sep=4pt,node contents={$\times$},rounded corners,thick]
\tikzstyle{nodeSum} = [rectangle, draw,inner sep=4pt,node contents={$\mathbf{+}$},rounded corners,thick]

\definecolor{red}{rgb}{0.75, 0.0, 0.0}

\tikzset{fcname/.store in =\fcname, fcname={}}
\tikzset{funame/.store in =\funame, funame={}}
\tikzset{rcname/.store in =\rcname, rcname={}}
\tikzset{runame/.store in =\runame, runame={}}
\tikzset{whereres/.store in =\whereres, whereres=0.5}
\tikzset{wherefun/.store in =\wherefun, wherefun=0.5}
\tikzset{relres/.store in =\relres, relres={above}}
\tikzset{relfun/.store in =\relfun, relfun={above}}
\tikzset{posres/.store in =\posres, posres=1}
\tikzset{posfun/.store in =\posfun, posfun=1}
\tikzset{loos/.store in =\loos, loos=2}
\tikzset{feedback/.store in =\feedback, feedback=0}
\tikzset{
   DP/.style={%
      label/.style={
         font=\everymath\expandafter{\the\everymath\scriptstyle},
         inner sep=5pt,
         node distance=2pt and -2pt},
      semithick,
      node distance=1 and 1,
      rconn/.style={color=white,opacity=0.0,postaction={decorate}, shorten <=3.2pt, shorten >= 0.8,
      decoration={markings, 
      mark= at position 0 with {
               \coordinate (a);
      },
      mark=at position .5 with
      {
              \ifthenelse{\equal{\feedback}{1}}{\def\angleOut{90}\def\angleIn{90}}{\def\angleOut{0}\def\angleIn{180}}    
              \coordinate (b);
              \draw[dashed,dpred,opacity=1.0] (a) to[out=\angleOut,in=\angleIn,looseness=\loos] 
              node[pos=\posres,\relres=\whereres mm,dpred,opacity=1,fill=white,inner sep=1pt,outer sep=1pt]{\footnotesize{\rcname}} (b);
      },
      mark= at position 1 with 
      {
             \ifthenelse{\equal{\feedback}{1}}{\def\angleOut{0}\def\angleIn{0}}{\def\angleOut{180}\def\angleIn{0}} 
              \ifthenelse{\equal{\feedback}{1}}{\def\symbol{\succeq}}{\def\symbol{\preceq}} 
              \coordinate (c);
              \draw[dpgreen,opacity=1.0] (c) to[out=\angleOut,in=\angleIn,looseness=\loos]
              node[pos=\posfun,\relfun=\wherefun mm,dpgreen,opacity=1,fill=white,inner sep=1pt,outer sep=1pt]{\footnotesize{\fcname}} (b){};
              \node[draw,circle,inner sep=0.5pt,color=black,fill=white,opacity=1.0] at (b) (nodepreceq) {$\symbol$}; 
      }
      }},
      runconn/.style={color=dpred,dashed,postaction={decorate},
      decoration={markings,
      mark= at position 1 with {
              \coordinate (a);
              \draw[dpred,opacity=1.0,dashed] ($(a)+(0.05,0)$) --++ (0.5,0) node[\relres,pos=\posres]{\footnotesize{\runame}};}
      }
      },
      funconn/.style={color=white,postaction={decorate},
      decoration={markings,
      mark= at position 0 with {
      \coordinate (a);
      \draw[dpgreen] ($(a)+(-0.05,0)$) -- ($(a)+(-0.5,0)$) node[\relfun, pos=\posfun]{\footnotesize{\funame}};}
      }
      },
      execute at begin picture={\tikzset{
         x=\dpx, y=\dpy,
         every fit/.style={inner xsep=\dpx, inner ysep=\dpy}}}
      },
   dpx/.store in=\dpx,
   dpx = 1.5cm,
   dpy/.store in=\dpy,
   dpy = 1.5ex,
   dp port sep/.store in=\dpportsep,
   dp port sep=2,
   dp port length/.store in=\dpportlen,
   dp port length=4pt,
   dp min width/.store in=\dpminwidth,
   dp min width=0.5cm,
   dp rounded corners/.store in=\dpcorners,
   dp rounded corners=2pt,
   dp small/.style={dp port sep=1, dp port length=2.5pt, dpx=.4cm, dp min width=.4cm, dpy=.7ex},
   dp/.code 2 args={%
      \pgfmathsetlengthmacro{\dpheight}{\dpportsep * (max(#1,#2)) * \dpy}
      \pgfkeysalso{draw,%
        minimum width=\dpminwidth,%
        minimum height=\dpheight,%
        font=\bfseries,
        outer sep=0pt,%
        inner sep=5pt,%
        rounded corners=\dpcorners,
        thick,
        prefix after command={\pgfextra{\let\fixname\tikzlastnode}},
        append after command={\pgfextra{\draw
            \ifnum #1=0{} \else foreach \i in {1,...,#1} { 
            ($(\fixname.north west)!{\i/(#1+1)}!(\fixname.south west)$) +(0,0) node[solid,left,circle,color=dpgreen,draw,fill=dpgreen,scale=0.3] {} coordinate (\fixname_fun\i) -- +(0,0) coordinate (\fixname_fun\i')}\fi
            \ifnum #2=0{} \else foreach \i in {1,...,#2} {
            ($(\fixname.north east)!{\i/(#2+1)}!(\fixname.south east)$) +(0,0) coordinate (\fixname_res\i') -- +(0,0) node[solid,right,circle,color=dpred,draw,fill=dpred,scale=0.3] {} coordinate (\fixname_res\i)}\fi;
         }}}
         },
      dp name/.style={append after command={\pgfextra{\node[label=center,inner sep=2pt,fill=white] at (\fixname) {\textbf{#1}};}}}
   }

%% file: latex/algorithm.tex
\newcounter{myalgorithm}

%% file: latex/theorems.tex
\newcommand{\theoremsection}{\arabic{section}}

\newtheorem{theorem}{Theorem}[section]

\newtheorem{lemma}[theorem]{Lemma}
\newtheorem{proposition}[theorem]{Proposition}
\newtheorem{corollary}[theorem]{Corollary}

\theoremstyle{definition}
\newtheorem{assumption}[theorem]{Assumption}
\newtheorem{example}[theorem]{Example}

\newtheorem{definition}[theorem]{Definition}

\theoremstyle{remark}
\newtheorem{remark}[theorem]{Remark}

\crefname{theorem}{Theorem}{Theorems}
\Crefname{theorem}{Theorem}{Theorems}
\crefname{lemma}{Lemma}{Lemmas}
\Crefname{lemma}{Lemma}{Lemmas}
\crefname{proposition}{Proposition}{Propositions}
\Crefname{proposition}{Proposition}{Propositions}
\crefname{corollary}{Corollary}{Corollaries}
\Crefname{corollary}{Corollary}{Corollaries}
\crefname{assumption}{Assumption}{Assumptions}
\Crefname{assumption}{Assumption}{Assumptions}
\crefname{example}{Example}{Examples}
\Crefname{example}{Example}{Examples}
\crefname{problem}{Problem}{Problems}
\Crefname{problem}{Problem}{Problems}
\crefname{definition}{Definition}{Definitions}
\Crefname{definition}{Definition}{Definitions}
\crefname{remark}{Remark}{Remarks}
\Crefname{remark}{Remark}{Remarks}

%% file: latex/commands.tex
\NewDocumentCommand{\meshal}{s +m}{%
  \begingroup%
  \color{magenta}\IfBooleanTF{#1}{#2}{Meshal: #2}%
  \endgroup%
}
\NewDocumentCommand{\yujun}{s +m}{
  \begingroup%
  \color{brown}\IfBooleanTF{#1}{#2}{Yujun: #2}%
  \endgroup%
}
\NewDocumentCommand{\alan}{s +m}{
  \begingroup%
  \color{blue}\IfBooleanTF{#1}{#2}{Alan: #2}%
  \endgroup%
}

\newcommand{\Set}{C}

\newcommand{\Graph}{\operatorname{gph}}
\newcommand{\Tan}{T}
\newcommand{\OpPoint}[1]{\tilde{#1}}
\newcommand{\PerturbationScale}{h}

\newcommand{\Ball}{\mathbb{B}}
\newcommand{\Closure}[1]{\operatorname{cl}(#1)}

\newcommand{\GenericFunctionalitySpace}{\mathcal F}
\newcommand{\GenericParameterSpace}{\Theta}
\newcommand{\GenericResourceSpace}{\mathcal R}
\newcommand{\GenericIntermediateSpace}{\mathcal Q}

\NewDocumentCommand{\Func}{e{_^}}{%
  [\underline{f\IfValueT{#1}{_{#1}}},\overline{f\IfValueT{#1}{_{#1}}}]%
  \IfValueT{#2}{^{#2}}}
\NewDocumentCommand{\Res}{e{_^}}{%
  [\underline{r\IfValueT{#1}{_{#1}}},\overline{r\IfValueT{#1}{_{#1}}}]%
  \IfValueT{#2}{^{#2}}}
\NewDocumentCommand{\ThetaSpace}{e{_^}}{%
  [\underline{\theta\IfValueT{#1}{_{#1}}},\overline{\theta\IfValueT{#1}{_{#1}}}]%
  \IfValueT{#2}{^{#2}}}
\NewDocumentCommand{\IntermediateSpace}{e{_^}}{%
  [\underline{q\IfValueT{#1}{_{#1}}},\overline{q\IfValueT{#1}{_{#1}}}]%
  \IfValueT{#2}{^{#2}}}

\newcommand{\FuncDim}{n_f}
\newcommand{\ThetaDim}{n_\theta}
\newcommand{\ResDim}{n_r}
\newcommand{\ConstraintDim}{n_g}

\NewDocumentCommand{\FuncAmbientSpace}{e{_^}}{%
  \IfValueTF{#2}
    {\bigl(\mathbb R^{n_{f\IfValueT{#1}{_{#1}}}}\bigr)^{#2}}
    {\mathbb R^{n_{f\IfValueT{#1}{_{#1}}}}}%
}
\NewDocumentCommand{\ThetaAmbientSpace}{e{_^}}{%
  \IfValueTF{#2}
    {\bigl(\mathbb R^{n_{\theta\IfValueT{#1}{_{#1}}}}\bigr)^{#2}}
    {\mathbb R^{n_{\theta\IfValueT{#1}{_{#1}}}}}%
}
\NewDocumentCommand{\ResAmbientSpace}{e{_^}}{%
  \IfValueTF{#2}
    {\bigl(\mathbb R^{n_{r\IfValueT{#1}{_{#1}}}}\bigr)^{#2}}
    {\mathbb R^{n_{r\IfValueT{#1}{_{#1}}}}}%
}

\newcommand{\DesignProblem}{\mathbf d}

\newcommand{\tdp}{\mathbf d^T}

\newcommand{\GenericQuery}{\overline{\mathbf d}}

\newcommand{\GenericTangentQuery}{\overline{\mathbf d}^{\,T}}
\newcommand{\GenericFiniteTangentQuery}{\overline{\mathbf d}^{\,T\Delta}}

\NewDocumentCommand{\QueryInputSpace}{e{_^}}{%
  [\underline{x\IfValueT{#1}{_{#1}}},\overline{x\IfValueT{#1}{_{#1}}}]%
  \IfValueT{#2}{^{#2}}}
\NewDocumentCommand{\QueryOutputSpace}{e{_^}}{%
  [\underline{y\IfValueT{#1}{_{#1}}},\overline{y\IfValueT{#1}{_{#1}}}]%
  \IfValueT{#2}{^{#2}}}
\NewDocumentCommand{\QueryParameterSpace}{e{_^}}{%
  [\underline{\theta\IfValueT{#1}{_{#1}}},\overline{\theta\IfValueT{#1}{_{#1}}}]%
  \IfValueT{#2}{^{#2}}}

\NewDocumentCommand{\QueryInputAmbientSpace}{e{_}}{\mathbb R^{n_{x\IfValueT{#1}{_{#1}}}}}
\NewDocumentCommand{\QueryOutputAmbientSpace}{e{_}}{\mathbb R^{n_{y\IfValueT{#1}{_{#1}}}}}
\NewDocumentCommand{\QueryParameterAmbientSpace}{e{_}}{\mathbb R^{n_{\theta\IfValueT{#1}{_{#1}}}}}
\NewDocumentCommand{\QueryIntermediateAmbientSpace}{e{_}}{\mathbb R^{n_{q\IfValueT{#1}{_{#1}}}}}

\newcommand{\ThetaInputDim}{n_{\theta,\mathrm{in}}}
\newcommand{\ThetaOutputDim}{n_{\theta,\mathrm{out}}}

\NewDocumentCommand{\TDPFuncAmbientSpace}{e{_^}}{%
  \IfValueTF{#2}
    {\bigl(\mathbb R^{n_{f\IfValueT{#1}{_{#1}}}}\bigr)^{#2}}
    {\mathbb R^{n_{f\IfValueT{#1}{_{#1}}}}}%
}
\NewDocumentCommand{\TDPThetaAmbientSpace}{e{_^}}{%
  \IfValueTF{#2}
    {\bigl(\mathbb R^{n_{\theta\IfValueT{#1}{_{#1}}}}\bigr)^{#2}}
    {\mathbb R^{n_{\theta\IfValueT{#1}{_{#1}}}}}%
}
\NewDocumentCommand{\TDPResAmbientSpace}{e{_^}}{%
  \IfValueTF{#2}
    {\bigl(\mathbb R^{n_{r\IfValueT{#1}{_{#1}}}}\bigr)^{#2}}
    {\mathbb R^{n_{r\IfValueT{#1}{_{#1}}}}}%
}
\NewDocumentCommand{\TDPQueryInputAmbientSpace}{e{_}}{\mathbb R^{n_{x\IfValueT{#1}{_{#1}}}}}
\NewDocumentCommand{\TDPQueryOutputAmbientSpace}{e{_}}{\mathbb R^{n_{y\IfValueT{#1}{_{#1}}}}}
\NewDocumentCommand{\TDPQueryParameterAmbientSpace}{e{_}}{\mathbb R^{n_{\theta\IfValueT{#1}{_{#1}}}}}
\newcommand{\TDPThetaInputAmbientSpace}{\mathbb R^{\ThetaInputDim}}
\newcommand{\TDPThetaOutputAmbientSpace}{\mathbb R^{\ThetaOutputDim}}

\NewDocumentCommand{\FiniteVariationFuncSpace}{e{_^} m}{%
  (\Func\IfValueT{#1}{_{#1}}\IfValueT{#2}{^{#2}}-#3)}
\NewDocumentCommand{\FiniteVariationThetaSpace}{e{_^} m}{%
  (\ThetaSpace\IfValueT{#1}{_{#1}}\IfValueT{#2}{^{#2}}-#3)}
\NewDocumentCommand{\FiniteVariationResSpace}{e{_^} m}{%
  (\Res\IfValueT{#1}{_{#1}}\IfValueT{#2}{^{#2}}-#3)}
\NewDocumentCommand{\FiniteVariationQueryInputSpace}{e{_^} m}{%
  (\QueryInputSpace\IfValueT{#1}{_{#1}}\IfValueT{#2}{^{#2}}-#3)}
\NewDocumentCommand{\FiniteVariationQueryOutputSpace}{e{_^} m}{%
  (\QueryOutputSpace\IfValueT{#1}{_{#1}}\IfValueT{#2}{^{#2}}-#3)}

\NewDocumentCommand{\FiniteVariationFuncAmbientSpace}{e{_}}{\mathbb R^{n_{f\IfValueT{#1}{_{#1}}}}}
\NewDocumentCommand{\FiniteVariationThetaAmbientSpace}{e{_}}{\mathbb R^{n_{\theta\IfValueT{#1}{_{#1}}}}}
\NewDocumentCommand{\FiniteVariationResAmbientSpace}{e{_}}{\mathbb R^{n_{r\IfValueT{#1}{_{#1}}}}}

\NewDocumentCommand{\FiniteVariationQueryInputAmbientSpace}{e{_}}{\mathbb R^{n_{x\IfValueT{#1}{_{#1}}}}}
\NewDocumentCommand{\FiniteVariationQueryOutputAmbientSpace}{e{_}}{\mathbb R^{n_{y\IfValueT{#1}{_{#1}}}}}
\NewDocumentCommand{\FiniteVariationQueryParameterAmbientSpace}{e{_}}{\mathbb R^{n_{\theta\IfValueT{#1}{_{#1}}}}}

\newcommand{\Trace}{\operatorname{Tr}}

\newcommand{\Real}{\mathcal{M}}

\newcommand{\CDPGraph}{\mathcal{G}}
\newcommand{\TCDPGraph}{\CDPGraph^{T}}
\newcommand{\InternalRealization}{\rho}
\newcommand{\Nodes}{\mathcal{V}}
\newcommand{\Edges}{\mathcal{E}}
\newcommand{\TSeries}{\mathbin{\circ^{T}}}
\newcommand{\TParallel}{\mathbin{\otimes^{T}}}
\newcommand{\TIntersection}{\mathbin{\wedge^{T}}}
\newcommand{\TUnion}{\mathbin{\vee^{T}}}
\newcommand{\TTrace}{\Trace^{T}}

\definecolor{functors}{rgb}{0.5,0,0.5}

\newcommand{\Upper}{\operatorname{U}}
\newcommand{\Min}{\operatorname{Min}}

%% file: latex/glossary.tex
\newacronym{dp}{DP}{design problem}
\newacronym{cdp}{CDP}{co-design problem}
\newacronym{tdp}{TDP}{tangential design problem}
\newacronym{tcdp}{TCDP}{tangential co-design problem}
\newacronym{cq}{CQ}{constraint qualification}
\newacronym{mfcq}{MFCQ}{Mangasarian--Fromovitz constraint qualification}
\newacronym{lp}{LP}{linear program}
\newacronym{molp}{MOLP}{multiobjective linear program}
\newacronym{ldp}{LDP}{linear design problem}
\newacronym{poset}{poset}{partially ordered set}
\newacronym{cpo}{CPO}{complete partial order}
\newacronym{dcpo}{DCPO}{directed-complete partial order}
\newacronym{dpi}{DPI}{design problem with implementation}

%% file: sections/abstract.tex
Sensitivity analysis of coupled systems typically propagates derivatives of
single-valued maps, whereas multiobjective optimization generally treats
sensitivity at the system level rather than through interconnected subsystems.
We develop a mathematical framework for compositional sensitivity analysis
in multiobjective system design, building on monotone co-design.
Subsystems are represented as design problems (DPs), interconnected through
series, parallel, union, intersection, and feedback composition.
We introduce tangential DPs (TDPs) to describe first-order variations around
an operating point, show that they preserve the structure of DPs, and
characterize their queries through graphical derivatives.
To account for multiple internal realizations of a system-level operating
point, tangential co-design problems compose families of local TDPs by
taking a union over these realizations while retaining the original
interconnection structure.
We show that these compositions yield families of TDPs and establish
conditions under which sensitivity analysis commutes with individual
composition operations.
Beyond infinitesimal sensitivity, we show that bounded tangential models
of convex DPs provide tighter outer approximations of feasible finite
changes, which can be propagated through co-design problems without
union compositions.
We conclude with several convex and nonconvex numerical examples that
showcase the framework's accuracy and computational benefits.

%% file: sections/introduction.tex
\section{Introduction}
\label{sec:introduction}

Sensitivity analysis has long played a central role in engineering design
by quantifying how design properties change under perturbations of design parameters~\cite{sobieszczanski1990sensitivity,martins2013derivatives}.
Compared with methods that repeatedly perturb and solve the original design problem, sensitivity
analysis can avoid computationally expensive and numerically inaccurate finite-difference
approximations~\cite{haug1978designsensitivity,adelman1986sensitivity,martins2013derivatives}.
Sensitivity information can guide the search for improved designs
and quantify the effects of uncertainty and errors~\cite{haug1978designsensitivity,adelman1986sensitivity,
martins2013derivatives}.

When engineering design involves multiple interacting components and competing objectives,
sensitivity must account both for the propagation of local perturbations through coupled
subsystems~\cite{sobieszczanski1990sensitivity,martins2013derivatives}
and for variations of system-level trade-offs~\cite{tanino1988sensitivity,shi1991contingent}.
Monotone co-design~\cite{zardini2023codesigncomplexsystems, censi2015mathematicaltheorycodesign} provides a principled framework for handling
compositional structures by representing each subsystem as a \gls{dp}, a relation between the
functionalities it provides and the resources required to provide them. By solving
local \glspl{dp} and propagating their solutions through the system interconnections,
this framework retains the system decomposition and supports modular and scalable design
across multiple application domains~\cite{neumann2026strategiccodesign, stralz2026taskdrivencodesignheterogeneousmultirobot, zardini2022taskdriven}.

\begin{figure}[t]
\centering
\includegraphics[width=\columnwidth]
    {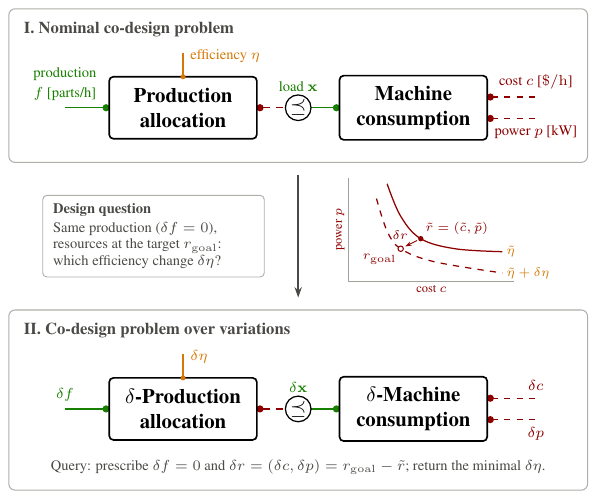}
\caption{
Around a given solution of a \glsentryshort{cdp} (I), questions about variations in resources, functionalities, or parameters can themselves be reformulated as a \glsentryshort{cdp} over variations (II).
}
\label{fig:machine-codesign}
\end{figure}

Recent extensions of monotone co-design have addressed
parametric uncertainty~\cite{huang2026distributional}
and informed sampling~\cite{alharbi2026compositional}.
Sensitivity information could complement these developments by
characterizing how parameter perturbations propagate through the
system composition to system-level trade-offs and by providing
local, gradient-like information to guide adaptive sampling.
However, sensitivity analysis has not yet been
investigated in the context of monotone co-design.
Multiobjective sensitivity methods can be applied after constructing the complete
system-level problem~\cite{tanino1988sensitivity,shi1991contingent,xue2020coderivatives},
but this approach loses the compositional structure.
Conversely, sensitivity methods for coupled multidisciplinary systems can propagate
local derivatives through a system
decomposition~\cite{sobieszczanski1990sensitivity,martins2013derivatives,gawthrop2000sensitivity}.
However, these methods typically rely on differentiable, single-valued
mappings (including those defined by parametric optimization
solvers~\cite{pirnay2012optimalsensitivity} or learned
surrogates~\cite{grontas2026pinet}), whereas a \gls{dp}
induces a set-valued map~\cite{zardini2023codesigncomplexsystems}. Consequently,
single-valued derivatives and their chain rule do not directly apply to monotone co-design.

This raises a natural question: what form should sensitivity analysis
take within the monotone co-design framework?

In this work, we address this question by viewing \glspl{dp} geometrically
as sets of realizable functionality--resource pairs.
Since the compositions of \glspl{dp} act directly on these sets,
their local geometry provides a natural basis for describing sensitivity.
We characterize this geometry through tangent cones and, when a \gls{dp}
is viewed as a set-valued map, through the corresponding graphical
derivatives~\cite{aubin2009setvalued,rockafellar1998variational}.
Intuitively, the graphical derivative fixes a point on the graph
and relates variations of the input to compatible variations of the output
through the tangent cone at that point.
This formulation retains the local set-valued variation of a \gls{dp},
which would be lost by reducing sensitivity to single-valued derivatives.

%% file: sections/problem_and_main_results.tex
\subsection{Problem setup}
\label{sec:problem-setup}

In this paper, we investigate the first-order sensitivity of a
\gls{cdp} \(\CDPGraph\)~\cite{zardini2023codesigncomplexsystems},
a system of interconnected \glspl{dp}.
In particular, we analyze the sensitivity of its associated query
\(\GenericQuery_{\CDPGraph}
:\QueryInputAmbientSpace\times\QueryParameterAmbientSpace
\rightrightarrows\QueryOutputAmbientSpace\),
a set-valued map that takes prescribed functionalities, resources,
and problem parameters as inputs and returns the functionalities
and resources that can be realized by \(\CDPGraph\).
Informally, our goal is to characterize how \(\GenericQuery_{\CDPGraph}\)
responds to small variations in its arguments. Given variations
\((\delta x,\delta\theta)\), we seek a first-order description of the change
from \(\GenericQuery_{\CDPGraph}(x,\theta)\) to
\(\GenericQuery_{\CDPGraph}
(x+\PerturbationScale\delta x,\theta+\PerturbationScale\delta\theta)\)
for small \(\PerturbationScale>0\).
This description plays a role analogous to that of a first-order
Taylor approximation, but for a set-valued map.

While the co-design framework allows general \glspl{poset} for functionalities
and resources, we restrict our analysis to Euclidean ambient spaces and
\glspl{dp} that are closed and contained in bounded boxes.
This setting covers a broad range of practical applications while providing
a foundation for sensitivity analysis.

\subsection{Main results}
\label{sec:main-results}

Our contribution is fourfold.

\textbf{C1.} We introduce \glspl{tdp}, \glspl{dp} over first-order
variations at an existing solution. We show in
Theorem~\ref{thm:tangent-model-design-problem} that this first-order
construction retains the structure of a \gls{dp}.
We characterize their queries in terms of the graphical
derivatives of the corresponding \gls{dp} queries.

\textbf{C2.} We introduce \glspl{tcdp}, which retain the co-design
interconnection structure while associating each block with the family of all
\glspl{tdp} that the block can ``generate''. We define how these families
compose under the standard co-design operations, and show in
Theorem~\ref{thm:tangential-design-problem-composition-closure} that each
member of the resulting family is itself a \gls{tdp}.
Consequently, a \gls{tcdp} induces a family of \glspl{tdp}
(Corollary~\ref{cor:tdp-family-composition-closure}).

\textbf{C3.} We analyze in Theorem~\ref{thm:composition-rules-system-sensitivity}
when individual co-design compositions commute with the sensitivity formulation.
For this analysis, we restrict the points under consideration to the interiors
of the boxes.
It remains an open question how to verify conditions on individual \glspl{dp}
that guarantee commutation at the system level.

\textbf{C4.} We then shift our attention from infinitesimally small
to finite variations
of the \gls{dp} by restricting the ambient variation spaces to the
bounded translated boxes induced by the original admissible boxes. Although
this formulation does not satisfy the assumptions of
Theorem~\ref{thm:composition-rules-system-sensitivity}, we show in
Theorem~\ref{thm:finite-variation-outer-approximation} that, for convex
\glspl{dp}, it yields a tighter outer approximation of feasible variations.
Moreover, for \glspl{cdp} containing no union compositions, this approximation
can be propagated through co-design composition using a single compatible internal
realization (Corollary~\ref{cor:finite-variations-composition}).

%% file: sections/preliminaries.tex
\section{Preliminaries}
\label{sec:preliminaries}

We first introduce the notation and conventions used throughout the paper.
We then review the main aspects of monotone co-design and 
variational analysis used in our work.

\begin{figure*}[!t]
    \centering
    \begin{subfigure}[b]{0.155\textwidth}
        \centering
        \input{tikz/general_dp.tikz}
        \subcaption{A \gls{dp}.}
        \label{fig:mathcodesign}
    \end{subfigure}%
    \hspace{0.012\textwidth}%
    \begin{subfigure}[b]{0.155\textwidth}
        \centering
        \input{tikz/series.tikz}
        \subcaption{Series ($\circ$).}
        \label{fig:dp-series-composition}
    \end{subfigure}%
    \hspace{0.010\textwidth}%
    \begin{subfigure}[b]{0.155\textwidth}
        \centering
        \input{tikz/parallel.tikz}
        \subcaption{Parallel ($\otimes$).}
        \label{fig:dp-parallel-composition}
    \end{subfigure}%
    \hspace{0.000\textwidth}%
    \begin{subfigure}[b]{0.155\textwidth}
        \centering
        \input{tikz/intersection.tikz}
        \subcaption{Intersection ($\wedge$).}
        \label{fig:dp-intersection-composition}
    \end{subfigure}%
    \hspace{-0.005\textwidth}%
    \begin{subfigure}[b]{0.155\textwidth}
        \centering
        \input{tikz/union.tikz}
        \subcaption{Union ($\vee$).}
        \label{fig:dp-union-composition}
    \end{subfigure}%
    \hspace{-0.02\textwidth}%
    \begin{subfigure}[b]{0.155\textwidth}
        \centering
        \input{tikz/loop.tikz}
        \subcaption{Loop ($\operatorname{Tr}$).}
        \label{fig:dp-feedback-composition}
    \end{subfigure}
    \caption{
        A \gls{dp} is a monotone relation between a \gls{poset} of functionalities and a
        \gls{poset} of resources for each parameter value~(a), and can be composed
        through the operations shown in~(b)--(f).
    }
    \label{fig:dp_and_compositions}
    \vspace{-3mm}
\end{figure*}
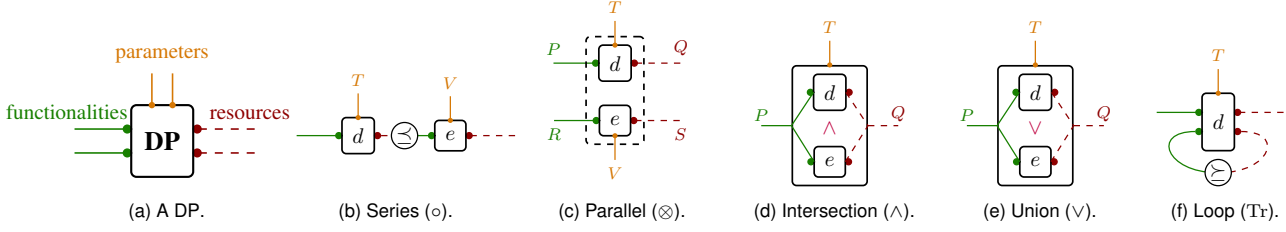

\subsection{Notation and conventions}
\label{subsec:notation-and-conventions}

For a set \(A\subseteq\mathbb{R}^n\), its closure is denoted by
\(\Closure{A}\), and \(\Ball\) denotes the closed unit ball.
For a family of sets \((A_i)_{i\in I}\), we denote their Cartesian
product by \(\prod_{i\in I} A_i\). For
\(a\in\prod_{i\in I} A_i\), we write
\(a=(a_i)_{i\in I}\), with \(a_i\in A_i\).
If \(A_i\subseteq\mathbb{R}^n\) for all \(i\in I\), then
\(\sum_{i\in I} A_i\) denotes their Minkowski sum.
For \(A\subseteq X\times Y\), its fiber at \(x\in X\) is
\(A_x:=\{y\in Y\mid (x,y)\in A\}\).
For a set \(X\), \(\operatorname{Id}_X\) denotes the identity map on \(X\).
Let \(F:X\rightrightarrows Y\) be a set-valued map. Its domain and graph are
\(\operatorname{dom}F:=\{x\in X\mid F(x)\neq\varnothing\}\) and
\(\Graph F:=\{(x,y)\in X\times Y\mid y\in F(x)\}\), respectively, while its inverse
\(F^{-1}:Y\rightrightarrows X\) is defined by
\(F^{-1}(y):=\{x\in X\mid y\in F(x)\}\).
Set operations on set-valued maps are understood pointwise.
For example, for set-valued maps \(F,G:X\rightrightarrows Y\),
\(
(F\cup G)(x):=F(x)\cup G(x),
\)
and inclusions and equalities between set-valued maps are likewise
understood pointwise.
Properties referring to their values are understood valuewise; for
example, \(F\) is closed-valued if \(F(x)\) is closed for every \(x\in X\).
Ordered Euclidean spaces are equipped with the componentwise order, i.e.,
\(x\preceq y \Leftrightarrow x_i\leq y_i \ \forall i\).
Vector inequalities, minima, and maxima are understood componentwise.
We write
\(
\mathbb R_+^n:=\{x\in\mathbb R^n\mid x\geq 0\}
\)
and
\(
\mathbb R_-^n:=\{x\in\mathbb R^n\mid x\leq 0\}.
\)
Products of ordered spaces are equipped with the product order.
We follow the standard order-theoretic conventions from~\cite{zardini2023codesigncomplexsystems, davey2002introduction}.
We write \(P=(P,\preceq_P)\) for a \gls{poset}.
Throughout the paper, we use the word \emph{preferred} as follows: if a quantity
is to be minimized, \(x\) is preferred to \(y\) if \(x \preceq y\); if it is
to be maximized, \(x\) is preferred to \(y\) if \(y \preceq x\).
The opposite \gls{poset} \(P^{\mathrm{op}}\) has the same elements as \(P\), with
\(
x\preceq_{P^{\mathrm{op}}}y
\Leftrightarrow
y\preceq_P x.
\)
For a \gls{poset} \((P,\preceq_P)\) and \(S\subseteq P\), the upper closure
of \(S\) is
\(
\uparrow S
:=
\{y\in P \mid \exists x\in S:\ x\preceq_P y\}.
\)
A set \(S\subseteq P\) is an upper set if \(\uparrow S=S\).
We denote by \(\Upper(P)\) the collection of upper sets of \(P\).
For \(A\subseteq P\), we denote its set of Pareto-minimal elements by
\(
\Min(A):=\{a\in A\mid \nexists\,a'\in A:
a'\preceq_P a,\ a'\neq a\}.
\)
Thus, \(\Min(A)\) is a set and may contain multiple incomparable elements.

\subsection{Background on monotone co-design}
\label{subsec:background-monotone-codesign}

Monotone co-design composes subsystem models of feasible
functionality--resource trade-offs into a system-level
model~\cite{censi2015mathematicaltheorycodesign,zardini2023codesigncomplexsystems}.

\smallskip
\noindent\emph{What does one design block describe?}
A \gls{dp} models one subsystem as the set of its feasible functionality--resource
trade-offs. In this work, we consider parametrized \glspl{dp},
where these trade-offs depend on a parameter \(\theta\in\GenericParameterSpace\).

\begin{definition}[Design problem]
\label{def:design-problem}
Given \glspl{poset} \(\GenericFunctionalitySpace\) and \(\GenericResourceSpace\) of functionalities and resources,
and a parameter space \(\GenericParameterSpace\) (not necessarily
ordered), a \gls{dp} is a subset
\(\DesignProblem\subseteq\GenericFunctionalitySpace\times\GenericParameterSpace\times\GenericResourceSpace\) such that, for every
\(\theta\in\GenericParameterSpace\), its fiber
\(
\DesignProblem_\theta
:=
\left\{(f,r)\in\GenericFunctionalitySpace\times\GenericResourceSpace
\;\middle|\;(f,\theta,r)\in\DesignProblem\right\}
\)
satisfies
\(
\DesignProblem_\theta
\in
\Upper(\GenericFunctionalitySpace^{\mathrm{op}}\times\GenericResourceSpace)
\).
\(
(f,\theta,r)
\in
\GenericFunctionalitySpace
\times
\GenericParameterSpace
\times
\GenericResourceSpace
\)
is said to be \emph{feasible} if
\(
(f,\theta,r)\in\DesignProblem.
\)
Since
\(\GenericFunctionalitySpace\times\{\theta\}\times\GenericResourceSpace\)
is isomorphic to
\(\GenericFunctionalitySpace\times\GenericResourceSpace\),
the fiber \(\DesignProblem_\theta\) is again a \gls{dp}~\cite{zardini2023codesigncomplexsystems}.
The singleton parameter set can be omitted when no confusion can arise.
\end{definition}

\cref{fig:machine-codesign} illustrates two \glspl{dp} linked by the machine loads \(x\).
Production allocation models how increasing the load \(x\) enables higher
production \(f\) at efficiency \(\eta\), while machine consumption models the
corresponding increase in cost \(c\) and power \(p\).
At fixed \(\theta\), the upper-set condition means that requesting less
functionality or allowing more resources preserves feasibility.
In this work, we specialize to
\(\GenericFunctionalitySpace=\Func\),
\(\GenericParameterSpace=\ThetaSpace\), and
\(\GenericResourceSpace=\Res\), and assume that each \(\DesignProblem\)
considered is closed.
Since sensitivity analysis concerns variations, we treat fixed quantities
as constants and exclude them from the corresponding spaces.
Accordingly, all coordinate intervals in the functionality, parameter,
and resource spaces are assumed to have strictly positive length,
while zero-dimensional spaces remain admissible.

\smallskip
\noindent\emph{What information does a \gls{dp} provide?}
A query is a set-valued map returning the feasible values of unspecified
functionality and resource coordinates.
Writing \(\Func=\Func_x\times\Func_y\) and
\(\Res=\Res_x\times\Res_y\), we collect the prescribed coordinates in
\(\QueryInputSpace\) and the returned coordinates in
\(\QueryOutputSpace\).
Resource coordinates used as inputs and functionality coordinates used as
outputs are equipped with the opposite order:
\(
\QueryInputSpace
:=
\Func_x\times\Res_x^{\mathrm{op}},
\QueryOutputSpace
:=
\Func_y^{\mathrm{op}}\times\Res_y.
\)
We identify \((x,\theta,y)\) with the corresponding \((f,\theta,r)\) through
this regrouping, including in set membership and map evaluations.

We adopt the set-valued map notation used in variational analysis~\cite{rockafellar1998variational} 
for the queries of monotone
co-design~\cite{zardini2023codesigncomplexsystems}. Accordingly,
we regard a query as a set-valued map between the corresponding ambient spaces, with empty 
values outside its effective domain.

\begin{definition}[Query]
\label{def:query-maps}
Let \(\DesignProblem\subseteq\Func\times\ThetaSpace\times\Res\) be a \gls{dp}, with \(\QueryInputSpace\) and \(\QueryOutputSpace\) as above.
For \(x=(f_x,r_x)\in\QueryInputSpace\) and \(\theta\in\ThetaSpace\), the associated query is
\(
\GenericQuery:
\QueryInputAmbientSpace\times\QueryParameterAmbientSpace
\rightrightarrows
\QueryOutputAmbientSpace,
\)
defined by
\(
\GenericQuery(x,\theta)
:=
\left\{
y=(f_y,r_y)\in\QueryOutputSpace
\;\middle|\;
((f_x,f_y),\theta,(r_x,r_y))\in\DesignProblem
\right\}.
\)
\end{definition}

For the production-allocation block in \cref{fig:machine-codesign},
the query returns feasible machine loads \(x\) for fixed production \(f\)
and efficiency \(\eta\).

\smallskip
\noindent\emph{How do interconnected blocks become one system?}
We use two descriptions of how \glspl{dp} are combined.
Composition operations act directly on \glspl{dp} to define the resulting
\gls{dp}, while a \gls{cdp} specifies the local design blocks and their
interconnections, which can be represented diagrammatically.
Series and feedback connect functionalities to resources across or within
subsystems, respectively; parallel combines subsystems, intersection imposes
simultaneous requirements, and union represents alternative choices.

The composition operations for \glspl{dp}, shown schematically in
\cref{fig:dp_and_compositions}, are defined as follows.\footnote{Any combination of these composition operations again yields a \gls{dp} (closure)~\cite{zardini2023codesigncomplexsystems}.}{}

\begin{definition}[Composition operations]
\label{def:composition-operations}

\emph{Series:}
For
\(\DesignProblem_1\subseteq\GenericFunctionalitySpace
\times\GenericParameterSpace_1\times\GenericIntermediateSpace\) and
\(\DesignProblem_2\subseteq\GenericIntermediateSpace
\times\GenericParameterSpace_2\times\GenericResourceSpace\), define
\(
(\DesignProblem_2\circ\DesignProblem_1)_{(\theta_1,\theta_2)}
:=
\{(f,r)\mid
\exists q\in\GenericIntermediateSpace:
(f,q)\in(\DesignProblem_1)_{\theta_1},
\ (q,r)\in(\DesignProblem_2)_{\theta_2}\}.
\)

\emph{Parallel:}
Given
\(\DesignProblem_1
\subseteq
\GenericFunctionalitySpace_1
\times\GenericParameterSpace_1
\times\GenericResourceSpace_1\)
and
\(\DesignProblem_2
\subseteq
\GenericFunctionalitySpace_2
\times\GenericParameterSpace_2
\times\GenericResourceSpace_2\),
define
\(
(\DesignProblem_1\otimes\DesignProblem_2)_{(\theta_1,\theta_2)}
:=
\{
\bigl((f_1,f_2),(r_1,r_2)\bigr)
\mid
(f_1,r_1)\in(\DesignProblem_1)_{\theta_1},
\ (f_2,r_2)\in(\DesignProblem_2)_{\theta_2}\}.
\)

\emph{Intersection and union:}
Given
\(
\DesignProblem_1,\DesignProblem_2
\subseteq
\GenericFunctionalitySpace
\times\GenericParameterSpace
\times\GenericResourceSpace
\),
define
\(
\DesignProblem_1\wedge\DesignProblem_2
:=\DesignProblem_1\cap\DesignProblem_2
\)
and
\(
\DesignProblem_1\vee\DesignProblem_2
:=\DesignProblem_1\cup\DesignProblem_2.
\)

\emph{Feedback:}
For \(\DesignProblem\subseteq
(\GenericFunctionalitySpace\times\GenericIntermediateSpace)
\times\GenericParameterSpace\times
(\GenericResourceSpace\times\GenericIntermediateSpace)\), set
\(
\Trace(\DesignProblem)_\theta
:=
\{(f,r)\mid
\exists q\in\GenericIntermediateSpace:
\bigl((f,q),(r,q)\bigr)\in\DesignProblem_\theta\}.
\)

\end{definition}

We now formalize the \gls{cdp} representation underlying
\cref{fig:machine-codesign,fig:rover-codesign-architecture,fig:drone-codesign}.

\begin{definition}[Co-design problem]
\label{def:codesign-problem}
A \gls{cdp} is a tuple
\(
\CDPGraph
=
\langle
\GenericFunctionalitySpace,
\GenericParameterSpace,
\GenericResourceSpace,
\langle\Nodes,\Edges\rangle
\rangle,
\)
where \(\Nodes\) indexes the local \glspl{dp}
\(
\DesignProblem_v
\subseteq
\GenericFunctionalitySpace_v
\times
\GenericParameterSpace_v
\times
\GenericResourceSpace_v,
\quad v\in\Nodes,
\)
and \(\Edges\) indexes their interconnections, with every functionality and
resource occurring in at most one edge.
Each
\(
e=\langle\langle v,i\rangle,\langle w,j\rangle\rangle\in\Edges
\)
connects matching \glspl{poset}
\(
\GenericResourceSpace_{v,i}
=
\GenericFunctionalitySpace_{w,j}
\)
and imposes
\(
r_e=f_e,
\)
with \(r_e:=r_{v,i}\) and \(f_e:=f_{w,j}\).

The unconnected functionalities and resources are called
\emph{system-level} functionalities and resources, with corresponding
\glspl{poset} \(\GenericFunctionalitySpace\) and
\(\GenericResourceSpace\).
The parameter space \(\GenericParameterSpace\) is obtained from the local
parameter spaces \(\GenericParameterSpace_v\) according to
\cref{def:composition-operations}.
\end{definition}

\cref{def:codesign-problem} uses equality on interconnections,
whereas in diagrams such as \cref{fig:machine-codesign} we use
\(r_e\preceq f_e\) to make explicit that one component may provide more than
another requires. By monotonicity, the two representations induce the same
system-level \gls{dp}.
If a functionality or resource must be interconnected through multiple edges,
introduce \glspl{dp} such as
\(\{(f_1,f_2,r)\mid f_1\preceq r,\ f_2\preceq r\}\)
and
\(\{(f,r_1,r_2)\mid f\preceq r_1,\ f\preceq r_2\}\).

To study sensitivity, we fix a feasible point of the system-level \gls{dp},
which we call an \emph{operating point} and denote by
\(
(\OpPoint{f},\OpPoint{\theta},\OpPoint{r})
\in\DesignProblem.
\)
For a \gls{cdp} \(\CDPGraph\), a system-level operating point
\((\OpPoint{f},\OpPoint{\theta},\OpPoint{r})\) may be obtained from different
choices of the internal interconnection quantities.
We call a choice
\(
\InternalRealization
:=
((\OpPoint{f}_v,\OpPoint{r}_v))_{v\in\Nodes}
\)
that makes every local \gls{dp} feasible and is compatible with
\((\OpPoint{f},\OpPoint{r})\) an \emph{internal realization}.
We denote by
\(
\Real_{\CDPGraph}
(\OpPoint{f},\OpPoint{r};
(\OpPoint{\theta}_v)_{v\in\Nodes})
\)
the set of such internal realizations \(\InternalRealization\).

The system-level \gls{dp} induced by \(\CDPGraph\) is the relation
\(\DesignProblem_{\CDPGraph}\subseteq
\Func\times\ThetaSpace\times\Res\) defined by
\begin{equation}
\label{eq:codesign-induced-design-problem}
    \DesignProblem_{\CDPGraph}
    :=
    \left\{
        (f,\theta,r)
        \;\middle|\;
        \begin{gathered}
            (f,\theta,r)\in\Func\times\ThetaSpace\times\Res,\\
            \Real_{\CDPGraph}
            \left(f,r;(\theta_v)_{v\in\Nodes}\right)\ne\varnothing
        \end{gathered}
    \right\}.
\end{equation}

\subsection{Background on variational analysis}
\label{subsec:background-variational-analysis}

We use variational analysis to describe how sets and set-valued maps change
under small perturbations.
The definitions are based on the work of Rockafellar and
Wets~\cite{rockafellar1998variational}.

Our analysis relies on three main ingredients: the tangent cone to describe
first-order variations of a set, the graphical derivative to relate input and
output variations of a set-valued map, and semidifferentiability to determine
when these models capture the local behavior of the underlying set
or map.

\begin{definition}[Tangent cone
{\cite[Definition~6.1]{rockafellar1998variational}}]
\label{def:tangent-cone}
Let \(\Set\subseteq\mathbb{R}^n\) and \(\OpPoint{z}\in\Set\).
The tangent cone to \(\Set\) at \(\OpPoint{z}\), denoted by
\(\Tan_{\Set}(\OpPoint{z})\subseteq\mathbb{R}^n\), is the set of directions
\(\delta z\) for which there exist sequences
\(\PerturbationScale_k\downarrow 0\) and
\(\delta z_k\to\delta z\) such that
\(\OpPoint{z}+\PerturbationScale_k\delta z_k\in\Set\)
for every \(k\).
\end{definition}

Tangent directions need not be straight feasible steps, i.e.,
\(
\delta z\in\Tan_{\Set}(\OpPoint{z})
\nRightarrow
\exists\,\PerturbationScale>0:
\OpPoint{z}+\PerturbationScale\delta z\in\Set
\),
but can be approached by feasible nearby variations.

If \(\mathbb{R}^n=\mathbb{R}^{n_x}\times\mathbb{R}^{n_y}\), we abbreviate
\(\Tan_{\Set}((\OpPoint{x},\OpPoint{y}))\) as
\(\Tan_{\Set}(\OpPoint{x},\OpPoint{y})\).

\begin{example}[Tangent cone to a box]
\label{rem:tangent-cone-box}
Given \(B=\prod_{i=1}^n[\underline z_i,\overline z_i]\) and \(z\in B\),
\(
T_B(z)=\prod_{i=1}^n T_{[\underline z_i,\overline z_i]}(z_i)
\),
where
\(
T_{[\underline z_i,\overline z_i]}(z_i)
=
\mathbb R
\)
if \(\underline z_i<z_i<\overline z_i\),
\(
T_{[\underline z_i,\overline z_i]}(z_i)
=
\mathbb R_+
\)
if \(z_i=\underline z_i<\overline z_i\),
\(
T_{[\underline z_i,\overline z_i]}(z_i)
=
\mathbb R_-
\)
if \(\underline z_i<z_i=\overline z_i\), and
\(
T_{[\underline z_i,\overline z_i]}(z_i)
=
\{0\}
\)
if \(\underline z_i=z_i=\overline z_i\).
\end{example}

\begin{definition}[Graphical derivative
{\cite[Definition~8.33]{rockafellar1998variational}}]
\label{def:graphical-derivative}

Let \(F:\mathbb{R}^{n_x}\rightrightarrows\mathbb{R}^{n_y}\) and let
\((\OpPoint{x},\OpPoint{y})\in\Graph F\).
The graphical derivative of \(F\) at
\((\OpPoint{x},\OpPoint{y})\) is the set-valued map
\(
DF(\OpPoint{x},\OpPoint{y}):
\mathbb{R}^{n_x}\rightrightarrows\mathbb{R}^{n_y}
\)
defined by
\(DF(\OpPoint{x},\OpPoint{y})(\delta x)
:=
\{\delta y\in\mathbb{R}^{n_y}
\mid
(\delta x,\delta y)\in
\Tan_{\Graph F}(\OpPoint{x},\OpPoint{y})\}\).
\end{definition}

At \((\OpPoint{x},\OpPoint{y})\), the tangent cone to the graph relates
input and output variations, while the graphical derivative collects the
compatible output variations for a prescribed input variation.

\begin{definition}[Semidifferentiability
{\cite[Eq.~(8.29), p.~332]{rockafellar1998variational}}]
\label{def:semidifferentiability}
Let \(F:\mathbb{R}^{n_x}\rightrightarrows\mathbb{R}^{n_y}\), and let
\((\OpPoint{x},\OpPoint{y})\in\Graph F\).
We say that \(F\) is \emph{semidifferentiable} at \(\OpPoint{x}\) for
\(\OpPoint{y}\) if
\(\operatorname{dom}DF(\OpPoint{x},\OpPoint{y})=\mathbb{R}^{n_x}\)
and, for all \(\delta x\in\mathbb{R}^{n_x}\),
\(\delta y\in DF(\OpPoint{x},\OpPoint{y})(\delta x)\),
\(\PerturbationScale_k\downarrow0\), and \(\delta x_k\to\delta x\),
there exists \(\delta y_k\to\delta y\) such that
\(\OpPoint{y}+\PerturbationScale_k\delta y_k
\in F(\OpPoint{x}+\PerturbationScale_k\delta x_k)\)
for all \(k\) large enough.
\end{definition}
Semidifferentiability rules out first-order predictions that are valid only
along some perturbation sequences.

We now introduce two notions that together provide
a sufficient condition for semidifferentiability.
The first is protodifferentiability:

\begin{definition}[Protodifferentiability
{\cite[Proposition~8.41]{rockafellar1998variational}}]
\label{def:protodifferentiability}
Let \(F:\mathbb{R}^{n_x}\rightrightarrows\mathbb{R}^{n_y}\) and
\((\OpPoint{x},\OpPoint{y})\in\Graph F\).
The map \(F\) is protodifferentiable at \(\OpPoint{x}\) for
\(\OpPoint{y}\) if, for every
\((\delta x,\delta y)\in\Tan_{\Graph F}(\OpPoint{x},\OpPoint{y})\),
there exists a mapping
\(\xi:[0,\varepsilon)\to\Graph F\) such that
\(\xi(0)=(\OpPoint{x},\OpPoint{y})\) and
\(\lim_{\PerturbationScale\downarrow0}
(\xi(\PerturbationScale)-(\OpPoint{x},\OpPoint{y}))
/\PerturbationScale
=(\delta x,\delta y)\).
\end{definition}

Protodifferentiability ensures that every variation described
by the graphical derivative can be realized by a curve on the
graph of the set-valued map.
\begin{figure}[t]
    \centering

    \includegraphics[width=0.80\linewidth]
    {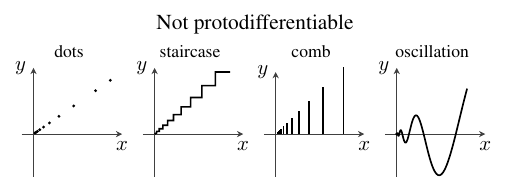}

    \vspace{0.5em}

    \includegraphics[width=0.80\linewidth]
    {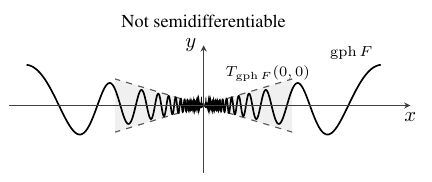}

    \caption{Top: graphs of maps that are not protodifferentiable.
    Bottom: graph of a protodifferentiable map that is not semidifferentiable.}
    \label{fig:regularity-properties}
\end{figure}
\cref{fig:regularity-properties} (top) illustrates several pathologies
excluded by protodifferentiability. In each case, the tangent cone contains
tangent directions generated by irregular local behavior that cannot be realized by a
first-order path on the graph.

The oscillation panel uses \(F(x)=x\sin(4\log x)\) for \(x>0\), with
\(F(0)=0\). The direction \((1,0)\) is tangent to its graph at the origin,
as witnessed by \(x_k=e^{-k\pi/4}\).
Assume \(\exists\,\xi:[0,\varepsilon)\to\Graph F\) such that
\(\xi(h)=(x(h),F(x(h)))\), \(\xi(0)=(0,0)\), and
\(\lim_{h\downarrow0}\xi(h)/h=(1,0)\).
Then \(x(h)/h\to1\).
Along \(h_k=e^{\pi/8-k\pi/2}\), however,
\(\log x(h_k)=\log h_k+o(1)\), so \(F(x(h_k))/h_k\to1\), contradicting
\(\lim_{h\downarrow0}F(x(h))/h=0\).
Thus, \(F\) is not protodifferentiable at the origin.

\cref{fig:regularity-properties} (bottom) shows the graph of
\(F(x)=|x|\sin(1/|x|)\), \(F(0)=0\), which is protodifferentiable at the
origin but not semidifferentiable: some tangent directions can be realized
along some perturbation sequences but not along others.

Rockafellar and Wets~\cite{rockafellar1998variational} show in
Proposition~9.50 that the additional stability needed to obtain
semidifferentiability is provided by the Aubin property.

\begin{definition}[Aubin property
{\cite[Definition~9.36]{rockafellar1998variational}}]
\label{def:aubin-property}
A set-valued map
\(F:\mathbb{R}^{n_x}\rightrightarrows\mathbb{R}^{n_y}\)
has the Aubin property around
\((\OpPoint{x},\OpPoint{y})\in\Graph F\) if \(\Graph F\) is locally closed at
\((\OpPoint{x},\OpPoint{y})\) and there exist neighborhoods \(U\) of \(\OpPoint{x}\)
and \(V\) of \(\OpPoint{y}\), and \(L\geq 0\), such that
\(
F(x_1)\cap V
\subseteq
F(x_2)+L\|x_1-x_2\|\Ball,
\)
for all \(x_1,x_2\in U\).
\end{definition}
The Aubin property is a local Lipschitz-type regularity condition around a
selected point \((\OpPoint{x},\OpPoint{y})\) of the graph. It ensures that small
changes in the input correspond to proportionally small changes in the outputs near
\(\OpPoint{y}\).

\cref{def:aubin-property,def:semidifferentiability} require
\(\OpPoint{x}\in\operatorname{int}(\operatorname{dom}F)\).
These regularity notions are very mild:
they mainly exclude pathological local behavior of the kind illustrated in
\cref{fig:regularity-properties},
and appear frequently in sensitivity analysis of optimization and generalized
equations; see, e.g.,
\cite{rockafellar1989protodifferentiability,rockafellar1998variational,li2009calculus}.

%% file: tikz/general_dp.tikz
\begin{tikzpicture}[DP]
            \node[dp={2}{2}] (cnt) {DP};
            \draw[runconn, runame={resources}, relres=above,posres=0.9] (cnt_res1){};
            \draw[runconn, runame={}, relres=above,posres=0.9] (cnt_res2){};
            \draw[funconn, funame={functionalities},relfun=above,posfun=1.15] (cnt_fun1){};
            \draw[funconn, funame={},relfun=above,posfun=1.15] (cnt_fun2){};
            \coordinate (cnt_param1) at ($(cnt.north west)!1/3!(cnt.north east)$);
            \coordinate (cnt_param2) at ($(cnt.north west)!2/3!(cnt.north east)$);
            \draw[dpyellow] (cnt_param1) -- ++(0,0.42cm);
            \draw[dpyellow] (cnt_param2) -- ++(0,0.42cm);
            \fill[dpyellow] (cnt_param1) circle[radius=1.2pt];
            \fill[dpyellow] (cnt_param2) circle[radius=1.2pt];
            \node[dpyellow, anchor=south] at ($(cnt.north)+(0,0.44cm)$)
                {\footnotesize parameters};
            \pgfresetboundingbox
            \path[use as bounding box]
                ($(cnt.center)+(-1.2cm,-0.6cm)$)
                rectangle
                ($(cnt.center)+(1.2cm,1.25cm)$);
\end{tikzpicture}

%% file: tikz/series.tikz
\scalebox{0.8}{\begin{tikzpicture}[DP]
    \node[dp={1}{1}] (f) {$d$};
    \node[dp={1}{1}, right=1cm of f] (g) {$e$};
    \draw[rconn, rcname={}, fcname={}] (f_res1)  to (g_fun1);
    \draw[runconn, runame={}] (g_res1);
    \draw[funconn, funame={}] (f_fun1);
    \draw[dpyellow] (f.north) -- ++(0,0.4cm);
    \draw[dpyellow] (g.north) -- ++(0,0.4cm);
    \fill[dpyellow] (f.north) circle[radius=1.2pt];
    \fill[dpyellow] (g.north) circle[radius=1.2pt];
    \node[dpyellow, anchor=south] at ($(f.north)+(0,0.4cm)$)
        {\footnotesize{$T$}};
    \node[dpyellow, anchor=south] at ($(g.north)+(0,0.4cm)$)
        {\footnotesize{$V$}};
\end{tikzpicture}}

%% file: tikz/parallel.tikz
\scalebox{0.8}{\begin{tikzpicture}[DP]
    % === Upper half: two independent DPs ===
    \node[dp={1}{1}] (d) {$d$};
    \node[dp={1}{1}, below=0.4cm of d] (e) {$e$};
    % Port connections with labels
    \draw[funconn, funame={$P$}, relfun=above] (d_fun1);
    \draw[runconn, runame={$Q$}, relres=above] (d_res1);
    \draw[funconn, funame={$R$}, relfun=below] (e_fun1);
    \draw[runconn, runame={$S$}, relres=below] (e_res1);
    % Dashed enclosing box
    \node[draw, dashed, rounded corners=2pt, thick,
          fit=(d)(e), inner xsep=6pt, inner ysep=4pt] (enc) {};
    \draw[dpyellow] (d.north) -- ++(0,0.4cm);
    \draw[dpyellow] (e.south) -- ++(0,-0.4cm);
    \fill[dpyellow] (d.north) circle[radius=1.2pt];
    \fill[dpyellow] (e.south) circle[radius=1.2pt];
    \node[dpyellow, anchor=south] at ($(d.north)+(0,0.4cm)$)
        {\footnotesize{$T$}};
    \node[dpyellow, anchor=north] at ($(e.south)+(0,-0.4cm)$)
        {\footnotesize{$V$}};
\end{tikzpicture}}

%% file: tikz/intersection.tikz
\scalebox{0.8}{\begin{tikzpicture}[DP]
    % === Upper half: intersection of two DPs ===
    \node[dp={1}{1}] (d) {$d$};
    \node[dp={1}{1}, below=0.6cm of d] (e) {$e$};
    % Solid enclosing box
    \node[draw, solid, rounded corners=2pt, thick,
          fit=(d)(e), inner xsep=10pt, inner ysep=4pt] (enc) {};
    % Split: input P forks to d and e
    \coordinate (split) at (enc.west);
    \draw[dpgreen] (split) -- (d_fun1);
    \draw[dpgreen] (split) -- (e_fun1);
    \draw[dpgreen] ($(split)+(-0.5cm,0)$) -- (split)
          node[above, pos=0, dpgreen, fill=white, inner sep=1pt, outer sep=1pt]
          {\footnotesize{$P$}};
    % Merge: outputs of d and e converge to single Q
    \coordinate (merge) at (enc.east);
    \draw[dpred, dashed] (d_res1) -- (merge);
    \draw[dpred, dashed] (e_res1) -- (merge);
    \draw[dpred, dashed] (merge) -- ($(merge)+(0.5cm,0)$)
          node[above, pos=1, dpred, fill=white, inner sep=1pt, outer sep=1pt]
          {\footnotesize{$Q$}};
    % Parameter connection for the composite box
    \draw[dpyellow] (enc.north) -- ++(0,0.4cm);
    \fill[dpyellow] (enc.north) circle[radius=1.2pt];
    \node[dpyellow, anchor=south] at ($(enc.north)+(0,0.4cm)$)
        {\footnotesize{$T$}};
    % Wedge symbol between boxes
    \node[text=purple] at ($(d.south)!0.5!(e.north)$) {$\wedge$};
\end{tikzpicture}}

%% file: tikz/union.tikz
\scalebox{0.8}{\begin{tikzpicture}[DP]
    % === Upper half: union of two DPs ===
    \node[dp={1}{1}] (d) {$d$};
    \node[dp={1}{1}, below=0.6cm of d] (e) {$e$};
    % Solid enclosing box
    \node[draw, solid, rounded corners=2pt, thick,
          fit=(d)(e), inner xsep=10pt, inner ysep=4pt] (enc) {};
    % Split: input P forks to d and e
    \coordinate (split) at (enc.west);
    \draw[dpgreen] (split) -- (d_fun1);
    \draw[dpgreen] (split) -- (e_fun1);
    \draw[dpgreen] ($(split)+(-0.5cm,0)$) -- (split)
          node[above, pos=0, dpgreen, fill=white, inner sep=1pt, outer sep=1pt]
          {\footnotesize{$P$}};
    % Merge: outputs of d and e converge to single Q
    \coordinate (merge) at (enc.east);
    \draw[dpred, dashed] (d_res1) -- (merge);
    \draw[dpred, dashed] (e_res1) -- (merge);
    \draw[dpred, dashed] (merge) -- ($(merge)+(0.5cm,0)$)
          node[above, pos=1, dpred, fill=white, inner sep=1pt, outer sep=1pt]
          {\footnotesize{$Q$}};
    % Parameter connection for the composite box
    \draw[dpyellow] (enc.north) -- ++(0,0.4cm);
    \fill[dpyellow] (enc.north) circle[radius=1.2pt];
    \node[dpyellow, anchor=south] at ($(enc.north)+(0,0.4cm)$)
        {\footnotesize{$T$}};
    % Union symbol between boxes
    \node[text=purple] at ($(d.south)!0.5!(e.north)$) {$\vee$};
\end{tikzpicture}}

%% file: tikz/loop.tikz
\scalebox{0.8}{\begin{tikzpicture}[
    DP,
    mirrored rconn/.style={
        color=white,
        opacity=0,
        postaction={decorate},
        shorten <=3.2pt,
        shorten >=0.8pt,
        decoration={
            markings,
            mark=at position 0 with {
                \coordinate (a);
            },
            mark=at position .5 with {
                \coordinate (b);
                \draw[dashed,dpred,opacity=1]
                    (a) to[out=-90,in=-90,looseness=3] (b);
            },
            mark=at position 1 with {
                \coordinate (c);
                \draw[dpgreen,opacity=1]
                    (c) to[out=0,in=0,looseness=3] (b);
                \node[
                    draw,
                    circle,
                    inner sep=0.5pt,
                    color=black,
                    fill=white,
                    opacity=1
                ] at (b) (nodepreceq) {$\succeq$};
            }
        }
    }
]
    \node[dp={2}{2}] (f) {$d$};

    % External upper functionality input and resource output.
    \draw[runconn, runame={}] (f_res1){};
    \draw[funconn, funame={}] (f_fun1){};

    % Vertically reflected form of the original rconn feedback route.
    \draw[mirrored rconn]
        (f_res2) -- ($(f)+(0,-4)$) |- (f_fun2);

    % Unlabelled parameter connection.
    \draw[dpyellow] (f.north) -- ++(0,0.4cm);
    \fill[dpyellow] (f.north) circle[radius=1.2pt];
    \node[dpyellow, anchor=south] at ($(f.north)+(0,0.4cm)$)
        {\footnotesize{$T$}};
\end{tikzpicture}}

%% file: sections/tangential_design_problems_and_queries.tex
\section{Tangential design problems and queries}
\label{sec:tangential-design-problems-and-queries}

We construct a first-order model around an existing solution
and characterize its induced query.
\subsection{Tangential design problems}
\label{subsec:tangential-design-problems}

We consider
\(
\Tan_{\DesignProblem}
(\OpPoint{f},\OpPoint{\theta},\OpPoint{r})
\)
as the basic ingredient of our first-order model around an operating point
\(
(\OpPoint{f},\OpPoint{\theta},\OpPoint{r})\in\DesignProblem
\).
The following result shows that the tangent cones naturally retain 
the upper-set structure of a \gls{dp}.

\begin{theorem}[Design-problem structure of the tangent cone]
\label{thm:tangent-model-design-problem}
Let
\(
\DesignProblem
\subseteq
\Func\times\ThetaSpace\times\Res
\)
be a \gls{dp}, and let
\(
(\OpPoint{f},\OpPoint{\theta},\OpPoint{r})\in\DesignProblem
\)
be such that
\(
\OpPoint{f}\in\operatorname{int}(\Func)
\)
and
\(
\OpPoint{r}\in\operatorname{int}(\Res).
\)
Then, for every
\(
\delta\theta\in\TDPThetaAmbientSpace,
\)
\(
\left(
\Tan_{\DesignProblem}
(\OpPoint{f},\OpPoint{\theta},\OpPoint{r})
\right)_{\delta\theta}
\in
\Upper\!\left(
    \left(\TDPFuncAmbientSpace\right)^{\mathrm{op}}
    \times
    \TDPResAmbientSpace
\right).
\)
\end{theorem}

\begin{proof}
Let
\(
(\delta f,\delta r)
\in
\left(
\Tan_{\DesignProblem}
(\OpPoint{f},\OpPoint{\theta},\OpPoint{r})
\right)_{\delta\theta},
\)
and let
\(
\delta f'\preceq\delta f
\)
and
\(
\delta r\preceq\delta r'
\).
By definition, there exist \(\PerturbationScale_k\downarrow0\) and
\(
(\delta f_k,\delta\theta_k,\delta r_k)
\to
(\delta f,\delta\theta,\delta r)
\)
such that
\(
(\OpPoint{f}+\PerturbationScale_k\delta f_k,
 \OpPoint{\theta}+\PerturbationScale_k\delta\theta_k,
 \OpPoint{r}+\PerturbationScale_k\delta r_k)
\in\DesignProblem.
\)
Define
\(
\delta f_k':=\delta f_k+\delta f'-\delta f
\)
and
\(
\delta r_k':=\delta r_k+\delta r'-\delta r
\).
Then
\(
\delta f_k'\to\delta f'
\),
\(
\delta r_k'\to\delta r'
\),
\(
\OpPoint{f}+\PerturbationScale_k\delta f_k'
\preceq
\OpPoint{f}+\PerturbationScale_k\delta f_k
\),
and
\(
\OpPoint{r}+\PerturbationScale_k\delta r_k
\preceq
\OpPoint{r}+\PerturbationScale_k\delta r_k'
\).
Since
\(
\OpPoint{f}\in\operatorname{int}(\Func)
\)
and
\(
\OpPoint{r}\in\operatorname{int}(\Res),
\)
and the sequences
\(
(\delta f_k')
\)
and
\(
(\delta r_k')
\)
are bounded, for all sufficiently large \(k\),
\(
\OpPoint{f}+\PerturbationScale_k\delta f_k'\in\Func
\)
and
\(
\OpPoint{r}+\PerturbationScale_k\delta r_k'\in\Res.
\)
Since \(\DesignProblem_{(\OpPoint{\theta}+\PerturbationScale_k\delta\theta_k)}\)
is upper in \(\Func^{\mathrm{op}}\times\Res\),
\(
(\OpPoint{f}+\PerturbationScale_k\delta f_k',
 \OpPoint{\theta}+\PerturbationScale_k\delta\theta_k,
 \OpPoint{r}+\PerturbationScale_k\delta r_k')
\in\DesignProblem.
\)
Moreover,
\(
(\delta f_k',\delta\theta_k,\delta r_k')
\to
(\delta f',\delta\theta,\delta r').
\)
Hence
\(
(\delta f',\delta\theta,\delta r')
\in
\Tan_{\DesignProblem}
(\OpPoint{f},\OpPoint{\theta},\OpPoint{r}),
\)
which proves the statement.
\end{proof}

The interior assumption is essential: at the boundary of the bounded domains,
\(\Tan_{\DesignProblem}\) need not be upper in the ambient space, as the
following example shows.

\begin{example}
Consider one-dimensional functionality and resource spaces
\(\FuncAmbientSpace=\ResAmbientSpace=\mathbb R\), with
\(\Func=\Res=[0,1]\) and \(\DesignProblem=\Func\times\Res\).
At \((\OpPoint{f},\OpPoint{r})=(0,1)\),
\(\Tan_{\DesignProblem}(0,1)=[0,\infty)\times(-\infty,0]\),
which is not upper in
\(\TDPFuncAmbientSpace^{\mathrm{op}}\times\TDPResAmbientSpace\),
whereas
\(\Tan_{\DesignProblem}(0,1)+(-\mathbb R_+\times\mathbb R_+)=\mathbb R^2\).
\end{example}

\cref{thm:tangent-model-design-problem} allows us to regard the tangent
cone itself as a \gls{dp}, provided that the operating point does not lie
on the boundary of the functionality or resource domains. To preserve the design-problem 
structure at boundary operating points as well, we
add the ordering cone
\(
K:=
-\mathbb R^{\FuncDim}_{+}
\times
\{0\}^{\ThetaDim}
\times
\mathbb R^{\ResDim}_{+}.
\)

\begin{definition}[Tangential design problem]
\label{def:tangential-design-problem}
Let $\DesignProblem\subseteq\Func\times\ThetaSpace\times\Res$ be a \gls{dp} and let
\((\OpPoint{f},\OpPoint{\theta},\OpPoint{r})\in\DesignProblem\). 
The \gls{tdp} at
\(
(\OpPoint{f},\OpPoint{\theta},\OpPoint{r})
\)
is the \gls{dp}
\(
\tdp_{\OpPoint{f},\OpPoint{\theta},\OpPoint{r}}
:=
\Tan_{\DesignProblem}
(\OpPoint{f},\OpPoint{\theta},\OpPoint{r})
+
K.
\)
We denote by
\(
\tdp
:=
\left\{
\tdp_{\OpPoint{f},\OpPoint{\theta},\OpPoint{r}}
\;\middle|\;
(\OpPoint{f},\OpPoint{\theta},\OpPoint{r})\in\DesignProblem
\right\}
\)
the family of \glspl{tdp} obtained from the feasible points of
\(\DesignProblem\).
\end{definition}

This addition restores upperness at boundary points while leaving the
construction unchanged in the interior:
\(
\Tan_{\DesignProblem}
(\OpPoint{f},\OpPoint{\theta},\OpPoint{r})+K
=
\Tan_{\DesignProblem}
(\OpPoint{f},\OpPoint{\theta},\OpPoint{r}),
\forall \OpPoint{f}\in\operatorname{int}(\Func),
\OpPoint{r}\in\operatorname{int}(\Res).
\)

\begin{lemma}[\glspl{tdp} as tangent cones]
\label{lem:tdp-tangent-ordering-cone-extension}
Let \(\DesignProblem\subseteq\Func\times\ThetaSpace\times\Res\) be a \gls{dp},
let \((\OpPoint f,\OpPoint\theta,\OpPoint r)\in\DesignProblem\), and let
\(K=-\mathbb{R}^{\FuncDim}_+\times\{0\}^{\ThetaDim}
\times\mathbb{R}^{\ResDim}_+\).
Then
\(
\tdp_{\OpPoint f,\OpPoint\theta,\OpPoint r}
= 
\Tan_{\DesignProblem+K}
(\OpPoint f,\OpPoint\theta,\OpPoint r).
\)
In particular,
\(\tdp_{\OpPoint f,\OpPoint\theta,\OpPoint r}\)
is a closed cone.
\end{lemma}

\begin{proof}
Take
\((\delta f,\delta\theta,\delta r)
\in
\Tan_{\DesignProblem}(\OpPoint f,\OpPoint\theta,\OpPoint r)\)
and \(\kappa\in K\).
There exist \(h_k\downarrow0\) and
\((\delta f_k,\delta\theta_k,\delta r_k)
\to(\delta f,\delta\theta,\delta r)\) such that
\((\OpPoint f+h_k\delta f_k,
\OpPoint\theta+h_k\delta\theta_k,
\OpPoint r+h_k\delta r_k)\in\DesignProblem\).
Since \(h_k\kappa\in K\),
\((\delta f,\delta\theta,\delta r)+\kappa
\in
\Tan_{\DesignProblem+K}
(\OpPoint f,\OpPoint\theta,\OpPoint r)\).
Conversely, take
\((\delta f,\delta\theta,\delta r)
\in
\Tan_{\DesignProblem+K}
(\OpPoint f,\OpPoint\theta,\OpPoint r)\).
There exist \(h_k\downarrow0\),
\((\delta f_k,\delta\theta_k,\delta r_k)
\to(\delta f,\delta\theta,\delta r)\),
\((f_k,\theta_k,r_k)\in\DesignProblem\), and
\(\kappa_k\in K\) such that
\(
(\OpPoint f+h_k\delta f_k,
 \OpPoint\theta+h_k\delta\theta_k,
 \OpPoint r+h_k\delta r_k)
=
(f_k,\theta_k,r_k)+\kappa_k.
\)
Hence \(\theta_k=\OpPoint\theta+h_k\delta\theta_k\),
\(\OpPoint f+h_k\delta f_k\preceq f_k\), and
\(r_k\preceq\OpPoint r+h_k\delta r_k\).
Define
\(
f'_k
:=
\max\{\underline f,\OpPoint f+h_k\delta f_k\},
\qquad
r'_k
:=
\min\{\overline r,\OpPoint r+h_k\delta r_k\},
\)
componentwise.
Then \(f'_k\preceq f_k\) and \(r_k\preceq r'_k\), so upperness of
\((\DesignProblem)_{\theta_k}\) gives
\((f'_k,\theta_k,r'_k)\in\DesignProblem\).
Set
\(\delta f'_k:=(f'_k-\OpPoint f)/h_k\) and
\(\delta r'_k:=(r'_k-\OpPoint r)/h_k\).
Since \(\delta f_k\to\delta f\) and \(\delta r_k\to\delta r\),
the componentwise max and min above imply that
\(\delta f'_k\to\delta f'\) and
\(\delta r'_k\to\delta r'\) for some finite
\(\delta f'\) and \(\delta r'\), with
\(\delta f-\delta f'\in-\mathbb{R}^{\FuncDim}_+\) and
\(\delta r-\delta r'\in\mathbb{R}^{\ResDim}_+\).
Therefore
\((\delta f',\delta\theta,\delta r')
\in
\Tan_{\DesignProblem}(\OpPoint f,\OpPoint\theta,\OpPoint r)\)
and
\((\delta f-\delta f',0,\delta r-\delta r')\in K\).
Thus
\((\delta f,\delta\theta,\delta r)
\in
\Tan_{\DesignProblem}(\OpPoint f,\OpPoint\theta,\OpPoint r)+K\).
\end{proof}

Since \glspl{tdp} are themselves \glspl{dp}, their queries are defined
analogously to those of the original \glspl{dp}. We additionally allow
parameter variations to be prescribed as inputs or returned as outputs. 
Accordingly, we write
\(
\TDPThetaAmbientSpace
=
\TDPThetaInputAmbientSpace\times\TDPThetaOutputAmbientSpace
\)
and
\(
\delta\theta
=
(\delta\theta_{\mathrm{in}},\delta\theta_{\mathrm{out}})
\).

\begin{definition}[Tangential query]
\label{def:tangential-query-maps}
Let
\(
\DesignProblem
\subseteq
\QueryInputSpace\times\ThetaSpace\times\QueryOutputSpace
\)
be a \gls{dp}, and let
\(
(\OpPoint{x},\OpPoint{\theta},\OpPoint{y})
\in\DesignProblem
\).
The tangential query
\(
\GenericTangentQuery_{\OpPoint{x},\OpPoint{\theta},\OpPoint{y}}:
\TDPQueryInputAmbientSpace\times\TDPThetaInputAmbientSpace
\rightrightarrows
\TDPQueryOutputAmbientSpace\times\TDPThetaOutputAmbientSpace
\)
is defined by
\(
\GenericTangentQuery_{\OpPoint{x},\OpPoint{\theta},\OpPoint{y}}
(\delta x,\delta\theta_{\mathrm{in}})
:=
\{(\delta y,\delta\theta_{\mathrm{out}})
\mid
(\delta x,
(\delta\theta_{\mathrm{in}},\delta\theta_{\mathrm{out}}),
\delta y)
\in
\tdp_{\OpPoint{x},\OpPoint{\theta},\OpPoint{y}}\}.
\)
\end{definition}

If no coordinates are assigned to one of the input or output spaces, we omit
the corresponding term.
We denote by \(K_x\) and \(K_y\) the ordering cones of the query input
and output spaces, respectively.

\begin{proposition}[Tangential queries as graphical derivatives]
\label{prop:tangential-queries-graphical-derivatives}
Let
\(\DesignProblem\subseteq\QueryInputSpace\times\ThetaSpace\times\QueryOutputSpace\)
be a \gls{dp}, and let
\(
(\OpPoint{x},\OpPoint{\theta},\OpPoint{y})\in\DesignProblem.
\)
Suppose that all parameter variations are prescribed,
so that \(\ThetaOutputDim=0\).
If
\(
\OpPoint{x}\in\operatorname{int}(\QueryInputSpace),
\)
the corresponding tangential query satisfies
\(
\GenericTangentQuery_{\OpPoint{x},\OpPoint{\theta},\OpPoint{y}}
(\delta x,\delta\theta)
=
D\GenericQuery
\bigl((\OpPoint{x},\OpPoint{\theta}),\OpPoint{y}\bigr)
(\delta x,\delta\theta)
+
K_y.
\)
If
\(
\OpPoint{x}\notin\operatorname{int}(\QueryInputSpace),
\)
then
\[
\GenericTangentQuery_{\OpPoint{x},\OpPoint{\theta},\OpPoint{y}}
(\delta x,\delta\theta)
=
\bigcup_{\mathmakebox[1.5em][c]{\kappa_x\in K_x}}
\left(
D\GenericQuery
\bigl((\OpPoint{x},\OpPoint{\theta}),\OpPoint{y}\bigr)
(\delta x+\kappa_x,\delta\theta)
+
K_y
\right).
\]
\end{proposition}

\begin{proof}
By \cref{def:graphical-derivative},
\(
D\GenericQuery
\bigl((\OpPoint{x},\OpPoint{\theta}),\OpPoint{y}\bigr)
\)
is obtained by taking the corresponding fibers of
\(
\Tan_{\DesignProblem}
(\OpPoint{x},\OpPoint{\theta},\OpPoint{y}).
\)
Moreover,
\(
K=(-K_x)\times\{0\}^{\ThetaDim}\times K_y.
\)
Therefore, taking the fiber of
\(
\tdp_{\OpPoint{x},\OpPoint{\theta},\OpPoint{y}}
=
\Tan_{\DesignProblem}
(\OpPoint{x},\OpPoint{\theta},\OpPoint{y})+K
\)
at \((\delta x,\delta\theta)\) gives the stated expression.
If
\(
\OpPoint{x}\in\operatorname{int}(\QueryInputSpace),
\)
the tangent cone is already upper in the \(-K_x\)-directions, so the union
reduces to the term corresponding to \(\kappa_x=0\), yielding the first
expression.
\end{proof}

%% file: sections/compositional_sensitivity_in_codesign.tex
\section{Compositional sensitivity in co-design}
\label{sec:compositional-sensitivity-in-codesign}

Since \glspl{dp} are composed to form system-level models,
their first-order counterparts should admit the same operations.
We therefore define a framework for composing families of \glspl{tdp} and show
that the resulting models are again \glspl{tdp}.

\begin{proposition}[\glspl{tdp} can form a \gls{cdp}]
\label{prop:tangential-codesign-problem}

Let \(\CDPGraph\) be a \gls{cdp} with local \glspl{dp}
\((\DesignProblem_v)_{v\in\Nodes}\), and let
\(
(\OpPoint{f},\OpPoint{\theta},\OpPoint{r})
\in\DesignProblem_{\CDPGraph}.
\)
For every
\(
\InternalRealization
\in
\Real_{\CDPGraph}
\left(
    \OpPoint{f},\OpPoint{r};
    (\OpPoint{\theta}_v)_{v\in\Nodes}
\right),
\)
replacing each \(\DesignProblem_v\) by
\(
\tdp_{v,\OpPoint{f}_v,\OpPoint{\theta}_v,\OpPoint{r}_v}
\)
while retaining the same connections defines a valid \gls{cdp}.
\end{proposition}

\begin{proof}
By \cref{def:tangential-design-problem}, each local \gls{tdp} is a
\gls{dp}, and connected ports induce compatible variation spaces.
\end{proof}

A natural choice would now be to define the tangential counterpart of a
\gls{cdp} for a fixed internal realization \(\InternalRealization\).
However, a system-level operating point may admit multiple internal
realizations. To account for this, we introduce a more flexible definition
that works directly with the families of \glspl{tdp} defined in
\cref{def:tangential-design-problem}.

\begin{definition}[Tangential co-design problem]
\label{def:tangential-codesign-problem}

Let
\(\CDPGraph
=
\left\langle
\Func,\ThetaSpace,\Res,
\langle\Nodes,\Edges\rangle
\right\rangle\)
be a \gls{cdp} with \glspl{dp}
\((\DesignProblem_v)_{v\in\Nodes}\).
The associated \gls{tcdp} is the tuple
\(\TCDPGraph
=
\left\langle
\TDPFuncAmbientSpace,
\TDPThetaAmbientSpace,
\TDPResAmbientSpace,
\langle\Nodes,\Edges\rangle
\right\rangle\).
Here, \(\TDPFuncAmbientSpace\), \(\TDPThetaAmbientSpace\), and
\(\TDPResAmbientSpace\) are the ambient Euclidean spaces of
\(\Func\), \(\ThetaSpace\), and \(\Res\), respectively.
The set \(\Nodes\) indexes the \gls{tdp} families
\((\tdp_v)_{v\in\Nodes}\) associated with
\((\DesignProblem_v)_{v\in\Nodes}\) according to
\cref{def:tangential-design-problem}, while
\(\Edges\) indexes their connections.
\end{definition}

For
\(
(\OpPoint f,\OpPoint\theta,\OpPoint r)
\in\DesignProblem_{\CDPGraph}
\)
and
\(
\InternalRealization
\in
\Real_{\CDPGraph}
\left(
\OpPoint f,\OpPoint r;
(\OpPoint\theta_v)_{v\in\Nodes}
\right),
\)
we write \(\TCDPGraph_{\InternalRealization}\) for the \gls{cdp}
obtained by selecting
\(
\tdp_{v,\OpPoint f_v,\OpPoint\theta_v,\OpPoint r_v}
\)
at each \(v\in\Nodes\), where
\(
(\OpPoint f_v,\OpPoint\theta_v,\OpPoint r_v)
\in\DesignProblem_v
\)
is specified by \(\InternalRealization\).
Following~\eqref{eq:codesign-induced-design-problem}, its induced
system-level \gls{dp} is denoted by
\(\DesignProblem_{\TCDPGraph_{\InternalRealization}}\).

\glspl{tcdp} can be represented using the same diagrams as \glspl{cdp}
by replacing the original variables by the corresponding variation variables, as
illustrated in \cref{fig:machine-codesign}.

Since each \gls{tdp} is itself a \gls{dp},
\cref{def:composition-operations} already allows us to compose \glspl{tdp}.
It does not, however, define how the families
of \glspl{tdp} should be composed.
We therefore introduce the corresponding family-level composition.

\begin{definition}[Composition of \gls{tdp} families]
\label{def:composition-tangential-design-problems}

The composition operations for families of \glspl{tdp}
are the following.

\emph{Series:}
Let
\(\DesignProblem_1\subseteq\Func\times\ThetaSpace_1\times\IntermediateSpace\)
and
\(\DesignProblem_2\subseteq\IntermediateSpace\times\ThetaSpace_2\times\Res\).
Let \(\CDPGraph\) be the \gls{cdp} consisting of
\(\DesignProblem_1\) and \(\DesignProblem_2\) connected in series through
\(\IntermediateSpace\).

Define the tangential series composition of $\tdp_1$ and $\tdp_2$ as the family of \glspl{tdp}
\[
\tdp_2\TSeries\tdp_1
:=
\{
(\tdp_2\TSeries\tdp_1)
_{\OpPoint f,(\OpPoint{\theta}_1,\OpPoint{\theta}_2),\OpPoint r}
 \mid
(\OpPoint f,(\OpPoint{\theta}_1,\OpPoint{\theta}_2),\OpPoint r)
\in
\DesignProblem_\CDPGraph
\},
\]
where each \gls{tdp} in the family
is obtained as follows:
\begin{equation}
(\tdp_2\TSeries\tdp_1)
_{\OpPoint f,(\OpPoint{\theta}_1,\OpPoint{\theta}_2),\OpPoint r}
:=
\Closure{
\bigcup_{\rho\in
\Real_{\CDPGraph}
(\OpPoint f,\OpPoint r;
(\OpPoint{\theta}_1,\OpPoint{\theta}_2))}
\DesignProblem_{\TCDPGraph_\rho}
}.
\label{eq:tdp-series-composition}
\end{equation}

\emph{Feedback:}
Let
\(
\DesignProblem
\subseteq
(\Func\times\IntermediateSpace)\times\ThetaSpace
\times(\Res\times\IntermediateSpace)
\).
Let \(\CDPGraph\) be the \gls{cdp} obtained by connecting the
\(\IntermediateSpace\)-resource of \(\DesignProblem\) in feedback to its
\(\IntermediateSpace\)-functionality.

Define the tangential feedback composition of $\tdp$ as the family
\[
\TTrace(\tdp)
:=
\left\{
\TTrace(\tdp)_{\OpPoint f,\OpPoint{\theta},\OpPoint r}
\;\middle|\;
(\OpPoint f,\OpPoint{\theta},\OpPoint r)
\in \DesignProblem_\CDPGraph
\right\},
\]
with each indexed \gls{tdp} in the family defined as follows:
\begin{equation}
\TTrace(\tdp)_{\OpPoint f,\OpPoint{\theta},\OpPoint r}
:=
\Closure{
\bigcup_{\rho\in
\Real_{\CDPGraph}
(\OpPoint f,\OpPoint r;\OpPoint{\theta})}
\DesignProblem_{\TCDPGraph_\rho}
}.
\label{eq:tdp-feedback-composition}
\end{equation}

\emph{Parallel, intersection, and union:}
For these three compositions, at each operating point, the composition is
exactly the corresponding composition of
\cref{def:composition-operations} applied to the \glspl{tdp}.
For union, only the branches containing the operating point are considered.
We denote the resulting family-level operations by
\(\TParallel\), \(\TIntersection\), and \(\TUnion\), respectively.  
\end{definition}

The use of closures in the series and feedback definitions is a choice in our
formulation, introduced to preserve closedness. For the \glspl{dp} considered here, series and feedback preserve closedness because internal variables lie in
bounded boxes. For \glspl{tdp}, internal variations range over unbounded
ambient spaces, and closedness may be lost. One drawback of our formulation is
the loss of associativity of series composition, as the following example
shows.

\begin{example}
\label{ex:nonclosed-tdp-series}

Let all nontrivial spaces be \([-1,1]\), and consider three \glspl{dp} in
series,
\(\DesignProblem_0=\{(x,f)\mid f\geq0\}\),
\(\DesignProblem_1
=\{(f,\theta,q)\mid f\leq0,\ q\geq0,\ \theta^2\leq-fq\}\),
and
\(\DesignProblem_2=\{(q,r)\mid q^3\leq r\}\).
At zero, the selected \glspl{tdp} are
\((\tdp_0)_{0,0}=\{(\delta x,\delta f)\mid\delta f\geq0\}\),
\((\tdp_1)_{0,0,0}
=\{(\delta f,\delta\theta,\delta q)\mid
\delta f\leq0,\ \delta q\geq0,\
\delta\theta^2\leq-\delta f\,\delta q\}\),
and
\((\tdp_2)_{0,0}=\{(\delta q,\delta r)\mid\delta r\geq0\}\).

Before closure,
\((\tdp_2)_{0,0}\circ(\tdp_1)_{0,0,0}\)
equals
\(\{(\delta f,\delta\theta,\delta r)\mid
\delta f<0,\ \delta r\geq0\}
\cup
\{(0,0,\delta r)\mid\delta r\geq0\}\),
which is not closed: \((-1/k,1,0)\) is feasible with \(\delta q=k\) and
converges to \((0,1,0)\).
The closure therefore adds this limit point.

Consequently,
\(\bigl((\tdp_2\TSeries\tdp_1)\TSeries\tdp_0\bigr)_{0,0,0}
=
\{(\delta x,\delta\theta,\delta r)\mid\delta r\geq0\}\),
whereas composing the first two blocks first forces
\(\delta f=0\) and hence \(\delta\theta=0\), giving
\(\bigl(\tdp_2\TSeries(\tdp_1\TSeries\tdp_0)\bigr)_{0,0,0}
=
\{(\delta x,0,\delta r)\mid\delta r\geq0\}\).
Thus, tangential series composition need not be associative because of the intermediate closure.

\end{example}

Thus, different composition orders for a \gls{tcdp} may yield different results;
all subsequent results hold for any chosen composition order.
In the finite-constraint setting of
\cref{app:finite-constraint-design-problems}, however, whenever one
of the discussed \glspl{cq} holds,
each \gls{tdp} is described by finitely many linear
constraints and is therefore polyhedral.
For each fixed
internal realization, series and feedback compositions then reduce to imposing
linear interconnection constraints and projecting out internal variations,
which preserves polyhedrality and hence closedness. If the set of internal 
realizations is finite, the full composition is a finite
union of closed polyhedral sets and is therefore closed, so the closures in
\cref{def:composition-tangential-design-problems} are redundant and
series composition is associative.

\subsection{Closure under composition}
\label{subsec:tdp-composition-closure}
Composing \glspl{dp} yields a \gls{dp}
\cite{zardini2023codesigncomplexsystems}; we now prove
the analogous closure property for families of \glspl{tdp}.

\begin{theorem}[Composed \gls{tdp} families induce \glspl{tdp}]
\label{thm:tangential-design-problem-composition-closure}
Let
\(
C
\subseteq
\TDPFuncAmbientSpace
\times
\TDPThetaAmbientSpace
\times
\TDPResAmbientSpace
\)
be the result of composing \glspl{tdp} according to
\cref{def:composition-tangential-design-problems} at the
system-level operating point
\(
(\OpPoint f,\OpPoint\theta,\OpPoint r)
\in\Func\times\ThetaSpace\times\Res.
\)
Then there exists a \gls{dp}
\(\DesignProblem\subseteq\Func\times\ThetaSpace\times\Res\) such that
\(
C
=
\DesignProblem^T_{\OpPoint f,\OpPoint\theta,\OpPoint r}.
\)
Hence, the composition of \gls{tdp} families yields a \gls{tdp} at each
system-level operating point.
\end{theorem}

\begin{proof}
We first note that \(C\) is a closed cone and that
\(
C_{\delta\theta}
\in
\Upper(
\TDPFuncAmbientSpace^{\mathrm{op}}
\times
\TDPResAmbientSpace
)
\)
for every \(\delta\theta\in\TDPThetaAmbientSpace\).
This follows directly for parallel, intersection, and union.
For series and feedback, by
\eqref{eq:tdp-series-composition}
and~\eqref{eq:tdp-feedback-composition},
\(C\) is the closure of a union of sets
\(\DesignProblem_{\TCDPGraph_\rho}\),
each of which is a \gls{dp} and a cone, since scaling all
variations by any \(\lambda\geq 0\) preserves local feasibility
and the interconnection constraints.
Both union and closure preserve upperness and conicity,
while closedness follows from the outer closure.

We now construct a \gls{dp} whose \gls{tdp} is \(C\).
Define
\(
\DesignProblem
:=
\bigl(
(\OpPoint f,\OpPoint\theta,\OpPoint r)+C
\bigr)
\cap
(\Func\times\ThetaSpace\times\Res).
\)
Since
\(
C_{\delta\theta}
\in
\Upper(
\TDPFuncAmbientSpace^{\mathrm{op}}
\times
\TDPResAmbientSpace
)
\)
for every \(\delta\theta\in\TDPThetaAmbientSpace\),
\(
\DesignProblem_\theta
\in
\Upper(\Func^{\mathrm{op}}\times\Res)
\)
for every \(\theta\in\ThetaSpace\), and hence \(\DesignProblem\) is a \gls{dp}.

Since \(C\) is a closed cone,
\(
\Tan_{(\OpPoint f,\OpPoint\theta,\OpPoint r)+C}
(\OpPoint f,\OpPoint\theta,\OpPoint r)
=
C.
\)
Moreover,
\(
\DesignProblem
\subseteq
(\OpPoint f,\OpPoint\theta,\OpPoint r)+C
\)
and
\(
\DesignProblem
\subseteq
\Func\times\ThetaSpace\times\Res.
\)
Thus, the monotonicity of tangent cones under set inclusion,
established by Rockafellar and Wets~\cite{rockafellar1998variational}
in Proposition~6.2, gives
\(
\Tan_{\DesignProblem}(\OpPoint f,\OpPoint\theta,\OpPoint r)
\subseteq
C
\cap
\Tan_{\Func\times\ThetaSpace\times\Res}
(\OpPoint f,\OpPoint\theta,\OpPoint r).
\)
Conversely, for any
\(
(\delta f,\delta\theta,\delta r)
\in
C\cap
\Tan_{\Func\times\ThetaSpace\times\Res}
(\OpPoint f,\OpPoint\theta,\OpPoint r),
\)
since \(\Func\times\ThetaSpace\times\Res\) is a box,
\(
(\OpPoint f,\OpPoint\theta,\OpPoint r)
+\PerturbationScale(\delta f,\delta\theta,\delta r)
\in
\Func\times\ThetaSpace\times\Res
\)
for all sufficiently small \(\PerturbationScale>0\).
Moreover, since \(C\) is a cone,
\(
\PerturbationScale(\delta f,\delta\theta,\delta r)\in C
\).
Hence
\(
(\OpPoint f,\OpPoint\theta,\OpPoint r)
+\PerturbationScale(\delta f,\delta\theta,\delta r)
\in\DesignProblem
\),
so
\(
(\delta f,\delta\theta,\delta r)
\in
\Tan_{\DesignProblem}
(\OpPoint f,\OpPoint\theta,\OpPoint r)
\).
Therefore,
\(
\Tan_{\DesignProblem}
(\OpPoint f,\OpPoint\theta,\OpPoint r)
=
C
\cap
\Tan_{\Func\times\ThetaSpace\times\Res}
(\OpPoint f,\OpPoint\theta,\OpPoint r)
\).
We now show that
\(
\DesignProblem^T_{\OpPoint f,\OpPoint\theta,\OpPoint r}
=
\Tan_{\DesignProblem}
(\OpPoint f,\OpPoint\theta,\OpPoint r)
+K
=
C.
\)
The inclusion follows directly from
\(
\bigl(
C\cap
\Tan_{\Func\times\ThetaSpace\times\Res}
(\OpPoint f,\OpPoint\theta,\OpPoint r)
\bigr)+K
\subseteq C+K=C.
\)
For the reverse inclusion,
\cref{lem:ordering-cone-composition-commutation}
shows that every element of \(C\) can be written as the sum
of an element obtained by composing the corresponding local
tangent cones and an element of \(K\).
The former belongs to \(C\), and its external variations are
tangent to the system-level box. Therefore,
\(
C\subseteq
\bigl(
C\cap
\Tan_{\Func\times\ThetaSpace\times\Res}
(\OpPoint f,\OpPoint\theta,\OpPoint r)
\bigr)+K,
\)
and hence
\(
C=
\bigl(
C\cap
\Tan_{\Func\times\ThetaSpace\times\Res}
(\OpPoint f,\OpPoint\theta,\OpPoint r)
\bigr)+K=
\DesignProblem^T_{\OpPoint f,\OpPoint\theta,\OpPoint r}
\).
\end{proof}

\begin{corollary}[A \gls{tcdp} induces a family of \glspl{tdp}]
\label{cor:tdp-family-composition-closure}
Let \(\TCDPGraph\) be a \gls{tcdp} in the sense of
\cref{def:tangential-codesign-problem}.
Composing its \gls{tdp} families according to
\cref{def:composition-tangential-design-problems}
results in a family of \glspl{tdp}.
\end{corollary}

\begin{proof}
Apply \cref{thm:tangential-design-problem-composition-closure} at each
system-level operating point.
\end{proof}

%% file: sections/local_computation_of_system_level_sensitivity.tex
\section{Local computation of system-level sensitivity}
\label{sec:local-computation-system-level-sensitivity}

We study when composing local tangential models yields the tangential model
of the composed \gls{dp}.

The analysis below uses semidifferentiability
(\cref{def:semidifferentiability}), so we restrict the relevant
operating points and internal realizations to the interiors of their
corresponding boxes.
By \cref{def:tangential-design-problem}, this also lets us work
directly with the tangent cones, since the \(+K\) completion does not affect the
construction on the interior.

Importantly, \cref{lem:ordering-cone-composition-commutation} shows that
the \(+K\) completion commutes with tangential composition; see
\cref{app:tangential-composition-commutation}.
A detailed treatment of boundary effects is beyond the scope of this work,
but this result shows that the tangent-cone analysis can be separated from
the final ordering-cone completion when extending the analysis to boundary points.
Finally, we focus here on the convex case for series and feedback composition,
while more general conditions are developed in
\cref{app:beyond-convexity}.

We study the relation between these two constructions one composition
operation at a time.
By \cref{lem:parameter-output-regrouping}, it is sufficient to treat all
parameter variations as inputs throughout the analysis.

For the feedback composition, we use the following qualification.

\begin{assumption}[Feedback interior condition]
\label{ass:feedback-interior}
Let
\(
\DesignProblem
\subseteq
(\Func\times\IntermediateSpace)
\times\ThetaSpace
\times
(\Res\times\IntermediateSpace)
\)
be a \gls{dp}.
Assume that
\(
\exists\,((f,q_1),\theta,(r,q_2))
\in\operatorname{int}(\DesignProblem)
\)
such that \(q_1=q_2\).
\end{assumption}

\begin{remark}[Interpretation of \cref{ass:feedback-interior}]
\cref{ass:feedback-interior} is only an existence condition:
it requires some feasible design satisfying the interconnection constraint to
lie in \(\operatorname{int}(\DesignProblem)\).
In particular, such a design lies in the interior of the box and is not
Pareto-optimal.
The operating point itself is not required to satisfy these conditions.
\end{remark}

\begin{theorem}[Composition rules for system-level sensitivity]
\label{thm:composition-rules-system-sensitivity}

For each composition below, let
\(
(\OpPoint f,\OpPoint\theta,\OpPoint r)
\in
\operatorname{int}(\Func\times\ThetaSpace\times\Res)
\)
be an operating point of the composed \gls{dp}, and assume that the 
corresponding query maps are semidifferentiable at the
relevant graph points.
Then the following statements hold:

\begin{enumerate}
    \renewcommand{\labelenumi}{(\roman{enumi})}

\item \emph{Parallel:}
Let
\(
\DesignProblem_i
\subseteq
\Func_i\times\ThetaSpace_i\times\Res_i
\),
\(i=1,2\), be \glspl{dp}.
Then
\(
(\DesignProblem_1\otimes\DesignProblem_2)^T_
{(\OpPoint f,\OpPoint\theta,\OpPoint r)}
=
\left(
\DesignProblem_1^T\TParallel\DesignProblem_2^T
\right)_
{(\OpPoint f,\OpPoint\theta,\OpPoint r)}
\).

\item \emph{Union:}
Let
\(
\DesignProblem_i
\subseteq
\Func\times\ThetaSpace\times\Res
\),
\(i=1,2\), be \glspl{dp}.
Then
\(
(\DesignProblem_1\vee\DesignProblem_2)^T_
{(\OpPoint f,\OpPoint\theta,\OpPoint r)}
=
\left(
\DesignProblem_1^T\TUnion\DesignProblem_2^T
\right)_
{(\OpPoint f,\OpPoint\theta,\OpPoint r)}
\)
if
\(
(\OpPoint f,\OpPoint\theta,\OpPoint r)
\in
\DesignProblem_1\cap\DesignProblem_2
\),
and
\(
(\DesignProblem_1\vee\DesignProblem_2)^T_
{(\OpPoint f,\OpPoint\theta,\OpPoint r)}
=
(\DesignProblem_i^T)_
{(\OpPoint f,\OpPoint\theta,\OpPoint r)}
\)
if
\(
(\OpPoint f,\OpPoint\theta,\OpPoint r)
\in
\DesignProblem_i\setminus\DesignProblem_j
\),
\(i\neq j\).

\item \emph{Intersection:}
Let
\(
\DesignProblem_i
\subseteq
\Func\times\ThetaSpace\times\Res
\),
\(i=1,2\), be \glspl{dp}.
Then
\(
(\DesignProblem_1\wedge\DesignProblem_2)^T_
{(\OpPoint f,\OpPoint\theta,\OpPoint r)}
=
\left(
\DesignProblem_1^T\TIntersection\DesignProblem_2^T
\right)_
{(\OpPoint f,\OpPoint\theta,\OpPoint r)}
\).

\item \emph{Series:}
Let
\(
\DesignProblem_1
\subseteq
\Func\times\ThetaSpace_1\times\IntermediateSpace
\)
and
\(
\DesignProblem_2
\subseteq
\IntermediateSpace\times\ThetaSpace_2\times\Res
\)
be \glspl{dp}.
Then
\(
\left(
\DesignProblem_2^T\TSeries\DesignProblem_1^T
\right)_
{(\OpPoint f,\OpPoint\theta,\OpPoint r)}
\subseteq
(\DesignProblem_2\circ\DesignProblem_1)^T_
{(\OpPoint f,\OpPoint\theta,\OpPoint r)}
\).
If \(\DesignProblem_1\) and \(\DesignProblem_2\) are convex sets, then
\(
\left(
\DesignProblem_2^T\TSeries\DesignProblem_1^T
\right)_
{(\OpPoint f,\OpPoint\theta,\OpPoint r)}
=
(\DesignProblem_2\circ\DesignProblem_1)^T_
{(\OpPoint f,\OpPoint\theta,\OpPoint r)}
\).

\item \emph{Feedback:}
Let
\(
\DesignProblem
\subseteq
(\Func\times\IntermediateSpace)
\times\ThetaSpace
\times
(\Res\times\IntermediateSpace)
\)
be a \gls{dp} and a convex set.
If \cref{ass:feedback-interior} holds, then
\(
(\Trace(\DesignProblem))^T_
{(\OpPoint f,\OpPoint\theta,\OpPoint r)}
=
\left(
\TTrace(\DesignProblem^T)
\right)_
{(\OpPoint f,\OpPoint\theta,\OpPoint r)}
\).

\end{enumerate}
\end{theorem}

We now prove these statements one composition operation at a time, recording
sharper assumptions where possible.
In the convex cases, the union over compatible internal realizations in
\eqref{eq:tdp-series-composition}
and~\eqref{eq:tdp-feedback-composition}
of \cref{def:composition-tangential-design-problems} is unnecessary.
The system-level tangent directions can be recovered from any compatible
internal realization.

The required regularity properties are not generally preserved by all 
composition operations. Therefore, the pointwise results above cannot 
yet be iterated to obtain a global commutation result for an arbitrary \gls{cdp}.

\noindent\emph{Parallel, item~\textnormal{(i)}:}
\label{subsec:parallel-composition}
By \cref{lem:ordering-cone-composition-commutation}, we may work directly with the
tangent cones \(\Tan_{\DesignProblem_1}\) and \(\Tan_{\DesignProblem_2}\).
Semidifferentiability implies protodifferentiability, so the tangent-cone
calculus summarized by Aubin and Frankowska~\cite{aubin2009setvalued} in
Table~4.2 applies.
In particular,
\(
\Tan_{\DesignProblem_1\otimes\DesignProblem_2}
=
\Tan_{\DesignProblem_1}\times\Tan_{\DesignProblem_2}
\).
Therefore, at the fixed operating point,
\(
(\DesignProblem_1\otimes\DesignProblem_2)^T_
{(\OpPoint f,\OpPoint\theta,\OpPoint r)}
=
\left(
\DesignProblem_1^T\TParallel\DesignProblem_2^T
\right)_
{(\OpPoint f,\OpPoint\theta,\OpPoint r)}
\).

\noindent\emph{Union, item~\textnormal{(ii)}:}
\label{subsec:union-composition}
No semidifferentiability is needed for the union composition.
By the tangent-cone calculus summarized by Aubin and
Frankowska~\cite{aubin2009setvalued} in Table~4.1,
\(
\Tan_{\DesignProblem_1\cup\DesignProblem_2}
=
\Tan_{\DesignProblem_1}\cup\Tan_{\DesignProblem_2}
\)
if
\(
(\OpPoint f,\OpPoint\theta,\OpPoint r)
\in\DesignProblem_1\cap\DesignProblem_2
\),
while
\(
\Tan_{\DesignProblem_1\cup\DesignProblem_2}
=
\Tan_{\DesignProblem_i}
\)
if
\(
(\OpPoint f,\OpPoint\theta,\OpPoint r)
\in\DesignProblem_i\setminus\DesignProblem_j
\),
\(i\neq j\).
Therefore, by \cref{lem:ordering-cone-composition-commutation}, at the fixed operating point,
\(
(\DesignProblem_1\vee\DesignProblem_2)^T_
{(\OpPoint f,\OpPoint\theta,\OpPoint r)}
=
\left(
\DesignProblem_1^T\TUnion\DesignProblem_2^T
\right)_
{(\OpPoint f,\OpPoint\theta,\OpPoint r)}
\)
if
\(
(\OpPoint f,\OpPoint\theta,\OpPoint r)
\in
\DesignProblem_1\cap\DesignProblem_2
\),
while
\(
(\DesignProblem_1\vee\DesignProblem_2)^T_
{(\OpPoint f,\OpPoint\theta,\OpPoint r)}
=
(\DesignProblem_i^T)_
{(\OpPoint f,\OpPoint\theta,\OpPoint r)}
\)
if
\(
(\OpPoint f,\OpPoint\theta,\OpPoint r)
\in
\DesignProblem_i\setminus\DesignProblem_j
\),
\(i\neq j\).

\noindent\emph{Intersection, item~\textnormal{(iii)}:}
\label{subsec:intersection-composition}
Let \(\GenericQuery_i\), \(i=1,2\), be the queries of
\(\DesignProblem_i\).
By the tangent-cone calculus summarized by Aubin and
Frankowska~\cite{aubin2009setvalued} in Table~4.1,
\(
\Tan_{\DesignProblem_1\cap\DesignProblem_2}
\subseteq
\Tan_{\DesignProblem_1}\cap\Tan_{\DesignProblem_2}
\).

For the reverse inclusion, fix
\(
(\delta x,\delta\theta)
\in
\QueryInputAmbientSpace\times\QueryParameterAmbientSpace
\)
and take
\(
\delta y\in
D\GenericQuery_1((\OpPoint x,\OpPoint\theta),\OpPoint y)
(\delta x,\delta\theta)
\cap
D\GenericQuery_2((\OpPoint x,\OpPoint\theta),\OpPoint y)
(\delta x,\delta\theta)
\).
For any \(\PerturbationScale_k\downarrow0\) and
\(
(\delta x_k,\delta\theta_k)\to(\delta x,\delta\theta)
\),
semidifferentiability gives
\(
\delta y_{i,k}\to\delta y
\),
\(i=1,2\), such that
\(
\OpPoint y+\PerturbationScale_k\delta y_{i,k}
\in
\GenericQuery_i(
\OpPoint x+\PerturbationScale_k\delta x_k,
\OpPoint\theta+\PerturbationScale_k\delta\theta_k)
\)
for all sufficiently large \(k\).
Set
\(
\delta y_k:=\max\{\delta y_{1,k},\delta y_{2,k}\}
\)
componentwise with respect to the order on \(\QueryOutputAmbientSpace\).
Then \(\delta y_k\to\delta y\), and upperness gives
\(
\OpPoint y+\PerturbationScale_k\delta y_k
\in
\GenericQuery_1(
\OpPoint x+\PerturbationScale_k\delta x_k,
\OpPoint\theta+\PerturbationScale_k\delta\theta_k)
\cap
\GenericQuery_2(
\OpPoint x+\PerturbationScale_k\delta x_k,
\OpPoint\theta+\PerturbationScale_k\delta\theta_k).
\)
Hence
\(
\Tan_{\DesignProblem_1\cap\DesignProblem_2}
=
\Tan_{\DesignProblem_1}\cap\Tan_{\DesignProblem_2}
\).

Therefore, by
\cref{lem:ordering-cone-composition-commutation}, at the fixed operating
point,
\(
(\DesignProblem_1\wedge\DesignProblem_2)^T_
{(\OpPoint f,\OpPoint\theta,\OpPoint r)}
=
\left(
\DesignProblem_1^T\TIntersection\DesignProblem_2^T
\right)_
{(\OpPoint f,\OpPoint\theta,\OpPoint r)}
\).
The argument above treats all parameter variations as inputs; by
\cref{lem:parameter-output-regrouping}, the result also holds when
parameter variations are included in the output.

\noindent\emph{Series composition, item~\textnormal{(iv)}:}
\label{subsec:series-composition}
Let
\(
\GenericQuery_1:
\QueryInputAmbientSpace\times\QueryParameterAmbientSpace_1
\rightrightarrows
\QueryIntermediateAmbientSpace
\)
and
\(
\GenericQuery_2:
\QueryIntermediateAmbientSpace\times\QueryParameterAmbientSpace_2
\rightrightarrows
\QueryOutputAmbientSpace
\)
be the queries of \(\DesignProblem_1\) and \(\DesignProblem_2\), respectively.
Neither inclusion holds in general without additional regularity assumptions;
see, e.g., the discussion by Aubin and
Frankowska~\cite{aubin2009setvalued} in Section~5.3.

We first prove the forward inclusion
\(
\left(
\DesignProblem_2^T\TSeries\DesignProblem_1^T
\right)_
{(\OpPoint f,\OpPoint\theta,\OpPoint r)}
\subseteq
(\DesignProblem_2\circ\DesignProblem_1)^T_
{(\OpPoint f,\OpPoint\theta,\OpPoint r)}
\).
\(
\left(
\DesignProblem_2^T\TSeries\DesignProblem_1^T
\right)_
{(\OpPoint f,\OpPoint\theta,\OpPoint r)}
\)
is defined through a closure; see
\cref{def:composition-tangential-design-problems},
equation~\eqref{eq:tdp-series-composition}.
By \cref{lem:tdp-tangent-ordering-cone-extension},
\((\DesignProblem_2\circ\DesignProblem_1)^T_
{\OpPoint f,\OpPoint\theta,\OpPoint r}\)
is closed, so it suffices to prove the inclusion before this closure.

Take
\(
(\delta x,(\delta\theta_1,\delta\theta_2),\delta y)
\)
in the set whose closure defines
\(
\left(
\DesignProblem_2^T\TSeries\DesignProblem_1^T
\right)_
{(\OpPoint f,\OpPoint\theta,\OpPoint r)}
\)
in equation~\eqref{eq:tdp-series-composition}.
Then there exist
\(
\OpPoint q,\delta q\in\QueryIntermediateAmbientSpace
\)
such that
\(
\OpPoint q\in\GenericQuery_1(\OpPoint x,\OpPoint\theta_1)
\),
\(
\OpPoint y\in\GenericQuery_2(\OpPoint q,\OpPoint\theta_2)
\),
and
\(
(\delta x,\delta\theta_1,\delta q)
\in
(\DesignProblem_1^T)_
{(\OpPoint x,\OpPoint\theta_1,\OpPoint q)}
\),
\(
(\delta q,\delta\theta_2,\delta y)
\in
(\DesignProblem_2^T)_
{(\OpPoint q,\OpPoint\theta_2,\OpPoint y)}
\).
 
Semidifferentiability of \(\GenericQuery_2\) implies, by
\cref{def:semidifferentiability},
that \(\OpPoint q\in\operatorname{int}(\IntermediateSpace)\).
Together with
\(
(\OpPoint f,\OpPoint\theta,\OpPoint r)
\in
\operatorname{int}(\Func\times\ThetaSpace\times\Res)
\),
\cref{def:tangential-design-problem} then gives
\(
(\delta x,\delta\theta_1,\delta q)
\in
\Tan_{\DesignProblem_1}
(\OpPoint x,\OpPoint\theta_1,\OpPoint q)
\)
and
\(
(\delta q,\delta\theta_2,\delta y)
\in
\Tan_{\DesignProblem_2}
(\OpPoint q,\OpPoint\theta_2,\OpPoint y)
\).
The first membership implies that there exist
\(\PerturbationScale_k\downarrow0\) and
\(
(\delta x_k,\delta\theta_{1,k},\delta q_k)
\to
(\delta x,\delta\theta_1,\delta q)
\)
such that
\(
(
\OpPoint x+\PerturbationScale_k\delta x_k,
\OpPoint\theta_1+\PerturbationScale_k\delta\theta_{1,k},
\OpPoint q+\PerturbationScale_k\delta q_k
)
\in
\DesignProblem_1
\), while the second gives
\(
\delta y\in
D\GenericQuery_2
((\OpPoint q,\OpPoint\theta_2),\OpPoint y)
(\delta q,\delta\theta_2)
\).
By semidifferentiability of \(\GenericQuery_2\), applied to
\(\PerturbationScale_k\downarrow0\) and
\(
(\delta q_k,\delta\theta_2)\to(\delta q,\delta\theta_2)
\),
there exist \(\delta y_k\to\delta y\) such that
\(
\OpPoint y+\PerturbationScale_k\delta y_k
\in
\GenericQuery_2(
\OpPoint q+\PerturbationScale_k\delta q_k,
\OpPoint\theta_2+\PerturbationScale_k\delta\theta_2)
\)
for all sufficiently large \(k\).
Hence
\(
(\delta x,(\delta\theta_1,\delta\theta_2),\delta y)
\in
\Tan_{\DesignProblem_2\circ\DesignProblem_1}
(\OpPoint x,\OpPoint\theta,\OpPoint y)
\subseteq
(\DesignProblem_2\circ\DesignProblem_1)^T_
{(\OpPoint f,\OpPoint\theta,\OpPoint r)}
\).
Since
\(
(\DesignProblem_2\circ\DesignProblem_1)^T_
{(\OpPoint f,\OpPoint\theta,\OpPoint r)}
\)
is closed, this proves \(
\left(
\DesignProblem_2^T\TSeries\DesignProblem_1^T
\right)_
{(\OpPoint f,\OpPoint\theta,\OpPoint r)}
\subseteq
(\DesignProblem_2\circ\DesignProblem_1)^T_
{(\OpPoint f,\OpPoint\theta,\OpPoint r)}
\).

Assume now that \(\DesignProblem_1\) and \(\DesignProblem_2\) are convex sets.
Then \(\GenericQuery_1\) and \(\GenericQuery_2\) are pseudo-convex at the
corresponding graph points, so Aubin and
Frankowska~\cite{aubin2009setvalued} show in Proposition~5.3.4 that, for every
compatible internal realization \(\OpPoint q\),
\(
D\!\left(
\GenericQuery_2
\circ
(\GenericQuery_1\times\operatorname{Id}_{\QueryParameterAmbientSpace_2})
\right)
\subseteq
\Closure{
D\GenericQuery_2
\circ
\left(
D\GenericQuery_1
\times
\operatorname{Id}_{\QueryParameterAmbientSpace_2}
\right)
}
\).
This implies, at the fixed operating point,
\(
(\DesignProblem_2\circ\DesignProblem_1)^T_
{(\OpPoint f,\OpPoint\theta,\OpPoint r)}
\subseteq
\left(
\DesignProblem_2^T\TSeries\DesignProblem_1^T
\right)_
{(\OpPoint f,\OpPoint\theta,\OpPoint r)}
\).
Together with the previous inclusion, this proves equality.

\noindent\emph{Feedback composition, item~\textnormal{(v)}:}
\label{subsec:feedback-composition}
By \cref{def:composition-operations},
\(\Trace(\DesignProblem)\) is obtained by an intersection and a projection.
Let
\(
M
:=
\{((f,q_1),\theta,(r,q_2))
\mid q_1=q_2\}
\)
and
\(
\pi:
(\FuncAmbientSpace\times\QueryIntermediateAmbientSpace)
\times\ThetaAmbientSpace
\times(\ResAmbientSpace\times\QueryIntermediateAmbientSpace)
\to
\FuncAmbientSpace\times\ThetaAmbientSpace\times\ResAmbientSpace
\)
be defined by
\(
\pi((f,q_1),\theta,(r,q_2))=(f,\theta,r)
\).
Then
\(
\Trace(\DesignProblem)=\pi(\DesignProblem\cap M)
\).
\cref{ass:feedback-interior} gives
\(
0\in\operatorname{int}(\DesignProblem-M)
\),
so, for every compatible \(\OpPoint q\),
\(
\Tan_{\DesignProblem\cap M}
((\OpPoint f,\OpPoint q),\OpPoint\theta,(\OpPoint r,\OpPoint q))
=
\Tan_{\DesignProblem}
((\OpPoint f,\OpPoint q),\OpPoint\theta,(\OpPoint r,\OpPoint q))
\cap
\Tan_M
((\OpPoint f,\OpPoint q),\OpPoint\theta,(\OpPoint r,\OpPoint q))
\),
as summarized by Aubin and Frankowska~\cite{aubin2009setvalued} in
Table~4.4, item~(5b).
Since \(\DesignProblem\cap M\) is convex and \(\pi\) is linear,
\(
\Tan_{\pi(\DesignProblem\cap M)}
(\OpPoint f,\OpPoint\theta,\OpPoint r)
=
\Closure{
\pi\!\left(
\Tan_{\DesignProblem\cap M}
((\OpPoint f,\OpPoint q),\OpPoint\theta,(\OpPoint r,\OpPoint q))
\right)
}
\).
Rockafellar and Wets~\cite{rockafellar1998variational} establish this equality
in Theorem~6.43.
By \cref{lem:ordering-cone-composition-commutation}, this tangent-cone
equality is preserved by the ordering-cone completion defining the corresponding
\glspl{tdp}.
The closure in \cref{def:composition-tangential-design-problems},
equation~\eqref{eq:tdp-feedback-composition}, therefore gives
\(
(\Trace(\DesignProblem))^T_
{(\OpPoint f,\OpPoint\theta,\OpPoint r)}
=
\left(
\TTrace(\DesignProblem^T)
\right)_
{(\OpPoint f,\OpPoint\theta,\OpPoint r)}
\).

%% file: sections/finite_variations.tex
\section{Finite variations}
\label{sec:finite-variations}

The tangential model describes perturbations
\(x+\PerturbationScale\delta x\) for small \(\PerturbationScale>0\).
We instead consider finite variations \(x+\Delta x\) that remain within the
original functionality, parameter, and resource boxes.
For convex \glspl{dp}, we show that this restriction yields a tighter
outer approximation of the exact finite variations than that provided by the original \gls{tdp}.

Fix an operating point
\(
(\OpPoint f,\OpPoint\theta,\OpPoint r)\in\DesignProblem
\)
and define
\(
V(\OpPoint f,\OpPoint\theta,\OpPoint r)
:=
\FiniteVariationFuncSpace{\OpPoint f}
\times
\FiniteVariationThetaSpace{\OpPoint\theta}
\times
\FiniteVariationResSpace{\OpPoint r}.
\)
The exact finite variations and the bounded tangential model are, respectively,
\(
\DesignProblem^\Delta_{\OpPoint f,\OpPoint\theta,\OpPoint r}
:=
(\DesignProblem-(\OpPoint f,\OpPoint\theta,\OpPoint r))
\cap
V(\OpPoint f,\OpPoint\theta,\OpPoint r),
\
{\tdp}^{\Delta}_{\OpPoint f,\OpPoint\theta,\OpPoint r}
:=
\tdp_{\OpPoint f,\OpPoint\theta,\OpPoint r}
\cap
V(\OpPoint f,\OpPoint\theta,\OpPoint r).
\)
Translation preserves upperness, so the first set is a \gls{dp} on the
translated boxes. The second one is also a \gls{dp}, including when the
operating point lies on the boundary. Indeed, every element of a translated box is
tangent to the original box at the operating point; hence
\cref{lem:relative-upperness-tangent} gives
\begin{equation}
\label{eq:bounded-tdp-raw-tangent}
    {\tdp}^{\Delta}_{\OpPoint f,\OpPoint\theta,\OpPoint r}
    =
    \Tan_{\DesignProblem}(\OpPoint f,\OpPoint\theta,\OpPoint r)\cap V(\OpPoint f,\OpPoint\theta,\OpPoint r).
\end{equation}
Thus, the ordering cone in
\cref{def:tangential-design-problem} adds nothing after restriction
to \(V\).

\begin{theorem}[Finite-change outer approximation]
\label{thm:finite-variation-outer-approximation}
Let \(\DesignProblem\) be a \gls{dp} and fix
\((\OpPoint f,\OpPoint\theta,\OpPoint r)\in\DesignProblem\).
If \(\DesignProblem\) is convex, then
\(
    \DesignProblem^\Delta_{\OpPoint f,\OpPoint\theta,\OpPoint r}
    \subseteq
    {\tdp}^{\Delta}_{\OpPoint f,\OpPoint\theta,\OpPoint r}
    \subseteq
    \tdp_{\OpPoint f,\OpPoint\theta,\OpPoint r}.
\)
\end{theorem}

\begin{proof}
Rockafellar and Wets~\cite{rockafellar1998variational} show in
Theorem~6.9 that, for a convex set, every displacement from one of its points
belongs to its tangent cone, so
\(\DesignProblem-(\OpPoint f,\OpPoint\theta,\OpPoint r)
\subseteq
\Tan_{\DesignProblem}(\OpPoint f,\OpPoint\theta,\OpPoint r)\).
Equation~\eqref{eq:bounded-tdp-raw-tangent} gives the first inclusion, and
the definition of the restriction gives the second.
\end{proof}

As for the family \(\tdp\) in
\cref{def:tangential-design-problem}, we denote by
\({\tdp}^{\Delta}\) the family of bounded \glspl{tdp}
\({\tdp}^{\Delta}_{\OpPoint f,\OpPoint\theta,\OpPoint r}\).
We compose them according to
\cref{def:composition-tangential-design-problems}, requiring every
variation to remain in the corresponding translated box.

\begin{corollary}[Finite variations from one internal realization]
\label{cor:finite-variations-composition}

Let \(\CDPGraph\) be a \gls{cdp} containing no union compositions, fix
a system-level operating point
\(
(\OpPoint f,\OpPoint\theta,\OpPoint r)\in\DesignProblem_{\CDPGraph},
\)
and fix any single internal realization
\(
\InternalRealization
\in
\Real_{\CDPGraph}
\left(
\OpPoint f,\OpPoint r;
(\OpPoint\theta_v)_{v\in\Nodes}
\right).
\)
Let \(\CDPGraph^{T\Delta}_{\InternalRealization}\) denote the \gls{cdp}
\(\CDPGraph^T_{\InternalRealization}\) with each local \gls{tdp} restricted to its
translated box. If every \(\DesignProblem_v\) is convex, then
\(
(\DesignProblem_{\CDPGraph})^\Delta
\subseteq
\DesignProblem_{\CDPGraph^{T\Delta}_{\InternalRealization}}
\subseteq
\DesignProblem_{\CDPGraph^T_{\InternalRealization}}.
\)
\end{corollary}

\begin{proof}
Take
\(
(\Delta f,\Delta\theta,\Delta r)
\in
(\DesignProblem_{\CDPGraph})^\Delta
\).
Then
\(
(\OpPoint f+\Delta f,\OpPoint\theta+\Delta\theta,
 \OpPoint r+\Delta r)
\in\DesignProblem_{\CDPGraph}
\),
so choose an internal realization of this perturbed point. Together with the
fixed \(\InternalRealization\), this defines, at every \(v\in\Nodes\), a finite
displacement from
\(
(\OpPoint f_v,\OpPoint\theta_v,\OpPoint r_v)
\)
to another point of \(\DesignProblem_v\). By 
\cref{thm:finite-variation-outer-approximation},
\(
\DesignProblem_v^\Delta
\subseteq
\DesignProblem_v^{T\Delta}
\subseteq
\DesignProblem_v^T
\)
for every \(v\in\Nodes\). Since every co-design composition in
\cref{def:composition-operations} is monotone under set inclusion,
both inclusions are preserved at the system level.
\end{proof}

The exclusion of union compositions in
\cref{cor:finite-variations-composition} is essential: a finite
change may move from one branch of a union to another.

\begin{example}
\label{ex:finite-variation-union}
Let
\(
\Func=\ThetaSpace=\mathbb R^0
\), whose unique element we denote by \(0\),
and
\(
\Res=[0,3]^2
\).
Consider the convex \glspl{dp}
\(
\DesignProblem_1
=
\{(0,0,r)\mid r_2\geq 2\},
\qquad
\DesignProblem_2
=
\{(0,0,r)\mid r_1\geq 2\}.
\)
At
\(
(\OpPoint f,\OpPoint\theta,\OpPoint r)
=
(0,0,(1,2))
\),
only \(\DesignProblem_1\) is feasible. Hence
\(
(\DesignProblem_1^{T\Delta}\TUnion\DesignProblem_2^{T\Delta})
_{\OpPoint f,\OpPoint\theta,\OpPoint r}
=
\{0\}\times\{0\}\times([-1,2]\times[0,1]).
\)
However,
\(
(0,0,(1,-1))
\)
is a feasible finite variation of
\(\DesignProblem_1\vee\DesignProblem_2\), since
\(
(0,0,(1,2))+(0,0,(1,-1))=(0,0,(2,1))\in\DesignProblem_2
\).
Therefore
\(
(0,0,(1,-1))
\in
(\DesignProblem_1\vee\DesignProblem_2)^\Delta_
{\OpPoint f,\OpPoint\theta,\OpPoint r}
\quad\text{but}\quad
(0,0,(1,-1))
\notin
(\DesignProblem_1^{T\Delta}\TUnion\DesignProblem_2^{T\Delta})_
{\OpPoint f,\OpPoint\theta,\OpPoint r},
\)
so the first inclusion in
\cref{cor:finite-variations-composition} fails.
\end{example}

\begin{remark} 
\label{rem:finite-variation-query-intersection} 
Let \(\CDPGraph\) be a \gls{cdp} composed of convex \glspl{dp} and containing 
no union compositions. Suppose we have already computed 
\(
\GenericQuery_{\CDPGraph}(\OpPoint x,\OpPoint\theta)
\)
for some
\(
(\OpPoint x,\OpPoint\theta)
\),
and now want to solve the query at
\(
(\OpPoint x+\Delta x,\OpPoint\theta+\Delta\theta)
\).
By \cref{cor:finite-variations-composition}, the finite-variation
approximations associated with all solutions of the previous query and all 
their internal realizations can be intersected to obtain 
\[
\GenericQuery_{\CDPGraph}
\left(\OpPoint x+\Delta x,\OpPoint\theta+\Delta\theta\right)
\subseteq 
\bigcap_{\mathmakebox[4em][c]{\substack{ 
    \OpPoint y\in\GenericQuery_{\CDPGraph}(\OpPoint x,\OpPoint\theta)\\ 
    \InternalRealization\in 
    \Real_{\CDPGraph}\left( 
        \OpPoint x,\OpPoint y; 
        (\OpPoint\theta_v)_{v\in\Nodes} 
    \right) 
}}} 
\left[ 
    \OpPoint y+
    \GenericQuery_{\CDPGraph^{T\Delta}_{\InternalRealization}} 
    (\Delta x,\Delta\theta) 
\right]. 
\]
Thus, their intersection yields a tighter outer approximation.
\end{remark}

%% file: sections/numerical_example.tex
\section{Numerical examples}
\label{sec:numerical-examples}

We use three representative \glspl{cdp} to illustrate the main theoretical results
and assess the computational value of the proposed sensitivity construction.
All \glspl{dp} considered in these examples are represented by finitely many
constraints; see \cref{app:finite-constraint-design-problems} for a
detailed discussion of this setting.

\subsection{Rover: commutation beyond convexity}
\label{subsec:rover-example}

\begin{figure}[t]
    \centering
    \includegraphics[
        width=\linewidth
    ]{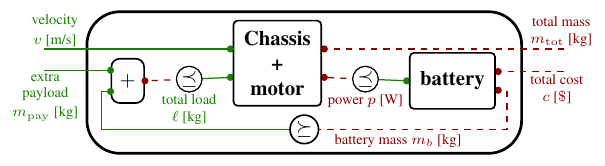}
    \caption{Rover co-design architecture.}
    \label{fig:rover-codesign-architecture}
\end{figure}

\begin{figure}[b]
    \centering
    \includegraphics[
        width=\linewidth
    ]{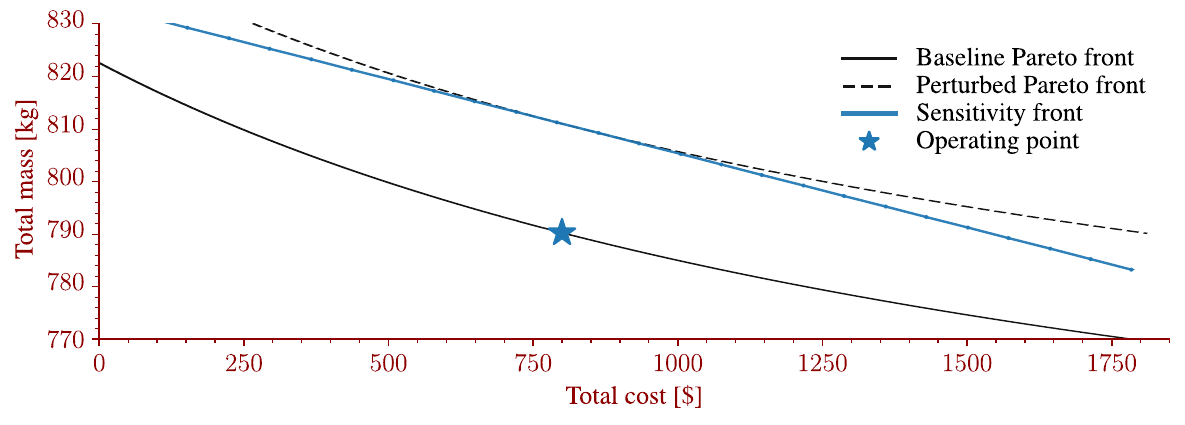}
    \caption{
        Rover sensitivity analysis. The baseline Pareto front is compared with
        the perturbed Pareto front and with the sensitivity front
        computed at the selected operating point.
    }
    \label{fig:rover-sensitivity-frontiers}
\end{figure}

We first consider the rover \gls{cdp} introduced by Cai et
al.~\cite{cai2026scalable}, shown in
\cref{fig:rover-codesign-architecture}. It consists of a chassis--motor
block and a battery block coupled through power and mass.
Since the conditions of
\cref{thm:composition-rules-system-sensitivity} are only sufficient,
we use this example to illustrate the sharper commutation conditions of
\cref{app:beyond-convexity} for a payload variation and the resulting
cost--mass trade-off.
As in \cref{sec:problem-setup}, all \glspl{dp} are defined over bounded
boxes, left implicit here and taken sufficiently large that the operating
points considered below lie in their interiors.

The chassis--motor block relates total load \(\ell\) and velocity \(v\) to
power \(p_{\mathrm{ch}}\) and total mass \(m_{\mathrm{tot}}\) through
\(p_{\mathrm{ch}}\geq p_0+\kappa(\ell+m_s)v\) and
\(m_{\mathrm{tot}}\geq\ell+m_s\), where \(m_s\) is the structural mass,
\(p_0\) the constant power required, and \(\kappa\) a coefficient.
The battery block is defined by
\(p_{\mathrm{bat}}\leq\alpha m_b+\beta c+\gamma m_b c\),
where \(p_{\mathrm{bat}}\) is the provided power, \(m_b\) and \(c\) are the
battery mass and cost, and \(\alpha\), \(\beta\), and \(\gamma\) are coefficients.

We first derive the system-level \gls{dp}. Series composition through power gives
\(p_0+\kappa(\ell+m_s)v
\leq \alpha m_b+\beta c+\gamma m_b c\)
and \(m_{\mathrm{tot}}\geq\ell+m_s\).
Closing the feedback loop with \(\ell=m_{\mathrm{pay}}+m_b\) gives
\(p_0+\kappa(m_{\mathrm{pay}}+m_s)v-\beta c
\leq(\alpha+\gamma c-\kappa v)m_b\)
and \(m_b\leq m_{\mathrm{tot}}-m_s-m_{\mathrm{pay}}\).
Assuming that
\(\alpha+\gamma c-\kappa v>0\), we may take
\(m_b=m_{\mathrm{tot}}-m_s-m_{\mathrm{pay}}\), yielding the system-level
\gls{dp}
\begin{equation}
    \label{eq:rover-system}
    \kappa m_{\mathrm{tot}}v
    +(\alpha+\gamma c)m_{\mathrm{pay}}
    \leq
    (\alpha+\gamma c)(m_{\mathrm{tot}}-m_s)
    +\beta c-p_0.
\end{equation}

At an operating point
\((\OpPoint{v},\OpPoint{m}_{\mathrm{pay}},
\OpPoint{c},\OpPoint{m}_{\mathrm{tot}})\) on the Pareto front,
\eqref{eq:rover-system} is active.
Since
\(\alpha+\gamma\OpPoint c-\kappa\OpPoint v>0\),
the \gls{mfcq} can be verified by increasing \(m_{\mathrm{tot}}\); see
\cref{app:finite-constraint-design-problems}.
Hence its \gls{tdp} is
\begin{equation}
    \label{eq:rover-system-sensitivity}
    \medmuskip=2mu
    \begin{aligned}
        &\kappa\OpPoint{m}_{\mathrm{tot}}\,\delta v
        +(\alpha+\gamma\OpPoint{c})\,\delta m_{\mathrm{pay}}
        \\
        &\leq
        \bigl(
            \beta+\gamma(
            \OpPoint{m}_{\mathrm{tot}}
            -m_s-\OpPoint{m}_{\mathrm{pay}})
        \bigr)\delta c
        +
        (\alpha+\gamma\OpPoint{c}-\kappa\OpPoint{v})
        \delta m_{\mathrm{tot}} .
    \end{aligned}
\end{equation}

For fixed \(\delta v\) and \(\delta m_{\mathrm{pay}}\),
\eqref{eq:rover-system-sensitivity} defines a linear Pareto front in
\((\delta c,\delta m_{\mathrm{tot}})\), providing a tangent approximation to
the perturbed nonlinear front.
\cref{fig:rover-sensitivity-frontiers} uses
\(m_s=770\), \(p_0=10\), \(\kappa=1\),
\(\alpha=0.975\), \(\beta=0.02733\), \(\gamma=0.0005\),
\((\OpPoint v,\OpPoint m_{\mathrm{pay}})=(0.05,0.1)\),
\((\OpPoint c,\OpPoint m_{\mathrm{tot}})\approx(799.86,790.21)\),
and \((\delta v,\delta m_{\mathrm{pay}})=(0,20)\).

This tangent half-space coincides with a tangent-plane cut used in the
outer-approximation construction of Cai et al.~\cite{cai2026scalable}.

We next recover the same system-level sensitivity compositionally from the
sensitivities of the two blocks. At the operating point above, let
\(\OpPoint m_b=\OpPoint m_{\mathrm{tot}}-m_s-\OpPoint m_{\mathrm{pay}}\),
\(\OpPoint\ell=\OpPoint m_{\mathrm{pay}}+\OpPoint m_b\), and
\(\OpPoint p=p_0+\kappa(\OpPoint\ell+m_s)\OpPoint v\).
Since \eqref{eq:rover-system} is active,
\(\OpPoint p
=
\alpha\OpPoint m_b+\beta\OpPoint c+\gamma\OpPoint m_b\OpPoint c\).

For the chassis--motor block, \gls{mfcq} can be verified by increasing
\(p_{\mathrm{ch}}\) and \(m_{\mathrm{tot}}\).
For the battery block, it can be verified by increasing \(m_b\) and \(c\).
Hence the queries of both blocks are semidifferentiable by
\cref{prop:query-regularity-mfcq}. The linearized chassis--motor constraints are
\(\kappa\OpPoint{v}\,\delta\ell
+\kappa(\OpPoint{\ell}+m_s)\,\delta v-\delta p\leq0\) and
\(\delta\ell-\delta m_{\mathrm{tot}}\leq0\), while the linearized battery
constraint is
\(\delta p-(\alpha+\gamma\OpPoint{c})\,\delta m_b
-(\beta+\gamma\OpPoint{m}_b)\,\delta c\leq0\).

For series, \cref{thm:composition-rules-system-sensitivity},
item~\textnormal{(iv)}, gives the forward inclusion, while
\cref{prop:series-composition-reverse-inclusion-nonconvex} gives the
reverse because
\(\delta\ell=\delta v=\delta m_{\mathrm{tot}}
=\delta m_b=\delta c=0\)
leaves \(-\delta p\leq0\) and \(\delta p\leq0\), forcing \(\delta p=0\).

For feedback, we rewrite the series-composed \gls{dp} \(\DesignProblem\) using
the internal functionality \(q\) and \(\ell=m_{\mathrm{pay}}+q\).
It is defined by
\(p_0+\kappa(m_{\mathrm{pay}}+q+m_s)v
\leq\alpha m_b+\beta c+\gamma m_b c\)
and
\(m_{\mathrm{tot}}\geq m_{\mathrm{pay}}+q+m_s\).
The feedback interconnection then identifies \(q=m_b\).
The compatible internal realization is unique because
\eqref{eq:rover-system} is active, with
\(\OpPoint q=\OpPoint m_b
=\OpPoint m_{\mathrm{tot}}-m_s-\OpPoint m_{\mathrm{pay}}\).
The equality \(q=m_b\) has a Jacobian of full row rank, and increasing \(c\) and
\(m_{\mathrm{tot}}\) strictly decreases the two active constraints, so the
intersection condition of
\cref{prop:feedback-composition-nonconvex} is satisfied. After imposing
\(q=m_b\), feasibility requires
\(q\geq
[p_0+\kappa(m_{\mathrm{pay}}+m_s)v-\beta c]/
(\alpha+\gamma c-\kappa v)\)
and \(q\leq m_{\mathrm{tot}}-m_s-m_{\mathrm{pay}}\).
Thus, \(q=m_{\mathrm{tot}}-m_s-m_{\mathrm{pay}}\) provides 
a feasible Lipschitz selection; hence, 
\cref{prop:feedback-composition-nonconvex} applies.

We have already obtained the direct system-level sensitivity in
\eqref{eq:rover-system-sensitivity}. 
Composing the linearized constraints eliminates \(\delta p,\delta\ell\), and
\(\delta m_b\); using \(\OpPoint\ell+m_s=\OpPoint m_{\mathrm{tot}}\) and
\(\OpPoint{m}_b
=\OpPoint{m}_{\mathrm{tot}}-m_s-\OpPoint{m}_{\mathrm{pay}}\)
gives exactly \eqref{eq:rover-system-sensitivity}.

\subsection{Machine: finite variations in the convex setting}
\label{subsec:machine-sensitivity-example}

\begin{figure}[t]
    \centering
    \includegraphics[width=0.9\columnwidth]
    {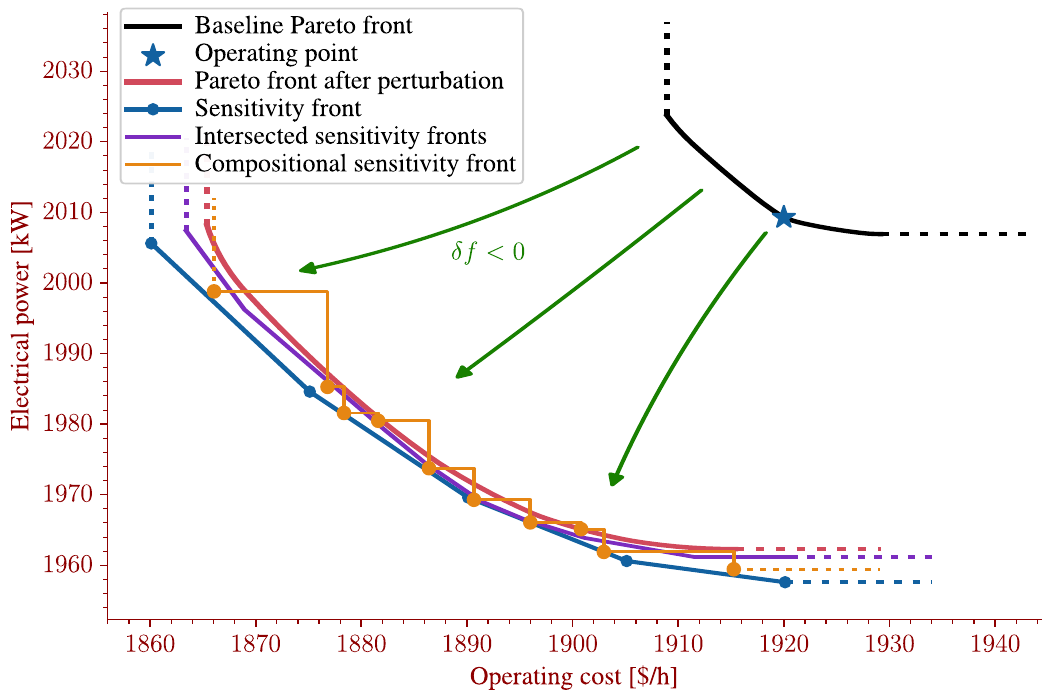}
    \caption{Sensitivity to a change in the production requirement.}
    \label{fig:machine-production-requirement-sensitivity}
\end{figure}

We now revisit the \gls{cdp} introduced in
\cref{sec:introduction} to illustrate the application of
\cref{sec:finite-variations}.

It consists of a production-allocation block and a machine-consumption block
connected in series through the normalized machine utilizations
\(x=(x_1,\ldots,x_6)\in[0,1]^6\).
We denote the resulting \gls{cdp} by \(\CDPGraph\).
The external functionality is the required production rate \(f\), the parameter
\(\eta\) describes the efficiency of the proprietary machines, and the external
resources are operating cost \(c\) and power consumption \(p\).

Machine \(i\) produces
\(q_i(x_i;\eta)=Q_i(\eta)(1-e^{-k_i x_i})\), where
\(Q_i(\eta)=Q_i^0(1+\gamma_i\eta)\), with \(\gamma_i=0\) for conventional
machines and \(\gamma_i>0\) for proprietary machines.
The resulting system-level feasible resource set is
\(\GenericQuery_{\CDPGraph}(f,\eta)
=
\{(c,p)\mid
\exists x\in[0,1]^6:
f\le\sum_{i=1}^{6}q_i(x_i;\eta),\
c\ge\sum_{i=1}^{6}a_i x_i,\
p\ge\sum_{i=1}^{6}(b_i x_i+d_i x_i^2)\}\).
Here, \(x\) is the resource of the production-allocation block and the
functionality of the machine-consumption block.
The parameter values are summarized in \cref{tab:machine-parameters},
rounded for readability.

\begin{table}[t]
    \centering
    \begin{tabular}{c|rrrrrr}
        \hline
        \(i\) & \(Q_i^0\) & \(k_i\) & \(\gamma_i\) & \(a_i\) & \(b_i\) & \(d_i\) \\
        \hline
        1 & 22.020 & 0.85 & 0    & 96.00  & 403.200 & 38.400 \\
        2 & 24.465 & 1.00 & 0    & 216.00 & 435.456 & 41.472 \\
        3 & 27.463 & 1.15 & 0    & 360.00 & 423.360 & 40.320 \\
        4 & 28.098 & 1.30 & 0.35 & 528.00 & 354.816 & 33.792 \\
        5 & 29.899 & 1.45 & 0.60 & 720.00 & 217.728 & 20.736 \\
        6 & 45.000 & 0.01 & 0    & 15.12  & 23.328  & 0.467 \\
        \hline
    \end{tabular}
    \caption{Parameters of the machine example.}
    \label{tab:machine-parameters}
\end{table}

We consider two experiments. In the first, we fix \(\eta\) and
vary \(f\) to compare the perturbed Pareto front with three sensitivity
approximations. In the second, we fix \(f\) and vary \(\eta\), using the local
sensitivity model to determine the efficiency improvement required to reach a
prescribed resource target. The problems are convex over the considered domain
for the fixed \(\eta\) used in the first experiment and the fixed \(f\) used in
the second, respectively, and both satisfy Slater's condition; we omit the
direct verification for brevity. All nonlinear problems are formulated in CVXPY and solved using Clarabel;
Pareto fronts are computed by sampling the cost axis and minimizing power at
each fixed cost. Since Slater's condition holds, we can obtain the tangent model
by linearizing the active constraints; see
\cref{app:finite-constraint-design-problems}.
The sensitivity Pareto front is obtained by fixing five
components of \(\Delta x\) at the bounds
\(-\OpPoint{x}_i\leq\Delta x_i\leq1-\OpPoint{x}_i\) and determining the sixth
from
\(\sum_{i=1}^{6}
Q_i(\OpPoint{\eta})k_i e^{-k_i\OpPoint{x}_i}\Delta x_i=\Delta f\).

\paragraph{Sensitivity prediction}

\cref{fig:machine-production-requirement-sensitivity} compares the
perturbed nonlinear solution with three sensitivity approximations: the
system-level surrogate
\(\GenericQuery_{\CDPGraph^{T\Delta}_{\InternalRealization}}\) from
\cref{cor:finite-variations-composition}, an approximation obtained by
propagating samples of the first local sensitivity
surrogate through the second, and an approximation obtained by intersecting
sensitivity surrogates constructed at multiple operating points, as described in
\cref{rem:finite-variation-query-intersection}.

We set
\(\OpPoint{x}=(1,1,1,1,1,0)\) and \(\OpPoint{\eta}=0.30\), yielding
\((\OpPoint{f},\OpPoint{\eta},\OpPoint{c},\OpPoint{p})
\approx(96.4,0.30,1920,2009)\).
We then perturb the production requirement by \(\Delta f=-0.8\), giving
\(\OpPoint{f}+\Delta f\approx95.6\).
The operating point is deliberately chosen away from the center of the Pareto
front, making the local nature of the approximation visible: the sensitivity
front is more accurate near the operating point and progressively less accurate
farther away.

Next, for each of the
\(\binom{6}{2}=15\) machine pairs \((i,j)\), we set \(\Delta x_k=0\) for
\(k\notin\{i,j\}\) and take the midpoint of the feasible segment
\(m_i\Delta x_i+m_j\Delta x_j=\Delta f\), with
\(m_i:=\frac{\partial q_i}{\partial x_i}
(\OpPoint{x}_i;\OpPoint{\eta})
=Q_i(\OpPoint{\eta})k_i e^{-k_i\OpPoint{x}_i}\),
subject to the corresponding box constraints. We propagate the resulting
samples \(\Delta x^{(k)}\) through the second tangential block via
\(\Delta c^{(k)}=a^\top\Delta x^{(k)}\) and
\(\Delta p^{(k)}=g^\top\Delta x^{(k)}\), where
\(g_i=b_i+2d_i\OpPoint{x}_i\),
and retain their minimal elements in \((\Delta c,\Delta p)\).
Finally, for the approximation of
\cref{rem:finite-variation-query-intersection}, we intersect sensitivity
surrogates constructed at five operating points selected uniformly in cost,
including both extremes.

\paragraph{Iterative redesign}

We next seek the minimum efficiency \(\eta\) required to make
\(r_{\mathrm{goal}}=(c_{\mathrm{goal}},p_{\mathrm{goal}})=(1880,1990)\)
feasible while keeping production fixed at \(\OpPoint{f}\).
We compare iterative redesign using the proposed sensitivity model with
bisection, a simple baseline applicable here because increasing \(\eta\)
strictly increases the production of every proprietary machine in use
(\(\gamma_i>0\)) without increasing its resource requirements.

At efficiency \(\eta_k\), the sensitivity-based method selects the Pareto point
\(r_k\) minimizing \(\|r_k-r_{\mathrm{goal}}\|_2\), computes
\(\Delta\eta
=
\min
\GenericFiniteTangentQuery_{\OpPoint{f},\eta_k,r_k}
(\Delta f=0,\Delta r=r_{\mathrm{goal}}-r_k)\), and sets
\(\eta_{k+1}=\eta_k+\Delta\eta\). The query enforces
\(\eta_k+\Delta\eta\leq\eta_{\max}\);
after recomputing the nonlinear front,
we proceed to the next iteration.

Bisection starts from
\([\eta_{\mathrm{lo}},\eta_{\mathrm{hi}}]=[0.30,0.35]\), which brackets the minimum
efficiency required to make the target feasible.
At each iteration, we solve the nonlinear \gls{cdp} at
\(\eta=(\eta_{\mathrm{lo}}+\eta_{\mathrm{hi}})/2\).
If the target is feasible, we set \(\eta_{\mathrm{hi}}=\eta\);
otherwise, we set \(\eta_{\mathrm{lo}}=\eta\).

\cref{fig:iterative-redesign-convergence} compares the two methods over
10 iterations. By
\cref{thm:finite-variation-outer-approximation}, the sensitivity front is
an outer approximation of the perturbed problem, so the predicted efficiency
variation does not exceed the variation required to reach the target.
The sensitivity-based method terminates after three iterations, as the remaining
efficiency variation becomes smaller than the numerical accuracy of the solver
and the computed tangent update is \(\Delta\eta=0\).

The distance in \cref{fig:iterative-redesign-convergence} is the Euclidean
distance from \(r_{\mathrm{goal}}\) to the Pareto front.
For the sensitivity-based method, \(r_{\mathrm{goal}}\) remains infeasible,
so we compute \(\min_r \|r-r_{\mathrm{goal}}\|_2\) subject to feasibility; in
all evaluations, the minimizer lies on the Pareto front.
For bisection, \(r_{\mathrm{goal}}\) may already be feasible, so its distance
to the feasible set can be zero even when its distance to the Pareto front is not.
We therefore estimate the latter using \(10^5\) samples of the nonlinear Pareto
front. Since the resulting values remain above the discretization scale, we
consider this approximation sufficiently accurate.

\begin{figure}[t]
    \centering
    \includegraphics[width=0.95\columnwidth]{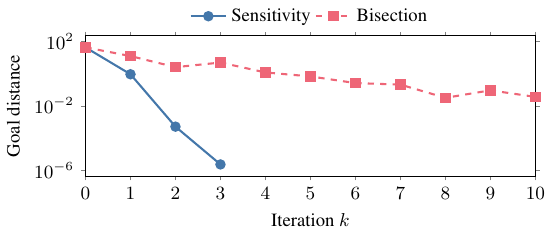}
    \caption{Distance to the goal as a function of the number of iterations.}
    \label{fig:iterative-redesign-convergence}
\end{figure}

\subsection{Drone: finite variations under nonconvex regime shifts}
\label{subsec:drone-example}

\pgfplotstableread[
    col sep=comma
]{data/timing_drone_fronts.csv}\dronetimings

\pgfplotstablegetelem{0}{time_seconds}\of{\dronetimings}
\edef\droneNonlinearTime{\pgfplotsretval}

\pgfplotstablegetelem{1}{time_seconds}\of{\dronetimings}
\edef\droneLocalTimeA{\pgfplotsretval}
\pgfplotstablegetelem{2}{time_seconds}\of{\dronetimings}
\edef\droneLocalTimeB{\pgfplotsretval}
\pgfplotstablegetelem{3}{time_seconds}\of{\dronetimings}
\edef\droneLocalTimeC{\pgfplotsretval}

\pgfmathparse{
    (\droneLocalTimeA+\droneLocalTimeB+\droneLocalTimeC)/3
}
\edef\droneLocalTime{\pgfmathresult}

\pgfplotstablegetelem{4}{time_seconds}\of{\dronetimings}
\edef\dronePointwiseTime{\pgfplotsretval}

\begin{figure}[t]
    \centering
    \includegraphics[
        width=\linewidth
    ]{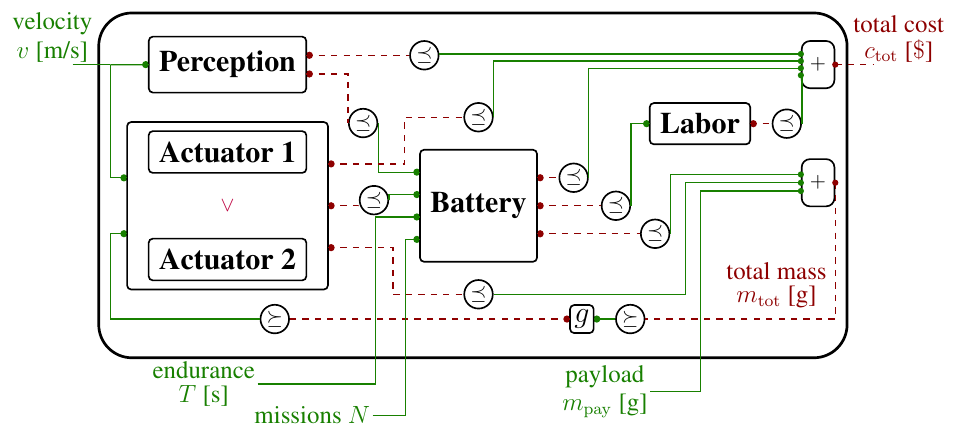}
    \caption{Co-design architecture of the drone example.}
    \label{fig:drone-codesign}
\end{figure}

\begin{table}[t]
    \centering
    \begin{tabular}{c|c|rr}
        \hline
        Parameter & Unit & \(i=1\) & \(i=2\) \\
        \hline
        \(m_{\mathrm{ref},i}\) & g & 100 & 100 \\
        \(p_{0,i}\) & W & 8 & 18 \\
        \(p_{1,i}\) & W/N & 2 & 1.5 \\
        \(p_{2,i}\) & W/N\(^{2}\) & 0.22 & 0.055 \\
        \(p_{3,i}\) & \(\mathrm{W}/[\mathrm{N}(\mathrm{m}/\mathrm{s})^2]\) & 1.5 & 0.5 \\
        \(c_{0,i}\) & USD & 30 & 45 \\
        \(c_{m,i}\) & USD/g & 0.25 & 0.25 \\
        \hline
    \end{tabular}
    \caption{Parameters of the two actuator technologies.}
    \label{tab:drone-actuator-parameters}
\end{table}

We finally consider the drone \gls{cdp} shown in
\cref{fig:drone-codesign}, based on the model of
Censi~\cite{censi2017uncertainty}, to explore bounded linearizations,
motivated by \cref{sec:finite-variations}, for a larger, nonconvex \gls{cdp}.
For some quantities, the variation boxes are left partly implicit here; their
unspecified bounds are taken sufficiently large that the operating points
considered below lie in their interiors.

The system consists of perception, actuation, battery, labor, and gravity
blocks, with the actuation block representing a choice between two
alternative actuator technologies.
The external functionalities are the payload mass \(m_{\mathrm{pay}}\),
endurance \(T\), velocity \(v\), and number of missions \(N\), while the
external resource is the total cost \(c_{\mathrm{tot}}\).

For actuator technology \(i\in\{1,2\}\), distinguish its actual mass
\(s_i\) from the resource allowances \((P_i,m_i,c_i)\).
The actuator \gls{dp} consists of the allowances for which there exists
\(s_i\in[40,220]\) g satisfying
\[
\begin{split}
m_i&\geq s_i,\qquad c_i\geq c_{0,i}+c_{m,i}s_i,\\
P_i s_i&\geq m_{\mathrm{ref},i}
(p_{0,i}+p_{1,i}L+p_{2,i}L^2+p_{3,i}Lv^2).
\end{split}
\]
Existentially eliminating \(s_i\) gives a relation upper in all three
resource allowances. The numerical resource boxes impose
\(40\leq m_i\leq220\) g, \(0\leq P_i\leq500\) W, and
\(0\leq c_i\leq250\) USD.
The two actuator \glspl{dp} are combined by union.
We denote the mass, power, and cost of the selected actuator by
\(m_{\mathrm{act}}\), \(P_{\mathrm{act}}\), and \(c_{\mathrm{act}}\),
respectively.
The actuator parameters are given in
\cref{tab:drone-actuator-parameters}.
The perception block requires power
\(P_{\mathrm{per}}\geq 5+2v^2\) W and cost
\(c_{\mathrm{per}}\geq 5+2.5v^2\) USD to provide velocity \(v\).
Writing \(E=T(P_{\mathrm{act}}+P_{\mathrm{per}})\) and
\(k(N)=\lceil N/750\rceil\), the battery requires mass
\(m_{\mathrm{bat}}\geq
E/702\) g and a replacement allowance \(n_{\mathrm{rep}}\geq k(N)\).
Its cost satisfies
\(c_{\mathrm{bat}}\geq k(N)
\max\{E/(2.84\cdot3600),12(E/10^5)^2\}\) USD.
Thus, cost depends on the required count, rather than its allowance.
The labor and gravity blocks are linear, with
\(c_{\mathrm{lab}}\geq 10n_{\mathrm{rep}}\) USD and
\(L\geq 9.81\,m_{\mathrm{tot}}/1000\) N, where
\(m_{\mathrm{tot}}\geq
m_{\mathrm{pay}}+m_{\mathrm{bat}}+m_{\mathrm{act}}\).
is the internal carried-mass requirement.

At the system level, the component costs are aggregated as
\(c_{\mathrm{tot}}\geq
c_{\mathrm{act}}+c_{\mathrm{bat}}+c_{\mathrm{per}}+c_{\mathrm{lab}}\).
Together with the interconnections above, this relation defines the
system-level feasible cost set
for each \((m_{\mathrm{pay}},T,v,N)\).

We consider two functionality specifications,
\((T,v,N)=(200,3,1000)\) and \((400,2.5,1000)\).
For each specification, we solve the nonlinear \gls{cdp} on a uniform grid
of \(100\) payload values \(m_{\mathrm{pay}}\in[100,2900]\) g using
\texttt{codesign-mcdp}~\cite{briat2026codesignmcdp}.
Since the solver represents continuous Pareto fronts through finite samples,
for each actuator technology, we use a uniform grid of \(100\) mass values
over \([40,220]\) g.
At each payload, the solver returns the minimum achievable
system-level cost $c_{\mathrm{tot}}$.

The mass grid approximates a continuous range of possible masses, while
the numerical sensitivity models linearize the active smooth constraints
of the corresponding continuous model for the selected technology.
At the chosen operating points, $\OpPoint{s}_i=\OpPoint{m}_i$, so the
mass-allowance constraint reduces to $\Delta s_i\leq\Delta m_i$.
Because reducing the mass allowance relaxes the downstream gravity and
total-mass constraints, setting $\Delta m_i=\Delta s_i$, as implemented
in the code, leaves the external projection unchanged.
The corresponding bounds on $m_i$ also preserve the physical bounds
on $s_i$. Throughout these experiments, $\Delta N=0$ and $k(N)=2$,
allowing us to eliminate replacement and labor variations.
The nonconvex example provides an empirical approximation, for which
the convex outer-approximation guarantee of
\cref{cor:finite-variations-composition} does not apply.

We first reconstruct the second payload--cost curve pointwise from all
\(100\) baseline solutions.
At each point, we fix \(\Delta m_{\mathrm{pay}}=0\),
\(\Delta T=200\) s, \(\Delta v=-0.5\) m/s, and \(\Delta N=0\),
and minimize \(\Delta c_{\mathrm{tot}}\) over this bounded implementation
linearization.
Each problem is therefore an \gls{lp}, which we solve using
SciPy's \texttt{linprog} with the HiGHS backend, yielding the prediction
\(c_{\mathrm{tot}}+\Delta c_{\mathrm{tot}}\).

To highlight the structure captured by the pointwise sensitivity predictions,
we filter out locally nonmonotone samples within a \(200\) g payload window,
retaining the lower-cost prediction.
\cref{fig:drone-pointwise-sensitivity-regime} compares the resulting
curves with the two nonlinear payload--cost curves.
The shaded regions indicate changes in actuator technology or in the active
battery-cost branch between the baseline and perturbed nonlinear solutions.
The resulting curve shows that the sensitivity construction preserves the
overall shape of the perturbed Pareto front despite an error in its position.

We then consider selected operating points along the first payload--cost curve
under the same perturbation.
At each point, we construct the same bounded linearization and compute
its full local Pareto front with Bensolve~\cite{lohne2017vectorlinearprogramsolver}.
\cref{fig:drone-sensitivity-fronts} shows the resulting fronts in the
original coordinates, restricted for visualization to a box around each
operating point.

The sensitivity computations are also substantially cheaper than recomputing
the nonlinear \gls{cdp}.
Computing one nonlinear payload--cost curve takes on average
\(\pgfmathprintnumber[fixed,precision=0]{\droneNonlinearTime}\) s, whereas
computing a local Pareto front takes
\(\pgfmathprintnumber[fixed,precision=3]{\droneLocalTime}\) s on average
and the \(100\) pointwise \glspl{lp} take only
\(\pgfmathprintnumber[fixed,precision=2]{\dronePointwiseTime}\) s in total.

\begin{figure}[t]
    \centering
    \includegraphics[
        width=\linewidth
        ]{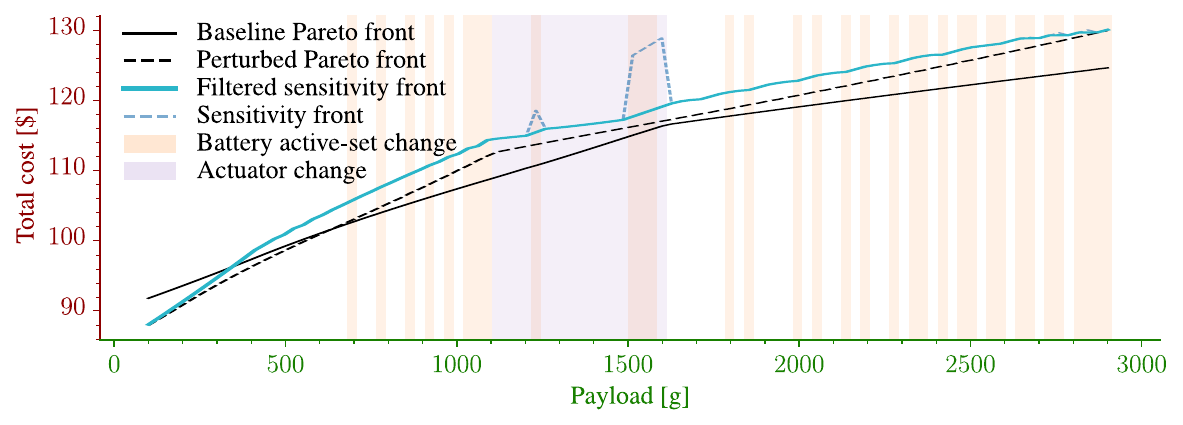}
        \caption{
        Pointwise sensitivity prediction of the perturbed payload--cost
        curve.
        Shaded regions indicate changes in the active battery-cost branch
        and actuator technology between the baseline and perturbed
        nonlinear solutions.
        }
    \label{fig:drone-pointwise-sensitivity-regime}
\end{figure}

\begin{figure}[t]
    \centering
    \includegraphics[
        width=\linewidth
    ]{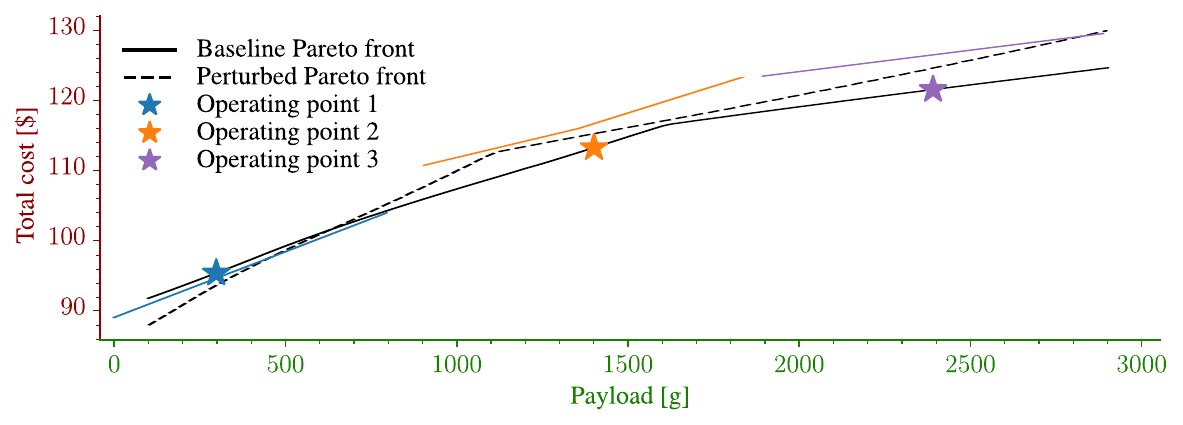}
    \caption{
        Local sensitivity fronts computed at selected operating points
        along the nominal payload--cost curve.
    }
    \label{fig:drone-sensitivity-fronts}
\end{figure}

%% file: sections/conclusion.tex
\section{Conclusion}
\label{sec:conclusion}
This paper develops a first-order sensitivity analysis for monotone co-design.
Viewing \glspl{dp} as sets of realizable functionality--resource pairs, we
describe their local behavior through tangent cones
and obtain \glspl{tdp}, which are \glspl{dp} over variations.
The resulting \glspl{tdp} can be queried within the same framework,
with the corresponding sensitivity queries characterized through
graphical derivatives.
The resulting sensitivity framework is compositional, i.e.,
system-level sensitivity can be modeled as a \gls{tcdp}
assembled from families of component \glspl{tdp} using the same
composition operations and interconnection structure as the original \gls{cdp}.
We establish conditions under which sensitivity analysis
commutes with individual co-design composition operations
and illustrate how these conditions can be verified
for an interconnected system using the rover example.
For convex \glspl{dp}, restricting variations to admissible boxes
yields, from a single internal realization, an outer approximation of the exact
finite variations that can be propagated through systems without union
compositions.

On the computational side, under the stated \glspl{cq},
finite-constraint \glspl{dp} yield polyhedral \glspl{tdp}, whose queries
can be solved using scalable methods for \glspl{ldp}~\cite{cai2026scalable}.
In the drone example, pointwise sensitivity predictions yield a speedup of
approximately four orders of magnitude over computing the nonlinear payload--cost curve.

This work opens several fertile directions for future research. To start, the current framework is limited to Euclidean spaces, but may be extended to the more general \glspl{poset} supported
by monotone co-design. 
Moreover, the machine example illustrates how sensitivity information can
guide iterative redesign, suggesting a potential application to
adaptive sampling in online learning for co-design~\cite{alharbi2026compositional}.
Finally, an open question with considerable potential for practical impact
is how sensitivity analysis can support
uncertainty propagation~\cite{huang2025composableparametricuncertaintysystems,
huang2026distributional}, for instance describing variations of probability distributions in the
Wasserstein space~\cite{lanzetti2024variational}.

%% file: sections/appendix_tangential_composition_reductions.tex
\section{Commutation results for tangential composition}
\label{app:tangential-composition-commutation}

This appendix collects auxiliary results for
\cref{sec:local-computation-system-level-sensitivity}.

\subsection{Ordering-cone completion}

This subsection shows that the ordering-cone completion used in the
definition of a \gls{tdp} commutes with tangential composition.

To start, we extend
\cref{thm:tangent-model-design-problem} to boundary operating points by
restricting to variations that keep the functionality and resource coordinates
inside the box to first order.

\begin{lemma}[Upperness of the tangent cone within the box]
\label{lem:relative-upperness-tangent}

Let
\(
\DesignProblem\subseteq\Func\times\ThetaSpace\times\Res
\)
be a \gls{dp}, and let
\(
(\OpPoint f,\OpPoint{\theta},\OpPoint r)\in\DesignProblem.
\)
If
\(
(\delta f,\delta\theta,\delta r)
\in
\Tan_{\DesignProblem}
(\OpPoint f,\OpPoint{\theta},\OpPoint r)
\),
and
\(
(\delta f',\delta\theta,\delta r')
\in
\Tan_{\Func\times\ThetaSpace\times\Res}
(\OpPoint f,\OpPoint{\theta},\OpPoint r)
\)
with
\(
\delta f'\preceq\delta f
\)
and
\(
\delta r\preceq\delta r',
\)
then
\(
(\delta f',\delta\theta,\delta r')
\in
\Tan_{\DesignProblem}
(\OpPoint f,\OpPoint{\theta},\OpPoint r).
\)
\end{lemma}

\begin{proof}
Take
\(
(\delta f_k,\delta\theta_k,\delta r_k)
\to
(\delta f,\delta\theta,\delta r)
\)
and
\(
\PerturbationScale_k\downarrow0
\)
such that
\(
(\OpPoint f+\PerturbationScale_k\delta f_k,
 \OpPoint{\theta}+\PerturbationScale_k\delta\theta_k,
 \OpPoint r+\PerturbationScale_k\delta r_k)
\in\DesignProblem.
\)
Set
\(
\delta f_k':=\min\{\delta f_k,\delta f'\}
\)
and
\(
\delta r_k':=\max\{\delta r_k,\delta r'\}
\),
componentwise. Then
\(
\delta f_k'\to\delta f'
\)
and
\(
\delta r_k'\to\delta r'.
\)
Since
\(
(\delta f',\delta\theta,\delta r')
\)
is tangent to the box, the straight perturbations
\(
\OpPoint f+\PerturbationScale_k\delta f'
\)
and
\(
\OpPoint r+\PerturbationScale_k\delta r'
\)
belong to \(\Func\) and \(\Res\) for all sufficiently large \(k\).
The original perturbed points also belong to the boxes; hence their
componentwise minimum and maximum do as well. Therefore
\(
\OpPoint f+\PerturbationScale_k\delta f_k'\in\Func
\)
and
\(
\OpPoint r+\PerturbationScale_k\delta r_k'\in\Res.
\)
By upperness of \(\DesignProblem\),
\(
(\OpPoint f+\PerturbationScale_k\delta f_k',
 \OpPoint{\theta}+\PerturbationScale_k\delta\theta_k,
 \OpPoint r+\PerturbationScale_k\delta r_k')
\in\DesignProblem.
\)
Hence
\(
(\delta f',\delta\theta,\delta r')
\in
\Tan_{\DesignProblem}
(\OpPoint f,\OpPoint{\theta},\OpPoint r).
\)
\end{proof}

\begin{lemma}[Commutation with ordering-cone completion]
\label{lem:ordering-cone-composition-commutation}
Let \(\TCDPGraph\) be a \gls{tcdp} as in
\cref{def:tangential-codesign-problem}.
Composition according to
\cref{def:composition-tangential-design-problems}
commutes with the ordering-cone completion of
\cref{def:tangential-design-problem}.
More precisely, if \(\tdp\) is obtained by composing the local \glspl{tdp} and
\(\mathcal T\) by applying the same composition operations to the corresponding
tangent cones, then
\(\tdp=\mathcal T+K_{\CDPGraph}\), where \(K_{\CDPGraph}\) is the ordering cone
of the system-level functionality--resource space.
\end{lemma}

\begin{proof}
It suffices to verify the claim for each composition operation.
Fix compatible operating points and denote the corresponding tangent cones by
\(T_i\). For series and feedback, let \(\IntermediateSpace\) denote the
internal-variable box.

\emph{Parallel:}
Since Cartesian products distribute over Minkowski sums,
\((T_1+K_1)\otimes(T_2+K_2)
=(T_1\otimes T_2)+K_{\otimes}\),
where \(K_{\otimes}=K_1\times K_2\), up to the canonical regrouping of
coordinates, is the ordering cone of the composed space.

\emph{Intersection:}
Since both tangent cones are defined over the same box, they have the same
ordering cone \(K\). The inclusion
\(
(T_1\cap T_2)+K
\subseteq
(T_1+K)\cap(T_2+K)
\)
is immediate. For the reverse inclusion, let
\(
(\delta f,\delta\theta,\delta r)
\in
(T_1+K)\cap(T_2+K).
\)
Then, for \(i=1,2\), there exist
\(
(\delta f_i,\delta\theta,\delta r_i)\in T_i
\)
such that
\(
\delta f\preceq\delta f_i
\)
and
\(
\delta r_i\preceq\delta r.
\)

Set, componentwise,
\(
\delta f_{\min}:=\min\{\delta f_1,\delta f_2\}
\)
and
\(
\delta r_{\max}:=\max\{\delta r_1,\delta r_2\}.
\)
Since both original variations are tangent to the common box,
\(
(\delta f_{\min},\delta\theta,\delta r_{\max})
\)
is also tangent to it. Hence, by
\cref{lem:relative-upperness-tangent},
\(
(\delta f_{\min},\delta\theta,\delta r_{\max})
\in T_1\cap T_2.
\)
Moreover,
\(
\delta f\preceq\delta f_{\min}
\)
and
\(
\delta r_{\max}\preceq\delta r,
\)
so
\(
(\delta f,\delta\theta,\delta r)
\in
(T_1\cap T_2)+K.
\)
Therefore
\(
(T_1+K)\wedge(T_2+K)
=
(T_1\wedge T_2)+K.
\)

\emph{Union:}
Since the Minkowski sum commutes with unions,
\((T_1+K)\vee(T_2+K)
=(T_1\vee T_2)+K\).
   
\emph{Series:}
Fix a compatible internal realization \(\OpPoint q\). We first show
\(
(T_2+K_2)\circ(T_1+K_1)
=
(T_2\circ T_1)+K_{\circ},
\)
where \(K_{\circ}\) is the ordering cone of the composed space.
The inclusion
\((T_2\circ T_1)+K_{\circ}
\subseteq (T_2+K_2)\circ(T_1+K_1)\)
is immediate.

For the reverse inclusion, let
\((\delta f,\delta\theta_1,\delta\theta_2,\delta r)
\in (T_2+K_2)\circ(T_1+K_1)\).
Then there exist an internal variation \(\delta q\) and tangent directions
\((\delta f_1,\delta\theta_1,\delta q_1)\in T_1\) and
\((\delta q_2,\delta\theta_2,\delta r_2)\in T_2\) such that
\(\delta f\preceq\delta f_1\),
\(\delta q_1\preceq\delta q\preceq\delta q_2\), and
\(\delta r_2\preceq\delta r\).

Since \(\delta q_1\in\Tan_{\IntermediateSpace}(\OpPoint q)\) and
\(\delta q_1\preceq\delta q_2\),
\cref{lem:relative-upperness-tangent} applied to \(T_2\) gives
\((\delta q_1,\delta\theta_2,\delta r_2)\in T_2\).
Hence
\((\delta f_1,\delta\theta_1,\delta\theta_2,\delta r_2)
\in T_2\circ T_1\), and therefore
\((\delta f,\delta\theta_1,\delta\theta_2,\delta r)
\in (T_2\circ T_1)+K_{\circ}\).

Taking the union over all compatible internal realizations and using that
Minkowski addition commutes with unions gives
\(
\bigcup_{\OpPoint q}
\bigl((T_2+K_2)\circ(T_1+K_1)\bigr)
=
\left(
\bigcup_{\OpPoint q} T_2\circ T_1
\right)+K_{\circ}.
\)
Taking the closures prescribed by
\cref{def:composition-tangential-design-problems}, it remains to show
\(\Closure{\bigcup_{\OpPoint q}T_2\circ T_1}+K_{\circ}
=
\Closure{(\bigcup_{\OpPoint q}T_2\circ T_1)+K_{\circ}}\).
The forward inclusion is immediate.
For the reverse inclusion, take a convergent sequence of elements of
\((\bigcup_{\OpPoint q}T_2\circ T_1)+K_{\circ}\) and clip each term
componentwise to the tangent cone of the composed box. Relative to the
corresponding element of \(\bigcup_{\OpPoint q}T_2\circ T_1\), this only
decreases functionality variations and increases resource variations.
Applying \cref{lem:relative-upperness-tangent} to the external
variations of \(T_1\) and \(T_2\) shows that the clipped sequence remains in the
same union. Its limit belongs to
\(\Closure{\bigcup_{\OpPoint q}T_2\circ T_1}\), while the difference between
the original and clipped limits belongs
to \(K_{\circ}\), giving the reverse inclusion.

\emph{Feedback:}
Fix a compatible internal realization \(\OpPoint q\), and let \(K\) and
\(K_{\Trace}\) denote the ordering cones before and after feedback,
respectively.
The inclusion \(\Trace(T)+K_{\Trace}\subseteq\Trace(T+K)\) is immediate.
For the reverse inclusion, let
\((\delta f,\delta\theta,\delta r)\in\Trace(T+K)\).
Then there exist \(\delta q\) and
\(\bigl((\delta f_0,\delta q_1),\delta\theta,(\delta r_0,\delta q_2)\bigr)\in T\)
such that \(\delta f\preceq\delta f_0\),
\(\delta q_2\preceq\delta q\preceq\delta q_1\), and
\(\delta r_0\preceq\delta r\).
Since \(\delta q_1,\delta q_2\in\Tan_{\IntermediateSpace}(\OpPoint q)\) and
\(\delta q_2\preceq\delta q_1\),
\cref{lem:relative-upperness-tangent} gives
\(\bigl((\delta f_0,\delta q_2),\delta\theta,(\delta r_0,\delta q_2)\bigr)\in T\).
Hence \((\delta f_0,\delta\theta,\delta r_0)\in\Trace(T)\), and therefore
\(\Trace(T+K)=\Trace(T)+K_{\Trace}\).
The union over compatible internal realizations
and the subsequent closure are handled as in the series case.

\end{proof}
\subsection{Splitting parameter variations}

\begin{lemma}[Splitting parameter variations]
\label{lem:parameter-output-regrouping}
Let
\(A,B\subseteq
\TDPQueryInputAmbientSpace
\times\TDPQueryParameterAmbientSpace
\times\TDPQueryOutputAmbientSpace\),
and let
\(\TDPQueryParameterAmbientSpace
=\TDPThetaInputAmbientSpace\times\TDPThetaOutputAmbientSpace\), with
\(\delta\theta=(\delta\theta_{\mathrm{in}},\delta\theta_{\mathrm{out}})\).
For \(C\in\{A,B\}\), define the query in which
all parameter variations are inputs by
\(C^{\mathrm{unsplit}}(\delta x,\delta\theta)
:=
\{\delta y\mid(\delta x,\delta\theta,\delta y)\in C\}\)
and the query in which the parameter variations are split between inputs and outputs by
\(C^{\mathrm{split}}(\delta x,\delta\theta_{\mathrm{in}})
:=
\{(\delta y,\delta\theta_{\mathrm{out}})
\mid(\delta x,\delta\theta_{\mathrm{in}},
\delta\theta_{\mathrm{out}},\delta y)\in C\}\).

Then
\(A^{\mathrm{unsplit}}\subseteq B^{\mathrm{unsplit}}
\Leftrightarrow
A^{\mathrm{split}}\subseteq B^{\mathrm{split}}\),
and
\(A^{\mathrm{unsplit}}=B^{\mathrm{unsplit}}
\Leftrightarrow
A^{\mathrm{split}}=B^{\mathrm{split}}\).
Here, inclusion between set-valued maps is understood pointwise.
\end{lemma}

\begin{proof}
Both query orientations represent the same underlying relations \(A\) and
\(B\), differing only in whether \(\delta\theta_{\mathrm{out}}\) is grouped
with the inputs or with the outputs. The stated inclusions and equalities
therefore follow directly from the definitions.
\end{proof}

%% file: sections/appendix_finite_constraint_design_problems.tex
\section{Finite-constraint \glspl{dp}}
\label{app:finite-constraint-design-problems}

We now specialize to the case in which the \gls{dp}
\(\DesignProblem\) is representable by finitely many smooth constraints:
\[
    \DesignProblem
    =
    \left\{
        (f,\theta,r)\in
        \Func\times\ThetaSpace\times\Res
        \;\middle|\;
        g(f,\theta,r)\leq 0
    \right\},
\]
with
\(
    g:\FuncAmbientSpace\times\ThetaAmbientSpace\times\ResAmbientSpace
        \to\mathbb R^{\ConstraintDim}
\)
continuously differentiable, where \(\ConstraintDim\) denotes the number of
inequality constraints and the inequality is understood componentwise.
We call these \glspl{dp} finite-constraint \glspl{dp} for short.

\subsection{Polyhedral representation of \glspl{tdp}}
\label{subsec:finite-constraint-linearized-models}

For an operating point
\((\OpPoint{f},\OpPoint{\theta},\OpPoint{r})\in\DesignProblem\),
define the active index set
\(
I(\OpPoint{f},\OpPoint{\theta},\OpPoint{r})
:=
\left\{
i\in\{1,\ldots,\ConstraintDim\}
\mid
g_i(\OpPoint{f},\OpPoint{\theta},\OpPoint{r})=0
\right\}.
\)

The associated linearized constraint system
\(
L_{\DesignProblem}
(\OpPoint{f},\OpPoint{\theta},\OpPoint{r})
\)
is defined as the set of all variations
\(
(\delta f,\delta\theta,\delta r)
\in
\Tan_{\Func\times\ThetaSpace\times\Res}
(\OpPoint{f},\OpPoint{\theta},\OpPoint{r})
\)
that satisfy
\(
\nabla_f g_i^\top\delta f
+\nabla_\theta g_i^\top\delta\theta
+\nabla_r g_i^\top\delta r
\leq 0
\)
for
\(
i\in I(\OpPoint{f},\OpPoint{\theta},\OpPoint{r})
\),
where \(\nabla g_i\) denotes the gradient of \(g_i\) evaluated at the operating point.

In general, we have
\(
\Tan_{\DesignProblem}
(\OpPoint{f},\OpPoint{\theta},\OpPoint{r})
\subseteq
L_{\DesignProblem}
(\OpPoint{f},\OpPoint{\theta},\OpPoint{r})
\).
If a suitable \gls{cq}~\cite{solodov2011constraintqualifications}
holds for  
\(g(f,\theta,r)\leq 0\) together with the box constraints at 
\((\OpPoint{f},\OpPoint{\theta},\OpPoint{r})\), 
then this inclusion becomes an equality: 
\begin{equation} 
\label{eq:tangent-linearized-feasible-set} 
    \Tan_{\DesignProblem} 
    (\OpPoint{f},\OpPoint{\theta},\OpPoint{r}) 
    = 
    L_{\DesignProblem} 
    (\OpPoint{f},\OpPoint{\theta},\OpPoint{r}). 
\end{equation}

When \eqref{eq:tangent-linearized-feasible-set} holds, the \gls{tdp} can therefore
be computed directly from the linearized constraints as
\(
\tdp_{\OpPoint{f},\OpPoint{\theta},\OpPoint{r}}
=
L_{\DesignProblem}
(\OpPoint{f},\OpPoint{\theta},\OpPoint{r})
+
K.
\)
Since both
\(L_{\DesignProblem}
(\OpPoint{f},\OpPoint{\theta},\OpPoint{r})\)
and \(K\) are polyhedral cones, the resulting \gls{tdp} is polyhedral.

\begin{remark}
\glspl{dp} over Euclidean spaces with polyhedral feasible 
sets have been studied by Cai et al.~\cite{cai2026scalable}, who 
refer to them as \glspl{ldp}. Their queries can be solved exactly 
by reformulating them as \glspl{molp}, which enables substantially
more scalable computation. Thus, for each fixed \(\delta\theta\), 
the corresponding fiber of the polyhedral \gls{tdp} above defines an 
\gls{ldp}. See Cai et al.~\cite{cai2026scalable} for these properties.
\end{remark}

\subsection{\Glsentrylongpl{cq}}
\label{subsec:constraint-qualifications}

Equality~\eqref{eq:tangent-linearized-feasible-set} is Abadie's
\gls{cq}. Several standard \glspl{cq} imply
this property; their relationships are surveyed by
Solodov~\cite{solodov2011constraintqualifications}. Rather than reproduce
that survey, we focus on sufficient conditions that simplify under the
order structure of \glspl{dp}.

A standard sufficient condition for Abadie's \gls{cq} is \gls{mfcq}.
For an operating point
\(
(\OpPoint{f},\OpPoint{\theta},\OpPoint{r})\in\DesignProblem
\),
\gls{mfcq} requires the existence of
\(
(\delta f,\delta\theta,\delta r)
\in
\operatorname{int}
\Tan_{\Func\times\ThetaSpace\times\Res}
(\OpPoint{f},\OpPoint{\theta},\OpPoint{r})
\)
such that
\(
\nabla_f g_i^\top\delta f
+\nabla_\theta g_i^\top\delta\theta
+\nabla_r g_i^\top\delta r<0
\)
for all
\(i\in I(\OpPoint{f},\OpPoint{\theta},\OpPoint{r})\).
For \glspl{dp}, if
\(
(\OpPoint{f},\OpPoint{\theta},\OpPoint{r})
\in\operatorname{int}(\Func\times\ThetaSpace\times\Res)
\),
upperness then makes the \gls{mfcq} condition
particularly simple.

\begin{proposition}[\Gls{mfcq} from \gls{dp} upperness]
\label{prop:mfcq-dp-upperness}
Let
\(
\DesignProblem
=
\left\{
(f,\theta,r)\in\Func\times\ThetaSpace\times\Res
\;\middle|\;
g(f,\theta,r)\leq0
\right\}
\)
be a finite-constraint \gls{dp}, and let
\(
(\OpPoint{f},\OpPoint{\theta},\OpPoint{r})
\in\operatorname{int}(\Func\times\ThetaSpace\times\Res)
\)
be an operating point.
Suppose that
\(
(\nabla_f g_i,\nabla_r g_i)\neq 0
\)
for every
\(i\in I(\OpPoint{f},\OpPoint{\theta},\OpPoint{r})\).
Then \gls{mfcq} holds at
\((\OpPoint{f},\OpPoint{\theta},\OpPoint{r})\).
\end{proposition}

\begin{proof}
By upperness, given
\(
(\OpPoint{f},\OpPoint{\theta},\OpPoint{r})
\in\operatorname{int}(\Func\times\ThetaSpace\times\Res)
\),
for every \(\delta f\geq0\) and \(\delta r\geq0\),
\(
(\OpPoint{f}-\PerturbationScale\delta f,
 \OpPoint{\theta},
 \OpPoint{r}+\PerturbationScale\delta r)
\in\DesignProblem
\)
for \(\PerturbationScale\geq0\) small enough.
For every
\(i\in I(\OpPoint{f},\OpPoint{\theta},\OpPoint{r})\),
differentiability therefore gives
\(\nabla_f g_i\geq0\) and
\(\nabla_r g_i\leq0\) componentwise.
For
\(
(\delta f,\delta\theta,\delta r)
=
(-\mathbf 1,\mathbf 0,\mathbf 1)
\),
every active inequality satisfies
\(
\nabla_f g_i^\top\delta f
+\nabla_\theta g_i^\top\delta\theta
+\nabla_r g_i^\top\delta r
<0
\),
where strictness follows from
\(
(\nabla_f g_i,\nabla_r g_i)\neq0
\).
Thus \gls{mfcq} holds.
\end{proof}

\begin{remark}[Boundary points]
At boundary points, some box constraints can oppose the monotone direction.
For example, for the choice
\(
(\delta f,\delta\theta,\delta r)
=
(-\mathbf 1,\mathbf 0,\mathbf 1)
\)
used in the proof, the direction \(\delta f_i=-1\) at a lower bound
\(f_i=\underline f_i\) points outside the box, so such constraints must be checked
separately. Active bounds on \(\theta\) must always be checked separately as well, since
\(\delta\theta=0\).
By contrast, this choice of \((\delta f,\delta\theta,\delta r)\) strictly
satisfies the linearized inequalities for active upper bounds on \(f\) and
lower bounds on \(r\). These constraints therefore require no additional
verification.
\end{remark}

We next show how \gls{mfcq} can be used to verify the regularity properties
of \cref{def:protodifferentiability,def:aubin-property,def:semidifferentiability}.

\begin{proposition}[Query regularity under \gls{mfcq}]
\label{prop:query-regularity-mfcq}
Let
\((\OpPoint{x},\OpPoint{\theta},\OpPoint{y})
\in\operatorname{int}(\QueryInputSpace\times\ThetaSpace\times\QueryOutputSpace)\)
be an operating point of a finite-constraint \gls{dp} \(\DesignProblem\).
If \gls{mfcq} holds at
\((\OpPoint{x},\OpPoint{\theta},\OpPoint{y})\),
then \(\GenericQuery\) is protodifferentiable at
\(((\OpPoint{x},\OpPoint{\theta}),\OpPoint{y})\).
If, moreover, \gls{mfcq} can be verified with
\(\delta x=0\) and \(\delta\theta=0\), then
\(\GenericQuery\) also has the Aubin property and is therefore
semidifferentiable.
\end{proposition}

\begin{proof}
Under \gls{mfcq},
Rockafellar and Wets~\cite{rockafellar1998variational}
show in Theorem~6.31 and Corollary~6.30 that every tangent direction to
\(\Graph(\GenericQuery)\)
is realizable by a feasible first-order path,
proving protodifferentiability.
If \gls{mfcq} can be verified with \((0,0,\delta y)\),
Rockafellar~\cite{rockafellar1989protodifferentiability}
also shows in Theorem~5.4 and Example~5.5 that the Aubin property holds;
Rockafellar and Wets~\cite{rockafellar1998variational} then show in
Proposition~9.50 that semidifferentiability follows.
\end{proof}

\subsection{Compositional verification of \glspl{cq}}
\label{subsec:cq-compositional}

\cref{def:composition-operations} shows how \glspl{dp} can be
composed. Here, we reverse the perspective: we use intersection to decompose
a single \gls{dp} into simpler components and exploit this representation to
verify \glspl{cq} compositionally.

\begin{proposition}[Compositional verification of Abadie's \gls{cq}]
\label{prop:compositional-cq}

Let
\(
\DesignProblem
\subseteq
\Func\times\ThetaSpace\times\Res
\)
be a finite-constraint \gls{dp}, and fix
\(
(\OpPoint{f},\OpPoint{\theta},\OpPoint{r})\in\DesignProblem
\).
Suppose that
\(
\DesignProblem=\bigcap_{i=1}^N\DesignProblem_i
\),
where each
\(
\DesignProblem_i
\subseteq
\Func\times\ThetaSpace\times\Res
\)
is itself a finite-constraint \gls{dp} on the same box, defined by a subset of
the inequality constraints defining \(\DesignProblem\), and every constraint
defining \(\DesignProblem\) appears in at least one
\(\DesignProblem_i\).

Suppose that the intersection rule of
\cref{thm:composition-rules-system-sensitivity} holds at this operating
point and that
\(
\Tan_{\DesignProblem_i}
(\OpPoint{f},\OpPoint{\theta},\OpPoint{r})
=
L_{\DesignProblem_i}
(\OpPoint{f},\OpPoint{\theta},\OpPoint{r})
\)
for every \(i\).
Then
\(
\Tan_{\DesignProblem}
(\OpPoint{f},\OpPoint{\theta},\OpPoint{r})
=
L_{\DesignProblem}
(\OpPoint{f},\OpPoint{\theta},\OpPoint{r})
\).

\end{proposition}

\begin{proof}
At the fixed operating point
\((\OpPoint{f},\OpPoint{\theta},\OpPoint{r})\),
\(
\Tan_{\DesignProblem}
=
\bigcap_{i=1}^N \Tan_{\DesignProblem_i}
=
\bigcap_{i=1}^N L_{\DesignProblem_i}
=
L_{\DesignProblem}.
\)
The equalities follow, respectively, from the intersection rule of
\cref{thm:composition-rules-system-sensitivity}, the assumed
\glspl{cq} for the \(\DesignProblem_i\), and the fact that the constraints of \(\DesignProblem\) are exactly those of
the \(\DesignProblem_i\).
\end{proof}

Thus, the \glspl{cq} may be verified separately for each
\(\DesignProblem_i\), or for each constraint when every constraint defines a
\gls{dp} on the common box.

\begin{example}
Suppose that \(\DesignProblem_i\) is defined by a single active inequality and
that
\(
(\OpPoint{f},\OpPoint{\theta},\OpPoint{r})
\in
\operatorname{int}(\Func\times\ThetaSpace\times\Res)
\),
so that
\(
g_i(\OpPoint{f},\OpPoint{\theta},\OpPoint{r})=0
\).
Quasinormality, which implies Abadie's
\gls{cq}~\cite{solodov2011constraintqualifications}, holds automatically if
\(\nabla g_i(\OpPoint{f},\OpPoint{\theta},\OpPoint{r})\neq0\).
If
\(\nabla g_i(\OpPoint{f},\OpPoint{\theta},\OpPoint{r})=0\),
it reduces to requiring
\(
g_i(f,\theta,r)\leq0
\)
in a neighborhood of the operating point, i.e., the operating point is a
local maximizer of \(g_i\).
\end{example}

%% file: sections/appendix_beyond_convexity.tex
\section{Beyond convexity}
\label{app:beyond-convexity}

This appendix extends the composition results of
\cref{thm:composition-rules-system-sensitivity}
beyond the convex setting.
We consider series and feedback composition separately.

\subsection{Series composition}
\label{subsec:beyond-convexity-series-composition}

For series composition, we use the graphical-derivative calculus of
Li, Meng, and Penot~\cite{li2009calculus}.
Our setting is less general than theirs, since we work in Euclidean spaces
and with bounded set-valued maps.

We recall the notion of directional semi-compactness,
defined by Li, Meng, and Penot~\cite{li2009calculus}
in Definition~2.
A set-valued map \(F:X\rightrightarrows Y\) is directionally semi-compact at
\(\OpPoint x\in\Closure{\operatorname{dom}F}\) if, along every first-order
approach \(\OpPoint x+\PerturbationScale_k\delta x_k\), with \(\PerturbationScale_k\downarrow0\),
\(\delta x_k\to \delta x\in\Tan_{\operatorname{dom}F}(\OpPoint x)\), and
\(\OpPoint x+\PerturbationScale_k\delta x_k\in\operatorname{dom}F\), one can select
\(y_k\in F(\OpPoint x+\PerturbationScale_k\delta x_k)\) so that a subsequence of \(y_k\) converges.
In our setting, \gls{dp} queries satisfy this condition automatically,
since their values lie in bounded Euclidean boxes.

\begin{proposition}[Reverse inclusion for series composition
{\cite[Propositions~15(b) and~19]{li2009calculus}}]
\label{prop:series-composition-reverse-inclusion-nonconvex}

Let
\(\DesignProblem_1\subseteq
\QueryInputSpace\times\ThetaSpace_1\times\IntermediateSpace\)
and
\(\DesignProblem_2\subseteq
\IntermediateSpace\times\ThetaSpace_2\times\QueryOutputSpace\)
be \glspl{dp} connected in series.
Let
\(\GenericQuery_1:
\QueryInputAmbientSpace\times\QueryParameterAmbientSpace_1
\rightrightarrows\QueryIntermediateAmbientSpace\)
and
\(\GenericQuery_2:
\QueryIntermediateAmbientSpace\times\QueryParameterAmbientSpace_2
\rightrightarrows\QueryOutputAmbientSpace\)
be the corresponding query maps.

Following Li, Meng, and Penot~\cite{li2009calculus}, define
\(\Xi:
\QueryInputAmbientSpace\times\QueryParameterAmbientSpace_1
\times\QueryParameterAmbientSpace_2
\times\QueryOutputAmbientSpace
\rightrightarrows\QueryIntermediateAmbientSpace\) by
\(
\Xi(x,\theta_1,\theta_2,y)
:=
\GenericQuery_1(x,\theta_1)
\cap
\bigl(\GenericQuery_2(\cdot,\theta_2)\bigr)^{-1}(y).
\)
By boundedness of \(\IntermediateSpace\), \(\Xi\) is directionally
semi-compact.

Fix
\(((\OpPoint x,\OpPoint{\theta}_1,\OpPoint{\theta}_2),\OpPoint y)
\in\Graph(\GenericQuery_2\circ
(\GenericQuery_1\times\operatorname{Id}_{\mathbb R^{n_{\theta_2}}}))\).
If, for every
\(\OpPoint q\in
\Xi(\OpPoint x,\OpPoint{\theta}_1,\OpPoint{\theta}_2,\OpPoint y)\),
\(
D\Xi
\left(
(\OpPoint x,\OpPoint{\theta}_1,\OpPoint{\theta}_2,\OpPoint y),
\OpPoint q
\right)
(0,0,0,0)
=
\{0\},
\)
then
\[
\begin{aligned}
&D(\GenericQuery_2\circ
(\GenericQuery_1\times\operatorname{Id}_{\mathbb R^{n_{\theta_2}}}))
\left(
    (\OpPoint x,\OpPoint{\theta}_1,\OpPoint{\theta}_2),
    \OpPoint y
\right)
\subseteq
\bigcup_{
    \OpPoint q\in
    \Xi(
        \OpPoint x,
        \OpPoint{\theta}_1,
        \OpPoint{\theta}_2,
        \OpPoint y
    )
}
\\
&\qquad
D\GenericQuery_2
\left(
    (\OpPoint q,\OpPoint{\theta}_2),
    \OpPoint y
\right)
\circ
\left(
    D\GenericQuery_1
    \left(
        (\OpPoint x,\OpPoint{\theta}_1),
        \OpPoint q
    \right)
    \times
    \operatorname{Id}_{\mathbb R^{n_{\theta_2}}}
\right).
\end{aligned}
\]
\end{proposition}

Li, Meng, and Penot~\cite{li2009calculus} show in Proposition~19 that a sufficient condition for
\(D\Xi((\OpPoint x,\OpPoint{\theta}_1,\OpPoint{\theta}_2,\OpPoint y),\OpPoint q)(0,0,0,0)=\{0\}\)
is
\(D\GenericQuery_1((\OpPoint x,\OpPoint{\theta}_1),\OpPoint q)(0,0)
\cap
\{\delta q \mid
0 \in D\GenericQuery_2((\OpPoint q,\OpPoint{\theta}_2),\OpPoint y)(\delta q,0)\}
=
\{0\}\).
In the finite-constraint setting of
\cref{app:finite-constraint-design-problems}, assume that
\eqref{eq:tangent-linearized-feasible-set} holds for both \glspl{dp}. The
condition above can equivalently be verified, for each relevant \(\OpPoint q\), by
checking that the following linear feasibility problem has only the solution
\(\delta q=0\):
\[
\begin{array}{ll}
\text{find} & \delta q
\\[0.2em]
\text{subject to}
&
\nabla_q g_{1,i}^{\top}
(\OpPoint x,\OpPoint{\theta}_1,\OpPoint q)\,\delta q\leq0,
\qquad
i\in I_1(\OpPoint x,\OpPoint{\theta}_1,\OpPoint q),
\\
&
\nabla_q g_{2,i}^{\top}
(\OpPoint q,\OpPoint{\theta}_2,\OpPoint y)\,\delta q\leq0,
\qquad
i\in I_2(\OpPoint q,\OpPoint{\theta}_2,\OpPoint y),
\\
&
\delta q\in
\Tan_{\IntermediateSpace}(\OpPoint q).
\end{array}
\]

\subsection{Feedback composition}
\label{subsec:beyond-convexity-feedback-composition}

For feedback composition, we follow the notation of
\cref{subsec:feedback-composition}, where
\(\Trace(\DesignProblem)=\pi(\DesignProblem\cap M)\).
We use the image-set calculus of Benko~\cite{benko2021innercalmness} to treat the
projection in the nonconvex setting.
We recall the notion of inner calmness,
defined by Benko~\cite{benko2021innercalmness}
in Definition~2.2.
A set-valued map \(F:X\rightrightarrows Y\) is inner calm at
\((\OpPoint x,\OpPoint y)\in\Graph F\) relative to
\(\operatorname{dom}F\) if there exists \(\kappa>0\) such that, for every
\(x_k\to\OpPoint x\) with \(x_k\in\operatorname{dom}F\), one can select
\(y_k\in F(x_k)\) satisfying
\(\|y_k-\OpPoint y\|\leq\kappa\|x_k-\OpPoint x\|\)
for all sufficiently large \(k\).

\begin{proposition}[Feedback composition beyond convexity]
\label{prop:feedback-composition-nonconvex}

Let
\(\DesignProblem
\subseteq
(\Func\times\IntermediateSpace)
\times\ThetaSpace
\times(\Res\times\IntermediateSpace)\)
be a finite-constraint \gls{dp}, and let
\((\OpPoint f,\OpPoint\theta,\OpPoint r)
\in\Trace(\DesignProblem)\).
Assume that, for every
\(\OpPoint q\) such that
\(((\OpPoint f,\OpPoint q),\OpPoint\theta,
(\OpPoint r,\OpPoint q))\in\DesignProblem\cap M\),
\glspl{cq} hold, ensuring that
\(\Tan_{\DesignProblem}=L_{\DesignProblem}\) and
\(\Tan_{\DesignProblem\cap M}=L_{\DesignProblem\cap M}\).
Let
\(
\Psi:
\FuncAmbientSpace\times\ThetaAmbientSpace\times\ResAmbientSpace
\rightrightarrows
(\FuncAmbientSpace\times\QueryIntermediateAmbientSpace)
\times\ThetaAmbientSpace
\times(\ResAmbientSpace\times\QueryIntermediateAmbientSpace)
\)
be defined by
\(
\Psi(f,\theta,r)
:=
\pi^{-1}(f,\theta,r)\cap(\DesignProblem\cap M).
\)
Assume that there exists one such \(\OpPoint q\) for which \(\Psi\) is inner
calm at
\(
\left(
(\OpPoint f,\OpPoint\theta,\OpPoint r),
((\OpPoint f,\OpPoint q),\OpPoint\theta,
(\OpPoint r,\OpPoint q))
\right)
\)
relative to \(\operatorname{dom}\Psi\).
Then
\(
(\Trace(\DesignProblem))^T_
{(\OpPoint f,\OpPoint\theta,\OpPoint r)}
=
\left(
\TTrace(\DesignProblem^T)
\right)_
{(\OpPoint f,\OpPoint\theta,\OpPoint r)}
\).

\end{proposition}

\begin{proof}
For every compatible \(\OpPoint q\), at
\(
((\OpPoint f,\OpPoint q),\OpPoint\theta,
(\OpPoint r,\OpPoint q))
\in\DesignProblem\cap M
\),
since \(M\) is a linear subspace, \(\Tan_M=L_M\). Therefore,
\(
\Tan_{\DesignProblem\cap M}
=
L_{\DesignProblem\cap M}
=
L_{\DesignProblem}\cap L_M
=
\Tan_{\DesignProblem}\cap\Tan_M
\).
For the compatible \(\OpPoint q\) at which \(\Psi\) is inner calm,
Benko~\cite{benko2021innercalmness} shows in Theorem~4.1,
applied to \(\Psi\) and the projection \(\pi\), that
\(
\Tan_{\Trace(\DesignProblem)}
=\pi\bigl(\Tan_{\DesignProblem\cap M}\bigr)
=\pi\bigl(\Tan_{\DesignProblem}\cap\Tan_M\bigr)
\), since inner calmness provides an \(O(\PerturbationScale_k)\) lift of every
tangent sequence in \(\operatorname{dom}\Psi=\Trace(\DesignProblem)\).
By \cref{lem:ordering-cone-composition-commutation}, this tangent-cone
equality is preserved by the ordering-cone completion defining the corresponding
\glspl{tdp}.
For every compatible \(\OpPoint q\),
\(
\pi\bigl(\Tan_{\DesignProblem}\cap\Tan_M\bigr)
\subseteq
\Tan_{\Trace(\DesignProblem)}
(\OpPoint f,\OpPoint\theta,\OpPoint r),
\)
while for the \(\OpPoint q\) satisfying the assumptions above,
the preceding projection equality applies.
Hence the union already equals the closed tangent cone, so the closure in
\cref{def:composition-tangential-design-problems},
equation~\eqref{eq:tdp-feedback-composition}, adds nothing and
\(
(\Trace(\DesignProblem))^T_
{(\OpPoint f,\OpPoint\theta,\OpPoint r)}
=
\left(
\TTrace(\DesignProblem^T)
\right)_
{(\OpPoint f,\OpPoint\theta,\OpPoint r)}
\).
\end{proof}